%% file: arxiv.tex
\documentclass{amsart}

\usepackage{amssymb}
\usepackage{bbm}
\usepackage[margin=1.2in]{geometry}
\usepackage[colorlinks=true, linkcolor=black, citecolor=black, urlcolor=blue]{hyperref}

\newenvironment{ack}{\medskip\noindent\textbf{Acknowledgments.}\enspace}{\medskip}
\newenvironment{funding}{\noindent\textbf{Funding.}\enspace}{\medskip}

\renewcommand{\paragraph}[1]{\medskip\noindent\textbf{#1.}\,}

\input{macros}

\setlist{leftmargin=*,labelindent=0.5em}

\renewcommand{\kbb}{\mathbbm{k}}

\begin{document}

\title{An Algebraic Introduction to Persistence}

\author{Ulrich Bauer}
\address{U.~Bauer: Department of Mathematics, TUM School of Computation, Information and Technology, Technical University of Munich (TUM), Munich, Germany; Munich Center for Machine Learning (MCML), Munich, Germany; and Munich Data Science Institute (MDSI), Technical University of Munich (TUM), Munich, Germany}
\email{ulrich.bauer@tum.de}

\author{Thomas Br\"ustle}
\address{T.~Br\"ustle: Department of Mathematics, Bishop's University, Sherbrooke, QC, Canada; and D\'epartement de math\'ematiques, Universit\'e de Sherbrooke, Sherbrooke, QC, Canada}
\email{tbruestl@ubishops.ca}

\author{Luis Scoccola\textsuperscript{*}}
\address{L.~Scoccola: Laboratoire d'Alg\`ebre, de Combinatoire et d'Informatique Math\'ematique, Universit\'e du Qu\'ebec \`a Montr\'eal, Montr\'eal, QC, Canada}
\email{luis.scoccola@gmail.com}


\thanks{\textsuperscript{*}Corresponding author.}

\thanks{Thomas Br\"ustle was partially supported by NSERC Discovery Grant RGPIN-2025-05047, as well as Bishop's University and Universit\'e de Sherbrooke.}

\begin{abstract}
\input{abstract}
\end{abstract}

\maketitle
\tableofcontents
\newpage

\input{content-main}

\input{content-appendix}

\bibliographystyle{amsalpha}
\bibliography{persistence-survey}

\end{document}

%% file: macros.tex
\usepackage{mathtools}
\usepackage{mathrsfs}
\makeatletter
\@ifpackageloaded{enumitem}{}{\usepackage{enumitem}}
\@ifpackageloaded{xcolor}{}{\usepackage{xcolor}}
\makeatother
\usepackage{tikz}
\usetikzlibrary{arrows.meta,decorations.pathreplacing,calc,positioning}
\usepackage{aliascnt}

\newcommand{\kbb}{\mathbb{k}}
\newcommand{\Pcal}{\mathcal{P}}

\newcommand{\Veck}{\mathrm{Vec}}
\newcommand{\veck}{\mathrm{vec}}

\newcommand{\RepP}{\Veck^{\,\Pcal}}
\newcommand{\repP}{\veck^{\,\Pcal}}

\newcommand{\repR}{\veck^{\,\R}}

\newcommand{\RepRn}{\Veck^{\,\R^n}}

\newcommand{\rep}{\mathrm{rep}}

\newcommand{\R}{\mathbb{R}}
\newcommand{\Z}{\mathbb{Z}}
\newcommand{\N}{\mathbb{N}}

\renewcommand{\epsilon}{\varepsilon}
\renewcommand{\phi}{\varphi}

\def\Top{\mathrm{Top}}

\makeatletter
\newcommand\DEFINEALPHABETLOOP[3]{%
  \ifx\relax#3\expandafter\@gobble\else\expandafter\@firstofone\fi
  {\expandafter\providecommand\expandafter*\csname#3#1\endcsname{#2{#3}}%
   \DEFINEALPHABETLOOP{#1}{#2}}%
}%
\newcommand\Definealphabet[2]{%
  \DEFINEALPHABETLOOP{#1}{#2}abcdefghijklmnopqrstuvwxyzABCDEFGHIJKLMNOPQRSTUVWXYZ\relax
}%
\makeatother
\Definealphabet{bb}{\mathbb}
\Definealphabet{sf}{\mathsf}
\Definealphabet{scr}{\mathscr}
\Definealphabet{cal}{\mathcal}

\renewcommand{\mod}{\mathrm{mod}}

\DeclareMathOperator{\Int}{Int}
\let\rep\undefined
\DeclareMathOperator{\rep}{rep}
\newcommand{\define}[1]{\textbf{#1}}
\let\coker\undefined
\DeclareMathOperator{\coker}{coker}
\DeclareMathOperator{\Span}{Span}
\DeclareMathOperator{\End}{End}
\DeclareMathOperator{\Hom}{Hom}
\DeclareMathOperator{\Ext}{Ext}
\DeclareMathOperator{\Tor}{Tor}
\newcommand{\Ab}{\mathrm{Ab}}
\newcommand{\rk}{\mathrm{rk}}
\newcommand{\rkit}{\emph{rk}}
\newcommand{\sprd}{\emph{Sprd}}
\newcommand{\im}{\mathrm{Im}}

\let\Lcal\undefined
\newcommand{\Lcal}{\mathcal{L}}

\newcommand{\pfd}{\ensuremath{\mathit{pfd}}}
\newcommand{\fp}{\ensuremath{\mathit{fp}}}
\newcommand{\PH}{\mathrm{PH}}
\newcommand{\EPH}{\mathrm{EPH}}
\newcommand{\RISH}{\mathrm{RISH}}
\newcommand{\spl}{\ensuremath{\mathit{spl}}}
\newcommand{\dimhom}{\mathrm{dimhom}}
\newcommand{\op}{\mathrm{op}}

\theoremstyle{definition}
\newtheorem{definition}{Definition}[section]

\newaliascnt{example}{definition}

\aliascntresetthe{example}

\newaliascnt{remark}{definition}
\newtheorem{remark}[remark]{Remark}
\aliascntresetthe{remark}

\newaliascnt{notation}{definition}
\newtheorem{notation}[notation]{Notation}
\aliascntresetthe{notation}

\theoremstyle{plain}
\newaliascnt{theorem}{definition}
\newtheorem{theorem}[theorem]{Theorem}
\aliascntresetthe{theorem}

\newaliascnt{lemma}{definition}
\newtheorem{lemma}[lemma]{Lemma}
\aliascntresetthe{lemma}

\newaliascnt{proposition}{definition}
\newtheorem{proposition}[proposition]{Proposition}
\aliascntresetthe{proposition}

\newaliascnt{corollary}{definition}
\newtheorem{corollary}[corollary]{Corollary}
\aliascntresetthe{corollary}

\usepackage{cleveref}

\crefname{theorem}{Theorem}{Theorems}
\crefname{lemma}{Lemma}{Lemmas}
\crefname{proposition}{Proposition}{Propositions}
\crefname{corollary}{Corollary}{Corollaries}
\crefname{definition}{Definition}{Definitions}
\crefname{example}{Example}{Examples}
\crefname{remark}{Remark}{Remarks}
\crefname{notation}{Notation}{Notations}
\crefname{section}{Section}{Sections}
\crefname{figure}{Figure}{Figures}
\crefname{equation}{Equation}{Equations}

\def\noteson{%
\gdef\luis##1{\noindent{\color{blue}[Luis: ##1]}}%
\gdef\uli##1{\noindent{\color{red}[Uli: ##1]}}%
\gdef\thomas##1{\noindent{\color{olive}[Thomas: ##1]}}%
\gdef\todo##1{\noindent{\color{violet}[to do: ##1]}}}

\noteson

\newsavebox{\captionlistbox}
\ifx\captionindent\undefined
  \newlength{\captionindent}
\fi

%% file: abstract.tex
We introduce persistence with an emphasis on its algebraic foundations, using the representation theory of posets.
Linear representations of posets arise in several areas of mathematics,
including the representation theory of quivers and finite dimensional algebras,
Morse theory and other areas of geometry,
as well as topological inference and topological data analysis---often via persistent homology.
In some of these contexts,
the category of poset representations of interest admits a metric structure given by the so-called interleaving distance.
Persistence studies
the algebraic properties of these poset representations
and their behavior under perturbations in the interleaving distance.
We survey fundamental results in the area and applications to pure and applied mathematics, as well as theoretical challenges and open questions.

%% file: content-main.tex
\section{Introduction}

The main goal of this article is to present the theories of one-parameter persistence and of multiparameter persistence.
Briefly, these theories study the representations of total orders such as $\Rbb$, and of products of total orders, respectively.
Such representations arise naturally when studying fundamental geometric objects such as
real valued functions on topological spaces,
and weighted graphs, such as Vietoris--Rips complexes of metric spaces.
More context and motivations are given throughout the article.

\paragraph{Structure of the article}
In \cref{section:fundamentals-persistence} we present the fundamentals of persistence from the algebraic perspective.
We recall the notion of linear representation of a poset
(\cref{section:poset-representations}),
we introduce \emph{persistent homology}
(\cref{section:persistent-homology}),
we introduce the \emph{interleaving distance} and the main theoretical results of one-parameter persistence
(\cref{section:fundamentals-one-parameter-persistence}),
and we give a bit of history (\cref{section:history}).

In \cref{section:motivations-applications}, we survey motivations and applications of persistence coming from two main sources:
data analysis (\cref{section:persistence-data-analysis})
and geometry and analysis (\cref{section:persistence-geometry-analysis}).

In the next three sections, which can be read independently of each other, we introduce three advanced topics.
First, we present \emph{level set persistent homology} (\cref{section:levelset-persistence}), which goes beyond one-parameter persistence, while still admitting a well-understood theory.
Second, we overview \emph{multiparameter persistence} (\cref{section:multiparameter-persistence}),
an active area of research with several open questions.
Finally, we overview the representation theory of posets as it relates to persistence, in particular in regards to the representation theory of infinite posets (\cref{section:challenges-infinite-continuous}).

We conclude with \cref{section:computational-aspects}, which 
overviews the time and space complexity of several constructions in persistence.

\paragraph{What is not covered}
Although persistence is a relatively new area, it is not possible to cover all of its aspects in detail here.
We can only scratch the surface regarding some of the topics we do cover, and there are several other topics that we do not cover at all, including the following.
\begin{itemize}
\item 
The statistical aspects of persistent homology \cite{fasy-et-al,bobrowski-kahle-skraba,turner-et-al,hiraoka-et-al,lim,bobrowski2024universality}.
\item
The theory of Reeb graphs and merge trees (a.k.a.~contour trees or dendrograms), which are essentially persistent sets, i.e., set-valued functors on posets \cite{morozov-et-al,desilva-munch-patel,cardona-et-al}.
\item
The connections between persistence and discrete Morse theory \cite{mischaikow-nanda,bauer-roll}.
\item
The connections between persistence and microlocal sheaf theory \cite{kashiwara-schapira,kashiwara-schapira-2,miller-stratification}.
\item
The persistence theory of filtered chain complexes \cite{chacholski-giunti-landi,memoli-zhou}.
\item
Persistent homotopy theory \cite{blumberg-lesnick,jardine2020persistent,lanari-scoccola,cardona,medina-zhou,mader-waas,oudot2026function}.
\end{itemize}

\paragraph{Other surveys}
There already exist several good introductions to different aspects of persistence.
For one-parameter persistence and its connections to computational geometry and data analysis, we refer to~\cite{edelsbrunner-harer,edelsbrunner-harer-2,oudot-book,chazal-structure-stability,dey-wang} for longer manuscripts,
and to~\cite{weinberger,carlsson,edelsbrunner2017persistent,ghrist,chazal2021introduction} for shorter notes.
For the connections between one-parameter persistence and geometry and analysis, see~\cite{polterovich-rosen-samvelyan-zhang,weinberger-2}.
Multiparameter persistence from the perspective of applied topology and topological inference is treated in~\cite{botnan-lesnick-survey,lesnick-course};
there also exists specialized introductions focusing on connections with
order theory~\cite{kim-memoli-survey},
representation theory~\cite{blanchette-brustle-hanson-survey},
and differential calculus~\cite{oudot2024differential}.

\begin{ack}
We thank Uzu Lim, Fernando Martin, and Théo Prosper for detailed feedback on an early draft.
We also thank Raphael Bennett-Tennenhaus, Håvard Bjerkevik, Jan Jendrysiak, Vidit Nanda, Steve Oudot, Baptiste Rognerud, and Matthias Söls for their helpful comments and suggestions.
We thank Won Seong for spotting a gap in an earlier version of the no-go stability result in \cref{section:multiparameter-bottleneck}.
\end{ack}

\begin{funding}
  Thomas Brüstle was partially supported by NSERC Discovery Grant RGPIN-2025-05047, as well as Bishop's University and Université de Sherbrooke.
\end{funding}

\section{Fundamentals of Persistence}
\label{section:fundamentals-persistence}

\subsection{Poset Representations}
\label{section:poset-representations}

A $\kbb$-linear \define{representation} of a poset $\Pcal$ is a functor of the form $M\colon \Pcal \to \Veck$, where $\Veck$ denotes the category of vector spaces over a field $\kbb$, and the poset $\Pcal$ is seen as a category with one object for each element of $\Pcal$, a single morphism from $x$ to $y$ when $x \leq y \in \Pcal$, and no other morphisms.
We write $\RepP$ for the category of representations of $\Pcal$, and $\repP$ for the full subcategory of \define{pointwise finite dimensional} representations, abbreviated \define{\pfd{}}.

In the persistence literature, representations of posets are usually known as \emph{persistence modules}.
A \emph{one-parameter persistence module} is a representation of a totally ordered set, often the poset of real numbers with the standard order $(\R,\leq)$, and a \emph{multiparameter persistence module} is a representation of a product\footnote{Recall that in a product poset the order relation is defined component-wise.} of total orders, often the poset $\R^n$, with $n \geq 2$.
Poset representations admit several interpretations that connect persistence to various areas of mathematics:

\begin{itemize}
\item
\emph{Representation Theory.}
When $\Pcal$ is a finite poset, the category $\RepP$ is equivalent to the category of modules over the incidence algebra of $\Pcal$, which is a finite dimensional algebra \cite{Simson1992,BotnanOppermannOudot2023,BlanchetteBrustleHanson2024,asashiba-et-al,botnan-et-al}.
\item
\emph{Graded Commutative Algebra.}
When $\Pcal = \Z^n$, the category $\RepP$ is equivalent to the category of $\Z^n$-graded modules over the $\Z^n$-graded polynomial ring $\kbb[x_1,\ldots,x_n]$ \cite{calrsson-zomorodian}.
Similarly, when $\Pcal = \R^n$, the category $\RepP$ is equivalent to the category of 
$\R^n$-graded modules over the graded real exponent polynomial ring on $n$ variables \cite{geist-miller,miller-siaga}.
\item
\emph{Sheaf Theory.}
The representations of any poset $\Pcal$ can be viewed as sheaves on $\Pcal$, when $\Pcal$ is endowed with the Alexandrov topology;
in the special case of $\Pcal = \R^n$, poset representations can also be mapped to usual sheaves over the $n$-dimensional Euclidean space \cite{curry,kashiwara-schapira,berkouk-petit}.
\end{itemize}

Poset representations have been of interest for a long time \cite{mitchell-2}.
Of particular importance is the work of Nazarova and Roiter~\cite{nazarova-roiter}, leading to the solution to the second Brauer--Thrall conjecture for finite-dimensional algebras \cite{nazarova-roiter-2},
and the work of Gabriel \cite{gabriel} on quivers of finite type, both of which rely on filtered representations of posets.
Although filtered representations can be mapped to representations in the sense of this paper, the representation theory of the two notions can differ significantly; see, e.g., \cite{GabrielRoiter1992,Simson1992}.

The main motivation for the development of persistence is persistent homology, a construction providing a rich source of poset representations.

\subsection{Persistent Homology}
\label{section:persistent-homology}
A \define{filtration} consists of a topological space $X$ together with a function $f : X \to \Pcal$ into a poset $\Pcal$, typically $\R^n$.
The sublevel sets of such a filtration assemble into a functor $\Scal(f) : \Pcal \to \Top$ whose value at $r \in \Pcal$ is the topological space $\{x \in X : f(x) \leq r \in \Pcal\} \subseteq X$, and whose action on $r \leq s \in \Pcal$ is just the inclusion of sublevel sets.
By post-composing the functor $\Scal(f)$ with the homology functor $H_d(- ; \kbb) : \Top \to \Veck$, one obtains the \define{sublevel set persistent homology} of $f$, denoted $\PH_d(f) \in \Veck^\Pcal$.
For convenience, we say that a filtration $f : X \to \Pcal$ is \define{\pfd{}} if its sublevel set persistent homology $\PH_d(f)$ is \pfd{} for every $d \in \N$.
See \cref{figure:function-barcode-pd} for an example of two filtrations and their persistent homology representations.

Persistent homology takes its roots in Morse theory \cite{morse,barannikov,bauer-medina-schmahl} and in data analysis
\cite{frosini-2,frosini1992measuring,frosini-3,edelsbrunner-letscher-zomorodian,carlsson2004persistence}.
In \cref{section:history,section:persistence-data-analysis,section:persistence-geometry-analysis} we elaborate on the history, motivations, and applications of persistent homology.

\subsection{One-Parameter Persistence}
\label{section:fundamentals-one-parameter-persistence}

Persistence cares about \emph{algebraic} and \emph{geometric} properties of poset representations.

\begin{figure}
    \centering
    \includegraphics[width=\linewidth]{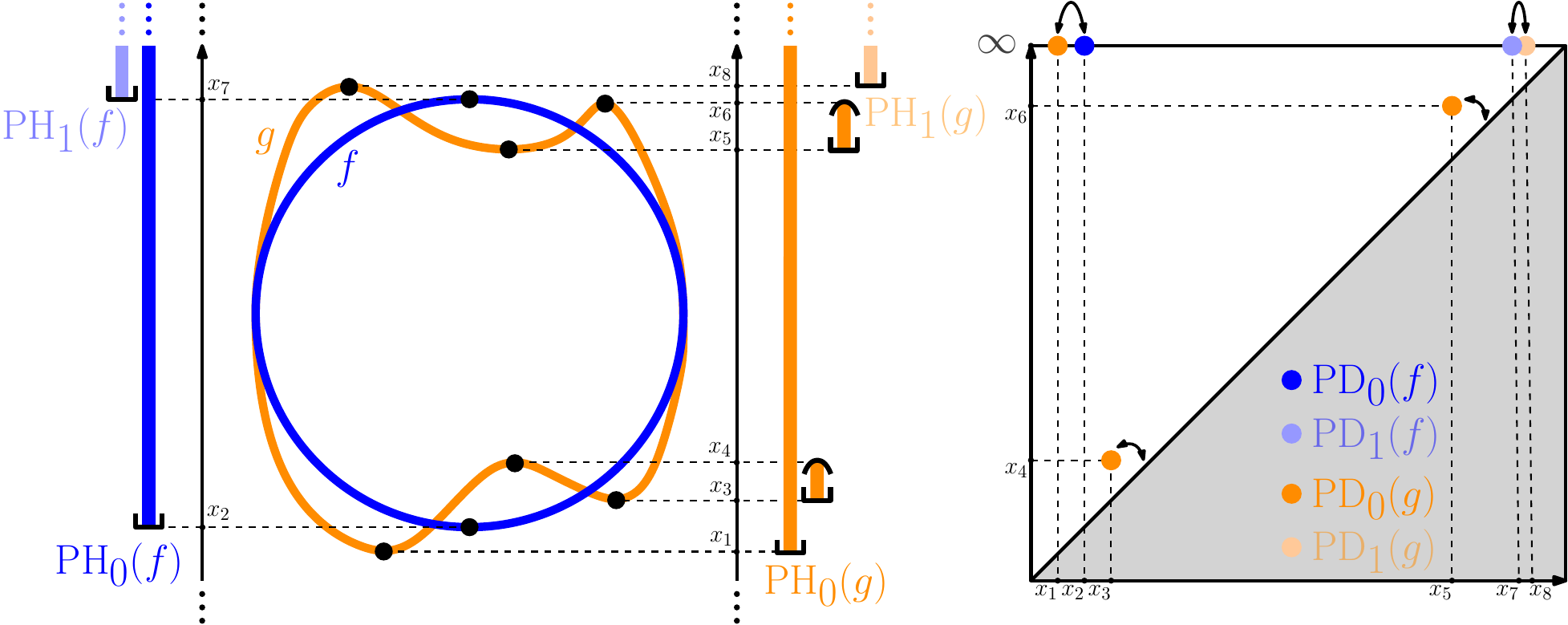}
    \caption{%
    \emph{Left.} Two real-valued functions on the circle and their persistent homologies in homological degrees zero and one, represented by their barcodes.
    The functions are represented as a projection, that is, the value of $f$ (resp.~$g$) at a point on the circle equals the projection of that point to the left (resp.~right) vertical real line.
    \emph{Right.} An alternative representation of each barcode as a \define{persistence diagram}, a multiset of points $(x,y)$ on the extended plane $[-\infty,\infty]^2$, with $x \leq y$, where each bar is represented as a point with coordinates given by the end-points of the bar.
    Depicted on the persistence diagrams is also a partial matching between the indecomposables, guaranteed to exist by the statement in \cref{equation:classical-stability-persistence}, since the two functions are point-wise close.
    When representing matchings on a persistence diagram, unmatched intervals are depicted as being matched to the diagonal; this is interpreted as the unmatched interval being matched to an empty interval, which can be thought of as an interval $[a,a)$ for some $a \in \Rbb$, thus corresponding to a point on the diagonal.
    }
    \label{figure:function-barcode-pd}
\end{figure}

On the algebraic side, the category $\Veck^\Pcal$ is an Abelian category, and often admits interesting subcategories where every object decomposes as a direct sum of indecomposable objects in an essentially unique way.
This holds, for example, for the subcategory of \pfd{} representations; see \cref{section:challenges-infinite-continuous} for more.
In the one-parameter case, the fundamental algebraic result is the following:

\medskip\noindent\textbf{Structure Theorem.}\;
\emph{Every representation $M \in \veck^\R$ decomposes uniquely as a direct sum of indecomposable representations, each of which is an \define{interval representation}, that is, the indicator representation $\kbb_I \in \repR$ for $I \subseteq \R$ an interval.}
\medskip

\noindent The canonical multiset of intervals determined by the indecomposable decomposition of a representation $M \in \repR$ is called the \define{barcode} of $M$.
See \cref{figure:function-barcode-pd} for an example.

\medskip

On the geometric side, the representations of $\R^n$ (and of other suitable posets) can be compared using \emph{interleavings}, which we now define.
The motivation for the notion of interleaving comes from the stability results in this section.

Let $\epsilon\ge0 $ in $\R$.
The \define{$\epsilon$-shift} of a representation $M\in\RepRn$, denoted $M[\epsilon]$, is the representation defined by
\[
  (M[\epsilon])_r \; \coloneqq \; M_{r+\epsilon\mathbf{1}}\; ,
  \qquad
  \varphi^{M[\epsilon]}_{r\le s}
 \;  \coloneqq \; \varphi^{M}_{r+\epsilon\mathbf{1}\,\le\,s+\epsilon\mathbf{1}}\; ,
\]
where $\varphi^M$ denotes the structure morphisms of $M$, and $\mathbf{1}=(1,\dots,1)\in\R^n$.
There is a canonical morphism of representations $\eta^M_\epsilon \colon M \longrightarrow M[\epsilon]$, whose component at $r\in\R^n$ is the structure map
$\eta^M_{\epsilon,r} := \varphi^M_{r \le r+\epsilon\mathbf{1}}$.

\medskip\noindent\textbf{Definition of Interleaving.}\;
An \define{$\epsilon$-interleaving} between representations $M,N \in \Veck^{\R^n}$ consists of morphisms
$f\colon M\to N[\epsilon]$ and $g\colon N\to M[\epsilon]$
such that
\[
  g[\epsilon]\circ f\; =\; \eta^M_{2\epsilon}\;\colon\; M\to M[2\epsilon]\;\;\;\;
  \text{and}\;\;\;\;
  f[\epsilon]\circ g\; =\; \eta^N_{2\epsilon}\;\colon\; N\to N[2\epsilon]\;.
\]
The \define{interleaving distance} is the (extended pseudo) distance on the objects of $\RepRn$ given by $d_I(M,N) := \inf\left(\{\epsilon\ge0 \mid
  \text{$M$ and $N$ are $\epsilon$-interleaved}\} \cup \{\infty\}\right)$.
\medskip

See \cref{figure:function-barcode-pd} for an example of representations at small interleaving distance.
Interleavings can be composed, which implies that the interleaving distance satisfies the triangle inequality.
Interleavings give a way of quantifying ``how isomorphic'' two representations are.
In particular, an isomorphism is the same thing as a $0$-interleaving.
An easy and instructive calculation shows the following:

\medskip\noindent\textbf{Soft Stability Theorem.}\;
\emph{If $f,g : X \to \R^n$ are any pair of filtrations on the same topological space $X$, then $d_I(\PH_d(f), \PH_d(g)) \leq \|f - g\|_\infty$ for every $d \in \N$.}
\medskip

The name of the above result comes from \cite{bubenik-silva-scott}; the corresponding ``hard'' stability theorem is algebraic/bottleneck stability, which is the forward implication in the next result,
and which, informally, says that the operation of decomposing a \pfd{} representation of $\R$ into indecomposables is stable, relating the algebra and geometry of one-parameter persistence modules:

\medskip\noindent\textbf{Isometry Theorem.}\;
\emph{Let $M,N \in \veck^\R$, and let $\epsilon \geq 0$.
The representations $M$ and $N$ are $\epsilon$-interleaved if and only if there exists a partial matching between the barcode of $M$ and the barcode of $N$, with the property that matched intervals differ by no more than $\epsilon$, and unmatched intervals are no longer than $2\epsilon$.}
\medskip

\noindent For a fully formal statement of the isometry theorem, see \cite[Theorem~6.4]{bauer-lesnick-induced} and \cref{section:history}.

A matching like the one in the Isometry Theorem is called an $\epsilon$-matching, and the infimum over $\epsilon$ for which there exists an $\epsilon$-matching is called the \define{bottleneck distance} between $M$ and $N$, denoted $d_B(M,N)$.
The isometry theorem is thus usually summarized in the following slightly weaker form:
\begin{center}
\emph{For all $M,N \in \veck^\Rbb$, we have $d_I(M,N) = d_B(M,N)$.}
\end{center}
The hard inequality is $d_B \leq d_I$, which is usually known as \emph{algebraic stability} or \emph{bottleneck stability}.
The inequality $d_I \leq d_B$ has an easy proof, and is sometimes known as \emph{converse stability}.

Combining soft stability with algebraic/bottleneck stability, one gets the following standard formulation of the stability of one-parameter persistent homology:
\begin{equation}
    \label{equation:classical-stability-persistence}
    \text{\emph{For all \pfd{} filtrations $f,g : X \to \R$}}\,,\;
    d_B\big(\,\PH_d(f)\,,\, \PH_d(g)\,\big)\,\leq\, \|f - g\|_\infty.
\end{equation}
See \cref{figure:function-barcode-pd} for an illustration of this stability result.

\subsection{A Bit of History}
\label{section:history}
The Structure Theorem may look familiar to readers acquainted with representation theory: indeed, the case of finitely presented representations is a consequence of the representation theory of quivers of type A.
Going from finitely presented to \pfd{} representations is not immediate~\cite{crawley-boevey}, but it was known to some experts before persistence was studied in its own right \cite[Section~3.6]{GabrielRoiter1992}.

The Isometry Theorem has an interesting history.
Its first general instantiation was the main result of \cite{cohen-steiner-edelsbrunner-harer}, which is the statement in \cref{equation:classical-stability-persistence}
in the special case of $X$ a finite simplicial complex and $f$ and $g$ defined simplex-wise.
The earlier paper \cite{d2003optimal} proves a particular case of this result, namely that of zero-dimensional persistent homology $d=0$.
Algebraic/bottleneck stability was first proved in \cite{chazal-socg}, and later generalized to \pfd{} representations and beyond in \cite{chazal-structure-stability,bauer-lesnick-induced}.
Converse stability was established in~\cite{lesnick-focm,bubenik-scott}.

The notion of interleaving for representations of $\R$ (one-parameter persistence) was introduced in \cite{chazal-socg}, and was later generalized to representations of $\R^n$ (multiparameter persistence) in \cite{lesnick-focm}.
The notion of interleaving admits several further generalizations, going well beyond the category of representations of $\R^n$; see for example \cite{bubenik-silva-scott,silva-munch-stefanou,scoccola2020locally,biran-cornea-zhang,mcfaddin-et-al}.

When the field $\kbb$ is a prime field, the interleaving distance satisfies a universal property that justifies its usage: it is the largest homology-invariant distance on representations of $\R^n$ that satisfies soft stability \cite[Section~5.2]{lesnick-focm}.
Universality of the interleaving distance for non-prime fields and in other settings is still an open question~\cite[Section~7]{lesnick-focm}.

More history and context are given throughout the article.

\section{Motivations and Applications}
\label{section:motivations-applications}

\subsection{Data Analysis}
\label{section:persistence-data-analysis}

Here, we mention some representative applications of persistence to data analysis.
For much more, see the DONUT database \cite{DONUT}.

Many of these applications are based on \emph{topological inference}, which we briefly recall below.
The main way in which persistence is used for topological inference is the following: one computes a barcode and interprets each bar as a topological feature; the length of a bar represents the topological prominence of the feature, so that long bars represent prominent features, while the short bars represent ``topological noise''.
This was first formulated in the generality of persistent homology in \cite{edelsbrunner-letscher-zomorodian},
but it is interesting to notice that, in the special case of zero-dimensional persistent homology, this same notion of topological prominence was introduced earlier in the context of topography, where it is used to measure how much a summit stands out from surrounding terrain \cite{fry1987defining}.

It is important to remark here that there exist many applications of persistence that are not directly based on topological inference or the above point of view.
We mention some at the end of this subsection.

\paragraph{Persistence for topological inference}
The problem here is that of estimating topological properties of spaces from incomplete information.
A classical example is \cite{niyogi-smale-weinberger}, which provides a statistically consistent algorithm for estimating the Betti numbers of a manifold from a finite sample.
This type of result usually relies on \emph{geometric complexes}, which are constructions taking as input geometric data (such as a finite sample of points in Euclidean space), and returning a filtered topological space.
Standard examples include the \emph{\v{C}ech complex} and the \emph{Vietoris--Rips complex}; see, e.g., \cite{boissonat-chazal-yvinec} for definitions.
These constructions are often stable in the \emph{Gromov--Hausdorff distance}, implying that samples of points that are close-by map to filtrations with persistent homologies that are close-by in the interleaving distance.
These stability results imply homological inference results stating that the homology of an unknown space can be estimated from a sufficiently good sample; e.g., \cite[Chapter~11]{boissonat-chazal-yvinec}, and \cite{MR2504289,MR2807544} for early examples.
Many more geometric complexes have been studied in the literature; e.g., \cite{chazal-desilva-oudot,natarajan2024morse}.
Several multiparameter geometric complexes for density-sensitive homological inference have been considered \cite{blumberg-lesnick-2}, and there are interesting open problems regarding computational complexity and polynomial-size approximations \cite{lesnick2024nerve,lesnick2024sparse}.

Several geometric problems in data analysis can be fruitfully interpreted through the lens of topological inference for different homological dimensions $d \in \Nbb$:

\begin{itemize}
\item \textbf{Clustering} $(d=0)$\textbf{.}
The notion of topological prominence and topological inference results are relevant to density-based clustering.
In \cite{chazal-JACM,rolle-scoccola}, persistence-based topological prominence is used as a way of estimating the number of clusters, with inference guarantees.
Persistence has also been used to reinterpret and reason about classical clustering algorithms such as single-linkage using one-parameter persistence \cite{carlsson-memoli}, and DBSCAN using multiparameter persistence \cite{mcinnes2017accelerated}.
\item \textbf{Circularity detection and parametrization} $(d=1)$\textbf{.}
Applications of persistence to the study of circularity in data include the analysis of periodicity in biological time series~\cite{perea-time-series} and of neural population activity in neuroscience~\cite{gardner-grid-cells,schneider-grid-cells}.
Of particular interest is the \emph{circular coordinates algorithm}~\cite{deSilva-Vejdemo-Johansson-circular} which constructs a map from a finite dataset to the circle $S^1$ that parametrizes a one-dimensional hole detected by persistent homology.
The construction relies on the representability theorem for cohomology combined with geometric inference techniques (such as Dirichlet energy estimation \cite{scoccola-toroidal}) to select canonical representatives.
\item \textbf{Void detection} $(d\geq 2)$\textbf{.}
Applications in astronomy \cite{wilding-et-al,calles-et-al} use persistence-based homological inference to detect 2-dimensional voids,
while higher-dimensional homological inference has been applied to physics \cite{sale-et-al,spitz2025topological}.
\end{itemize}

\paragraph{Applications not directly based on topological inference}
There are many applications in this category, and it is not our goal to survey all of them.
Let us only mention
quantifying connectivity and other patterns in neuroscience \cite{petri-et-al,curto2025topological},
quantifying small scale structure in material science, such as porosity and amorphous structure of glass \cite{nakamura-et-al,lee-et-al} and encoding local atomic environments in crystals for property prediction \cite{jiang2021topological},
and detecting spatial patterns in biology using one- \cite{rizvi-et-al,lawson2019persistent} and multiparameter persistence \cite{benjamin-et-al}.

\subsection{Geometry and Analysis}
\label{section:persistence-geometry-analysis}

Here, we survey motivations and applications in pure mathematics.

\paragraph{Morse theory}
One-parameter persistence offers a retrospective lens through which to view classical Morse theory.
For example, the critical values of a generic Morse function coincide with the end-points of the bars in its sublevel set barcode.
This point of view has been used to clarify and generalize constructions and results due to Morse \cite{morse}; see~\cite{bauer-medina-schmahl}.
In this context, it is interesting to observe that the barcode of a Morse function was already considered in Morse theory in \cite{barannikov}, using the language of normal forms for filtered chain complexes.
Multiparameter persistence connects to higher-dimensional Morse theory~\cite{whitney,smale,wan,gay-kirby} and Cerf theory~\cite{bubenik-catanzaro,cerri-ethier-frosini,budney-kaczynski,assif-baryshnikov}, although these connections remain largely unexplored.

\paragraph{Symplectic topology}
Filtered Floer complexes give rise to poset representations whose barcodes encode fundamental invariants.
Classical quantities including the spectral invariant~\cite{viterbo,schwarz,oh}, boundary depth~\cite{usher,usher-2}, and torsion exponents~\cite{fukaya-1,fukaya-2} translate into invariants of the barcode: infinite interval end-points~\cite{polterovich-shelukhin}, maximal finite interval length~\cite{polterovich-shelukhin}, and $k$th longest finite interval length~\cite{usher-zhang}, respectively.
This has been used to both reinterpret existing theory and prove new results~\cite{biran-cornea-zhang,shelukhin}.

\paragraph{Coarse counts in geometry}
The count of intervals in the barcode can be used to define approximations of classical geometric quantities that are stable under small perturbations of geometric data.
This approach has been used to prove coarse generalizations of fundamental results such as Courant's nodal theorem~\cite{buhovsky-et-al} and Bézout's intersection theorem~\cite{bujovsky-2}.

\paragraph{Other applications}
The reach of persistence-based methods extends further to fractal geometry~\cite{schweinhart-2}, metric geometry~\cite{lim-memoli-okutan,balitskiy-coskunuzer-memoli}, and quantitative homotopy theory~\cite{block-manin-weinberger}; an overview covering some of these topics is available in~\cite{polterovich-rosen-samvelyan-zhang}.

\section{Level Set Persistent Homology}
\label{section:levelset-persistence}

\begin{figure}
    \centering
    \includegraphics[width=\linewidth]{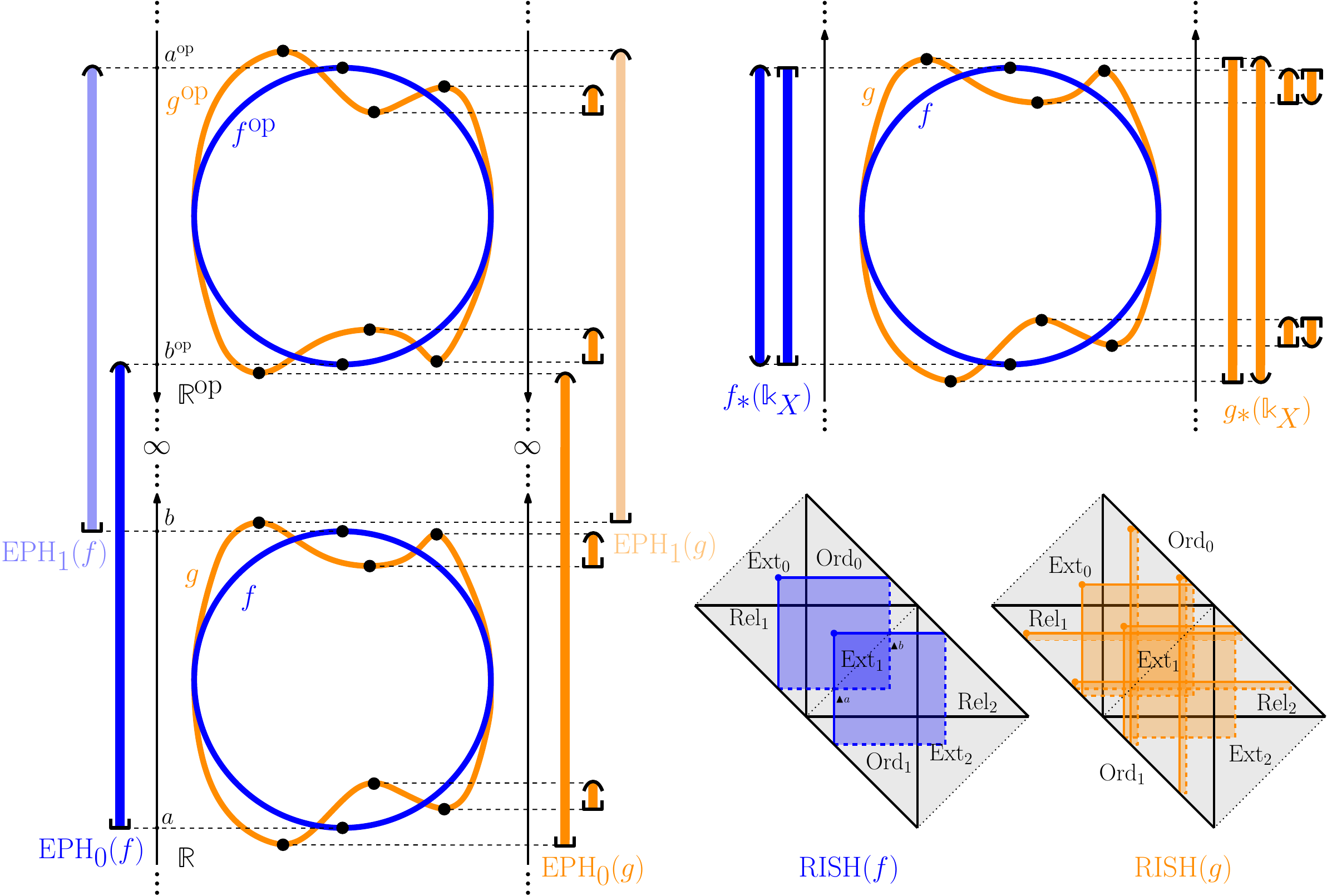}
    \caption{%
    The level set persistent homology of two functions $f,g : S^1 \to \Rbb$ encoded using extended persistence (left), derived sheaves (top right), and relative interlevel set persistence (bottom right).
    For relative interlevel set persistence we use the notation in \cite{bauer-botnan-fluhr}.
    Each construction determines the other, and satisfies a bottleneck stability result.
    Indeed, in each encoding there is a matching of bars of cost at most $\|f - g\|_\infty$. 
    The derived sheaf representation can be obtained from the relative interlevel set persistence by restriction to the dotted lines;
    similarly, the extended persistence can be obtained by restriction to a certain zig-zag; see \cite{bauer-botnan-fluhr} for details.
    }
    \label{figure:level-set-persistence}
\end{figure}

Before moving to multiparameter persistence, let us describe the theory of level set persistent homology.
This is more general than one-parameter persistent homology, but, as opposed to general multiparameter persistence, it still admits a well-understood theory of structure and bottleneck stability\footnote{From the perspective of multiparameter persistence, what we describe here is really one-parameter level set persistence. A corresponding theory of multiparameter level set persistence has not been fully developed yet, although it is tightly related to the sheaf-theoretic interpretation of persistence \cite{curry,kashiwara-schapira,berkouk-petit}.}.

Fix a real-valued function $f : X \to \Rbb$.
To capture more information about $f$, the goal is to extend the sublevel set persistent homology of $f$ from sublevel sets, i.e., sets of the form $f^{-1}((-\infty,a])$, to more general preimages, such as level sets $f^{-1}(a)$, interlevel sets $f^{-1}([a,b])$, suplevel sets $f^{-1}([a, \infty))$, and others.
Here we survey three, a posteriori, equivalent ways in which this has been done in the literature; see \cref{figure:level-set-persistence} for an illustration of these constructions.

\begin{itemize}
    \item
Define the poset
$\Lbb = \Rbb \sqcup \{\infty\} \sqcup \Rbb^\op$,
where an element is either $r \in \Rbb$, or $\infty$, or $s^\op$ with $s \in \Rbb$,
and the order is such that $r \leq r'$, $r \leq \infty$, $r \leq s^\op$, $\infty \leq s^\op$, and $s^\op \leq s'^\op$,
for all $r \leq r' \in \Rbb$ and $s' \leq s \in \Rbb$.
The \emph{extended persistent homology}
\cite{cohen-steiner-edelsbrunner-harer2009extending-persistence}
of $f$ in dimension $d \in \Nbb$
is the representation $\EPH_d(f) \in \Veck^\Lbb$ defined by
\begin{align*}
    \EPH_d(f)(r) \,&=\,
        H_d\left(\,f^{-1}\big((-\infty,r]\big)\,\right)
    \text{ for $r \in \Rbb$, and }\\
    \EPH_d(f)(s^\op) \,&=\,
        H_d\left(\,X\,;\, f^{-1}\big([s,\infty)\big)\,\right)
    \text{ for $s^\op \in \Rbb^\op$},
\end{align*}
where in the second line we are using relative homology \cite[Chapter~1,~Section~9]{munkres},
and where the structure morphisms are the natural ones.
When $\EPH_d(f)$ is \pfd{}, the Structure Theorem implies that it decomposes as a direct sum of interval representations of $\Lbb$ (since $\Lbb$ is isomorphic to $\Rbb$ as posets).
The multiset of intervals thus obtained can be represented as an \emph{extended persistence diagram}, which satisfies a bottleneck stability result \cite[Stability~Theorem]{cohen-steiner-edelsbrunner-harer2009extending-persistence}.
    \item
Let $\kbb_X \in \mathrm{D}(\mathrm{Sh}(X))$ be the constant sheaf on $X$, seen as an object of the derived category of vector space-valued sheaves on $X$.
Then, one can consider the derived pushforward $f_*(\kbb_X) \in \mathrm{D}(\mathrm{Sh}(\Rbb))$, which is now an object of the derived category of sheaves on $\Rbb$ (seen as a topological space) \cite{curry,kashiwara-schapira}.
Under certain tameness assumptions on $f$, the derived sheaf
$f_*(\kbb_X)$ is constructible and bounded, and thus, by the structure theorem \cite[Corollary~1.20]{kashiwara-schapira}, it decomposes as a direct sum of interval derived sheaves.
This decomposition can be represented as a \emph{graded barcode}, which satisfies a bottleneck stability result \cite[Theorem~5.10]{berkouk-ginot}.
    \item
Define the subposet $\Mbb \subseteq \Rbb \times \Rbb^\op$ as the convex hull of the lines $\ell_0, \ell_1 \subseteq \Rbb \times \Rbb^\op$ with slope $-1$ and intersecting the $x$-axis on $-\pi$ and $\pi$, respectively.
The \emph{relative interlevel set homology} of $f$ is a representation $\RISH(f) \in \Veck^\Mbb$ which encodes the
interlevel set and relative homology of $f$ in all homological degrees, as well as the structure morphisms induced by inclusions and by the connecting morphism
in the Mayer--Vietoris long exact sequence
\cite{bauer2022relativeinterlevelsetcohomology,fluhr-dissertation,bauer-botnan-fluhr}.
Under certain tameness assumptions on $f$, the structure theorem \cite[Theorem~C.5]{bauer-botnan-fluhr} implies that $\RISH(f)$ decomposes as a direct sum of indicator representations of maximal rectangles in $\Mbb$.
The multiset thus obtained serves as a barcode and comes with a notion of bottleneck distance.
This approach connects level set persistent homology to multiparameter persistence;
see \cite{botnan-lesnick} for an earlier approach.
\end{itemize}

Although, at a first glance, extended persistent homology may not seem to encode much more information than sublevel set persistent homology,
it encodes the homology of all level sets $f^{-1}(a)$ of Morse-like functions thanks to the Mayer--Vietoris long exact sequence
\cite{carlsson-de-silva-morozov}.
This is proven using
representations of a zig-zag poset
$\bullet \rightarrow \bullet \leftarrow \bullet \rightarrow \cdots \leftarrow \bullet$,
which is equivalently a quiver of type A.
In fact:

\medskip
\noindent\emph{For a Morse-like function $f : X \to \Rbb$,
the isomorphism types of $\,\EPH_\bullet(f) \in \Veck^\Lbb$, $f_*(\kbb_X) \in \mathrm{D}(\mathrm{Sh}(\Rbb))$, and $\RISH(f) \in \Veck^\Mbb$
can be recovered from one another.
Under this equivalence, the three bottleneck distances coincide, and are universal.}
\medskip

Considering the derived version of extended persistence, the statement can be strengthened to a natural isomorphism of functors that preserves interleavings;
see \cite{bauer-botnan-fluhr,fluhr-dissertation}, which builds on \cite{carlsson-de-silva-morozov,bendich2013interlevel}.
This equivalence gives three perspectives on the same phenomenon:
extended persistence is well suited for computations,
the derived level set persistence sheaf is a conceptually clean description,
and relative interlevel set homology is explicit from the structural point of view.
The derived equivalence between extended and level set persistence can be viewed as a continuously-indexed version of Happel's derived equivalence for the case of type $A$ quivers with different orientations~\cite{happel1988}; see \cite[Section~2]{bauer2022relativeinterlevelsetcohomology}.
This also connects to the representation theory of continuous quivers; see \cref{section:challenges-infinite-continuous}.

\section{Multiparameter Persistence}
\label{section:multiparameter-persistence}

\begin{figure}
    \centering
    \includegraphics[width=\linewidth]{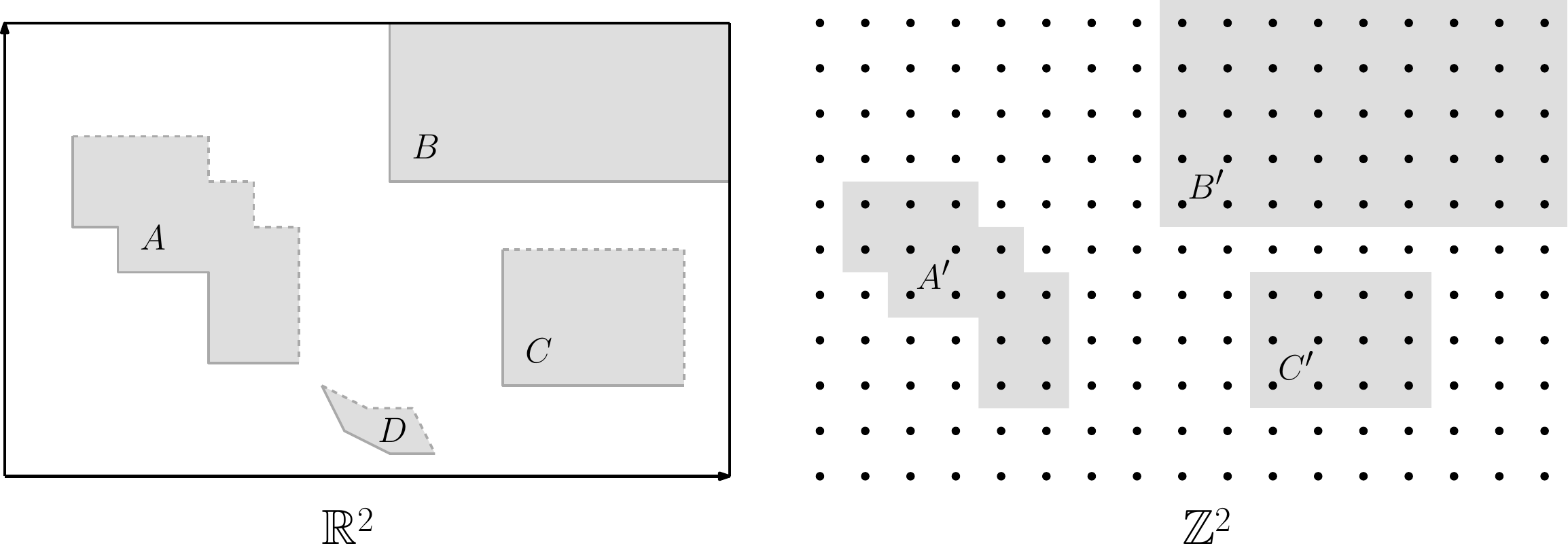}
    \caption{%
    \emph{Left.} Four subsets of the poset $\Rbb^2$ which are spreads.
    The spread representations $\kbb_A$, $\kbb_B$, and $\kbb_C$ are finitely presented, and in fact $\kbb_C$ is a rectangle representation, and $\kbb_B$ is an indecomposable projective since $B$ is the upset of a single point (its bottom left corner).
    The spread representation $\kbb_D$ is not finitely presentable since ``it is born along diagonals'' (see \cref{section:challenges-infinite-continuous}).
    \emph{Right.} The restriction of the spreads $A$, $B$, and $C$ to the subposet $\Zbb^2 \subseteq \Rbb^2$, which also happen to be spreads.
    }
    \label{figure:spread}
\end{figure}

We start by describing how the fundamental results of one-parameter persistence
discussed in \cref{section:fundamentals-one-parameter-persistence} break down in the multiparameter case.
We then go on to describe the main approaches for overcoming these challenges.
 
\subsection{The Main Challenges}
\label{section:main-challenges}

While it is true that every \pfd{} representation of any poset (and even any category) decomposes uniquely as a direct sum of indecomposable representations \cite{Azumaya1950,CrawleyBoevey1994,botnan-crawley},
the category of representations of non-linearly ordered sets such as $\R^n$ admits no simple characterization of the indecomposable objects, as it is expected from the classification of quivers according to their representation type:

\medskip
\noindent
\emph{%
The category of finite dimensional representations of a product of two or more total orders, each one with at least $4$ elements, is of wild type.}
\medskip

Informally, this means that classifying such representations is as hard as classifying the indecomposable representations of every possible finite dimensional algebra; see, e.g., \cite[Section~1]{bauer-scoccola} for a formal statement.
This issue cannot be fixed by restricting attention to, e.g., finitely presented representations (abbreviated \textbf{\fp{}}).
So, in multiparameter persistence, there is no useful analogue of the Structure Theorem.

The Isometry Theorem also fails dramatically \cite[Theorems~A~and~B]{bauer-scoccola}:

\medskip
\noindent
\emph{%
When $n \geq 2$, 
the space given by finitely presented representations of $\R^n$ endowed with the interleaving distance
has the property that the 
indecomposable representations form a dense (and even generic\footnote{Although this may sound similar to standard genericity statements in representation theory~\cite[Lemma~2.3]{reineke}, the relevant moduli space is constructed in a completely different way.}) subspace.}
\medskip

In particular, two multiparameter persistence modules that are very close in the interleaving distance can have radically different indecomposable decompositions, an observation already made in \cite{botnan-lesnick}; see also \cref{figure:stability-signed-invariants}.
So, although the bottleneck distance $d_B$ admits a simple generalization to multiparameter persistence as an infimum of the cost of all matchings between indecomposables (see, e.g., \cite[Section~1]{bauer-scoccola}), there is no stability result of the form $d_B \leq c \cdot d_I$ that applies to general multiparameter persistence modules.

There are two main ways in which multiparameter persistence modules are studied in the literature:
\begin{enumerate}
\item By restricting attention to suitable non-wild subcategories (often containing only spread representations, see below).
\item By considering additive invariants or suitable simplifications of the representations.
\end{enumerate}

\subsection{Spread Representations}

A \define{spread} of a poset $\Pcal$ is a subset $S \subseteq \Pcal$ that is \define{poset-connected} (i.e., non-empty and such that any two elements in $S$ are connected by a zigzag of comparable elements in $S$) and \define{poset-convex} (i.e., for any pair of elements in $S$, the closed segment between them is also in $S$).
A \define{spread representation} is any representation that is isomorphic to the indicator representation of a spread $S$, that is, the representation $\kbb_S : \Pcal \to \Veck$ that takes the value $\kbb$ on $S$ and zero elsewhere, with all morphisms that are not forced to be zero being the identity of $\kbb$.
For example, the interval representations in the Structure Theorem are spread representations;
see \cref{figure:spread} for multiparameter examples.

The notions of spread and of spread representation were introduced in persistence as generalizations---from one parameter to multiple parameters---of the concepts of interval and of interval representation.
Many references use the term \emph{interval} to refer to what we call spread;
we prefer to avoid this since this usage of the term interval does not agree with the notion of interval of a poset from order theory and combinatorics.

A relevant example of a subcategory of $\Veck^\Pcal$ whose objects are spread-decomposable is that of projective representations 
\cite[Chapter~9,~Corollary~7.3]{mitchell-2};
\cite[Corollary~9.2]{mitchell1978module};
\cite[Proposition~5]{hoeppner-lenzing}:

\medskip
\noindent\emph{%
Projective representations are free\footnote{Recall that, in a functor category $\mathrm{Set}^C$, a functor is \emph{free} if it is a coproduct of representables.
Note that the indicator representations of principal upsets are precisely the representables in $\Veck^\Pcal$, in the ($\kbb$-vector space) enriched sense.}.
That is,
every projective representation decomposes uniquely as a direct sum of indecomposables, each of which is isomorphic to the indicator representation $\kbb_{\{q \in \Pcal : q \geq p\}}$ of the principal upset of an element $p \in \Pcal$.
}
\medskip

Another interesting example is the category of rectangle-decomposable multiparameter persistence modules.
These are representations of products of linear orders that decompose as direct sums of indicator representations of rectangles; see \cite{botnan-et-al} for definitions and \cref{figure:spread} for an illustration.
Rectangle-decomposable representations are a generalization of block-decomposable representations, which show up naturally when studying level set persistence \cite{botnan-lesnick,bauer-botnan-fluhr} (see also \cref{section:levelset-persistence}).
There is interesting work characterizing rectangle- and block-decomposable representations using restrictions to certain subposets of finite representation type \cite{botnan-lebovici-oudot-2,cochoy,botnan-lebovici-oudot}.

\subsection{Additive Invariants}
\label{section:invariants}

Since the indecomposable representations of $\R^n$ and other posets cannot be effectively classified, much attention has been devoted to the study of additive invariants.
An \define{additive invariant} of an additive category $\Acal$ is a function $\alpha : \mathrm{Ob}(\Acal) \to G$ mapping the objects of $\Acal$ to an Abelian group $G$, which is isomorphism invariant (i.e., $\alpha(X) = \alpha(Y)$ whenever $X \cong Y \in \Acal$) and additive (i.e., $\alpha(X\oplus Y) = \alpha(X) + \alpha(Y)$ for all $X,Y \in \Acal$).

Equivalently, an additive invariant is a morphism 
of Abelian groups $\alpha : \Ksf^{\spl}(\Acal) \to G$,
with domain the \define{split Grothendieck group}, i.e., the quotient of the free Abelian group on $\mathrm{Ob}(\Acal)$ by relations $[A] - [B] + [C] = 0$ whenever there exists a split short exact sequence $0 \to A \to B \to C \to 0 \in \Acal$.
When every object of $\Acal$ decomposes as a finite direct sum of indecomposables,
these indecomposables form a basis of the group $\Ksf^{\spl}(\Acal)$,
and additivity just means that the invariant is determined by its value on indecomposables.

Given two such invariants $\alpha$ and $\beta$, one says that $\alpha$ is \define{finer} than $\beta$ if $\alpha(X) = \alpha(Y)$ implies $\beta(X) = \beta(Y)$ for all $X,Y \in \Acal$.
Such invariants are \define{equivalent} if each one is finer than the other one, and an
invariant $\alpha$ is \define{complete} if it is equivalent to isomorphism, i.e., $\alpha(X) = \alpha(Y)$ if and only if $X \cong Y$.

In \cref{figure:invariants}, we illustrate some of the main invariants that have been considered in the literature, and that we overview here.

\begin{figure}[p]
\savebox{\captionlistbox}{\parbox[t]{\dimexpr\linewidth-2\captionindent\relax}{\small%
    \hangindent=1.8em\hangafter=1\makebox[1.6em][l]{\emph{(a)}}A function $f : S^2 \to \Rbb^2$ from the two-sphere into the plane.
    \par\hangindent=1.8em\hangafter=1\makebox[1.6em][l]{\emph{(b)}}The one-dimensional sublevel set persistent homology $H_1(f) \in \Veck^{\Rbb^2}$, which happens to be indecomposable.
    \par\hangindent=1.8em\hangafter=1\makebox[1.6em][l]{\emph{(c)}}A representation $M \in \Veck^{\Zbb^2}$ obtained as the restriction of $H_1(f)$ to the poset $\Zbb^2$.
    We restrict to $\Zbb^2$ since the invariants depicted have not yet been developed for continuous representations such as $H_1(f) \in \Veck^{\Rbb^2}$.
    \par\hangindent=1.8em\hangafter=1\makebox[1.6em][l]{\emph{(d)}}The dimension vector of $M$.
    \par\hangindent=1.8em\hangafter=1\makebox[1.6em][l]{\emph{(e)}}The standard Betti tables of $M$, representing indecomposable projectives in degree zero (circles), one (crosses), and two (squares).
    \par\hangindent=1.8em\hangafter=1\makebox[1.6em][l]{\emph{(f)}}The end-curves of $M$ \cite{brustle2025counts}.
    \par\hangindent=1.8em\hangafter=1\makebox[1.6em][l]{\emph{(g)}}The signed barcode of $M$ \cite{BotnanOppermannOudot2023}.
    \par\hangindent=1.8em\hangafter=1\makebox[1.6em][l]{\emph{(h)}}The generalized persistence diagram of $M$ \cite{kim-memoli}.
    \par\hangindent=1.8em\hangafter=1\makebox[1.6em][l]{\emph{(i)}}The spread-Euler characteristic of $M$ \cite{escolar2024barcoding}.}}
    \centering
    \includegraphics[width=\linewidth, height=0.85\textheight, keepaspectratio]{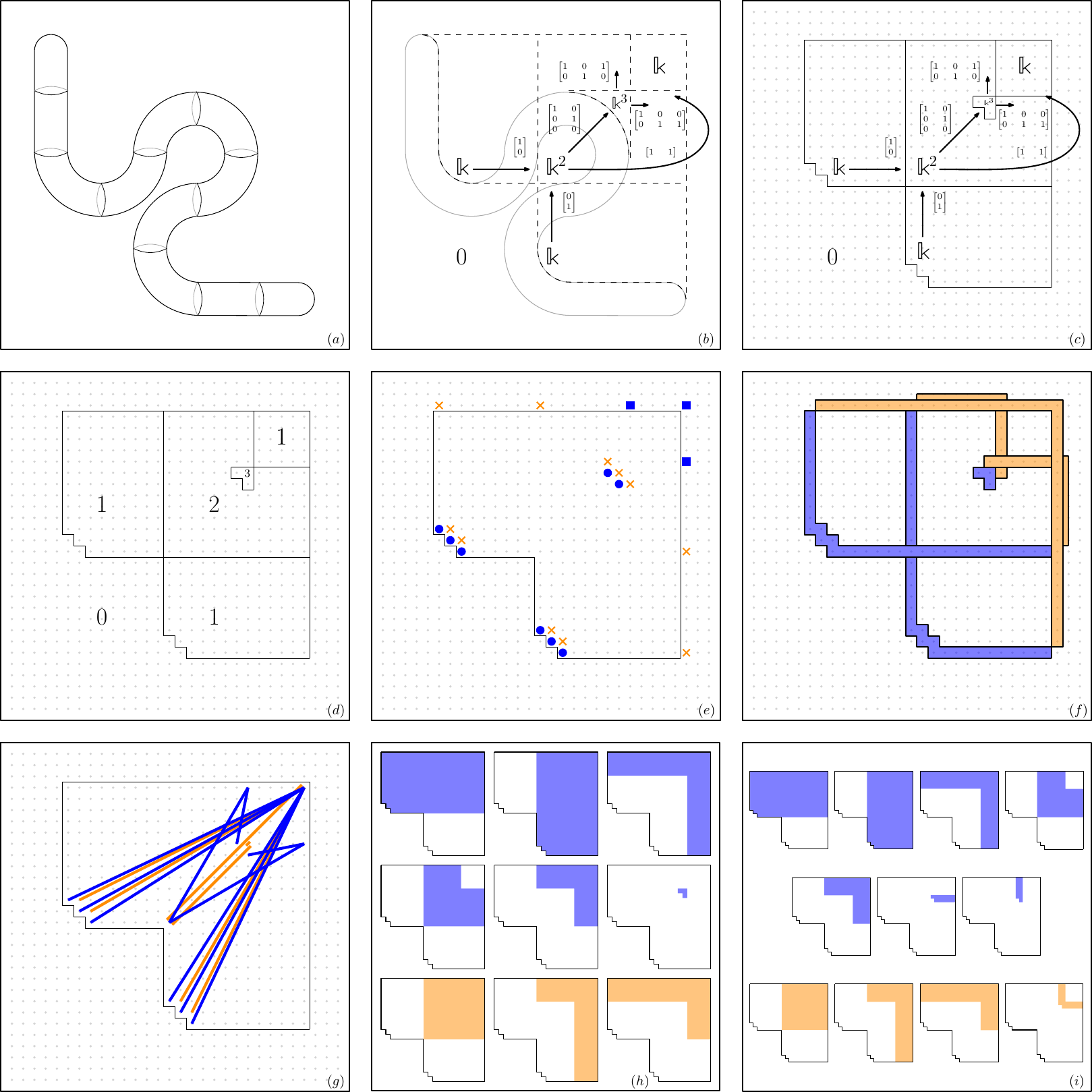}
    \caption[Six invariants of poset representations]{%
    Six invariants (panels $(d)$ to $(i)$) of a poset representation (panel $(c)$), obtained as the persistent homology of a function (panel $(a)$), restricted to $\Zbb^2$ (panel $(c)$).
    The invariants are ordered by their space complexity (see \cref{section:computational-aspects}).
    \smallskip
    \newline\usebox{\captionlistbox}}
    \label{figure:invariants}
\end{figure}

The additive invariants described in this section usually require some assumptions on the poset $\Pcal$ (e.g., finite) or on the representations under consideration (e.g., finitely presented).
For brevity and clarity, we let $\rep(\Pcal) \subseteq \Veck^\Pcal$ be a suitable category of poset representations, and for details about assumptions we refer the reader to the papers being cited.

\paragraph{Numerical invariants}
The \emph{pointwise dimension invariant} (also known as \emph{dimension vector} or \emph{Hilbert function}) is the additive invariant which simply records the pointwise dimension of each representation.
Formally, it is the function $\dim : \rep(\Pcal) \to \Zbb^\Pcal$ mapping a representation $M$ to the function $\dim(M)(p) = \dim_\kbb M(p)$.

In order to record information about the persistence of the elements of a representation and not just pointwise information, the seminal paper \cite{calrsson-zomorodian} introduces the \emph{rank invariant}, which consists of the rank of all the structure morphisms of each representation.
Formally, the rank invariant is the function $\rk : \rep(\Pcal) \to \Zbb^{\{(a,b) : a \leq b \in \Pcal\}}$ mapping a representation $M$ to the function $\rk(M)_{(a,b)} = \mathrm{rank}(M(a) \to M(b))$.
They prove that the rank invariant is complete in one-parameter persistence, and observe that this is not true in multiparameter persistence.

The rank invariant can be refined via the so-called \emph{generalized rank invariant}
$\rep(\Pcal) \to \Zbb^{\mathrm{spreads}(\Pcal)}$, which records the rank of the limit-to-colimit map of the restriction of the representation to each spread of the poset \cite{kim-memoli}\footnote{The original definition differs slightly from the most commonly used definition \cite{clause2022generalized}.}.
It is known that the generalized rank invariant of a representation $M$ on a spread $S$ coincides with the multiplicity of the indicator representation $\kbb_S$ in the indecomposable decomposition of the restriction of $M$ to $S$
\cite{chambers,brustle2025generalized,dey2026limit}.
This result and a notion essentially equivalent to the generalized rank invariant were considered earlier in the context of quiver representations in \cite{kinser,kinser-2}.

Another way to obtain information of a representation $M$ other than the pointwise dimension is to fix a set of representations $\Ycal$
and to consider the \emph{$\Ycal$-dimhom invariant}
\cite{BlanchetteBrustleHanson2024},
defined by taking the dimension of the hom-space from each $Y \in \Ycal$ to $M$.
Formally, one defines
$\dimhom^\Ycal : \rep(\Pcal) \to \Zbb^\Ycal$,
by $\dimhom^\Ycal(M)(Y) \coloneqq \dim_\kbb \Hom(Y, M)$.
The pointwise dimension invariant
is recovered by taking $\Ycal = \{\kbb_{\{q \in \Pcal : p \leq q\}}\}_{p \in \Pcal}$ to be the set of indecomposable projectives induced by principal upsets.

\paragraph{Invariants from homological algebra}
The homological algebra of poset representations has a long history \cite{mitchell-2,mitchell-3,baclawski,changchang,gerstenhaber-schack,igusa-zacharia}, and can be used to produce additive invariants as well.
Let $i \in \Nbb$.
The $i$th \emph{Betti table} is the additive invariant
$\beta_i : \rep(\Pcal) \to \Zbb^\Pcal$
defined by 
\[
    \beta_i(M)(p) \;\coloneqq\; \dim_\kbb\; \Ext^i_{\rep(\Pcal)}\big(\,M\,,\, \kbb_p\,\big)\;,\;\text{for $p \in \Pcal$}\,,
\]
where $\kbb_p \in \Veck^\Pcal$ is the simple representation with support $p$.
Equivalently (when $M$ admits a finite projective resolution),
the number $\beta_i(M)(p) \in \Zbb$ is the multiplicity of the indecomposable projective $\kbb_{\{q \in \Pcal : p \leq q\}}$ in the $i$th homological degree of any minimal projective resolution of $M$; see, e.g., \cite[Lemma~5.20]{botnan-et-al}.
See \cref{figure:invariants}~\emph{(e)} for an example.

The \emph{Euler characteristic} invariant, denoted $\chi$, is defined as the alternating sum of Betti tables $\sum_{i \in \Nbb} (-1)^i \beta_i : \rep(\Pcal) \to \Zbb^\Pcal$, so that
\[
    \chi(M) \; = \; \sum_{i \text{ even}} \beta_i(M) - \sum_{i \text{ odd}} \beta_i(M) \; = \; \beta_+(M) - \beta_-(M)\,.
\]
The sums of even and odd Betti tables $\beta_+ : \rep(\Pcal) \to \Zbb^\Pcal$ and $\beta_-: \rep(\Pcal) \to \Zbb^\Pcal$ will show up again when considering the bottleneck stability of these constructions in \cref{section:multiparameter-bottleneck}.

Betti tables are standard in
commutative algebra \cite{peeva,miller-strumfels}
and
persistence \cite{knudson,calrsson-zomorodian,lesnick-wright,oudot-scoccola}.
The following construction generalizes the above description of Betti tables from poset representations to representations of any finite dimensional algebra $\Lambda$, and may be familiar to readers acquainted with the representation theory of finite dimensional algebras \cite{auslander-reiten-smalo}:
Given a finite dimensional $\Lambda$-module $M$, one can define, for each $i \in \Nbb$,
\begin{equation}
    \label{definition:graded-homology}
    H^i(M)\, \coloneqq\, \Ext_\Lambda^i\big(\,M\,,\, \Lambda / \mathrm{rad}\, \Lambda\,\big) \,\in\, \mod_{\Lambda / \mathrm{rad}\,\Lambda},
\end{equation}
where $\mathrm{rad}\, \Lambda$ denotes the Jacobson radical of $\Lambda$.
The algebra $\Lambda / \mathrm{rad}\,\Lambda$ is semisimple,
so the module $H^i(M)$ is determined, to isomorphism, by the dimensions $\dim_\kbb \Ext_\Lambda^i(M, S)$ where $S$ ranges over the simple $\Lambda$-modules.
These dimensions are a natural extension of the notion of Betti table to this setting.
The construction in Definition~\ref{definition:graded-homology} is also standard in the context of graded algebras and Koszul duality \cite[Remark~2~after~Theorem~1.2.6]{beilinson-et-al}.

\paragraph{Invariants from relative homological algebra}
Many of the previous invariants can be generalized via homological algebra relative to an exact structure, in the sense of \cite{draxler-et-al,auslander-solberg}\footnote{Relative homological algebra should not be confused with the notion of relative homology from algebraic topology, used in \cref{section:levelset-persistence}.}.
Briefly, one can change the standard notion of exactness from homological algebra by choosing a collection of sequences $\Ecal$ declared to be the short exact sequences.
From this, one obtains new notions of $\Ecal$-projectivity, $\Ecal$-resolution,
and $\Ecal$-Grothendieck group.
A well-behaved exact structure $\Ecal$ provides a rich family of additive invariants:
\begin{itemize}
    \item \emph{$\Ecal$-Betti tables}: for an object $M$, these are the sequence of functions $\beta_i^\Ecal(M)$ ($i \in \Nbb$) recording the multiplicities of the indecomposable $\Ecal$-projectives appearing in the minimal $\Ecal$-projective resolution of $M$, assuming one exists\footnote{%
    Alternatively,
    in the case of modules over a finite dimensional algebra $\Lambda$, and $\Ecal$ obtained as the short sequences which are exact after mapping from each module in a finite collection $\Ycal$, 
    one can define relative Betti tables as the dimension of Definition~\ref{definition:graded-homology}, after replacing $\Lambda$ with the endomorphism algebra of $\bigoplus_{Y \in \Ycal} Y$, and $M$ by $\hom(\bigoplus_{Y \in \Ycal} Y, M)$.
    This process of reducing relative homological algebra to standard homological algebra is known as \emph{projectivization} \cite[Chapter~II]{auslander-reiten-smalo}.
    }.
    \item \emph{$\Ecal$-Euler characteristic}: the alternating sum $\chi^\Ecal = \sum (-1)^i \beta^\Ecal_i $ of $\Ecal$-Betti tables.
    \item \emph{$\Ecal$-Grothendieck group invariant}: the quotient map into the $\Ecal$-Grothendieck group.
\end{itemize}
The exact structures of interest in persistence include the following:
\begin{itemize}
    \item The \emph{standard exact structure}, where $\Ecal$ consists of the usual short exact sequences.
    \item The \emph{rank exact structure} \cite{BotnanOppermannOudot2023}, denoted $\rkit$, where the $\Ecal$ consists of those short exact sequences $0 \to A \to B \to C \to 0$ that are additive with respect to the rank invariant, i.e., such that $\rk(B) = \rk(A) + \rk(C)$.
    \item The \emph{spread exact structure} \cite{AsashibaEscolarNakashimaYoshiwaki2023}, denoted $\sprd$, where $\Ecal$ consists of those short exact sequences which remain exact after mapping from each spread representation.
\end{itemize}

\paragraph{Results on (relative) homological algebra of poset representations}
Let us start with some classical examples:
In the case of the standard exact structure, Hilbert's syzygy theorem \cite{hilbert} for multigraded polynomial rings can be reinterpreted in the language of poset representations as follows:
\[
    \mathrm{gl.dim}\left(\Veck_\fp^{\Zbb^n}\right) = 
    \mathrm{gl.dim}\left(\Veck_\fp^{\Rbb^n}\right) = n.
\]
Interestingly, the global dimension of the full category of representations of $\Rbb^n$ is shown to be $n+1$ in \cite{geist-miller}.
In \cite{igusa-zacharia}, it is shown that the Betti tables of simple representations in $\veck^\Pcal$ can be computed as the simplicial cohomology of the order complex of certain posets derived from $\Pcal$; in particular it is shown that there exist posets for which the global dimension of $\veck^\Pcal$ depends on the characteristic of the field~$\kbb$.
Examples coming from the persistence literature include the following.
The global dimension of the rank exact structure for finitely presented multiparameter persistence modules is computed in \cite{botnan-et-al}:
\[
 \mathrm{gl.dim}^{\rkit}\left(\Veck_\fp^{\Zbb^n}\right) = \mathrm{gl.dim}^{\rkit}\left(\Veck^{\R^n}_{\fp{}}\right) = 2n - 2\,.
\]
In \cite{aoki-et-al} it is proven that the global dimension of the spread exact structure is monotonic with respect to inclusions of posets (a property that is not satisfied by the standard global dimension):
\begin{center}
\emph{%
For every (full) subposet $\Qcal \subseteq \Pcal$ of a finite poset $\Pcal$, 
    $\displaystyle \mathrm{gl.dim}^\sprd(\Qcal) \leq \mathrm{gl.dim}^\sprd(\Pcal)$
}
\end{center}
and \cite{blanchette2025stabilization} gives a positive answer to a conjecture from \cite{AsashibaEscolarNakashimaYoshiwaki2023}, by proving the existence of a uniform bound on the global dimension of the spread exact structure on certain families of posets:
\begin{center}
\emph{%
For every finite poset $\Pcal$, we have
$\displaystyle \sup_{\substack{\Lcal \text{ finite}\\\text{total order}}} \mathrm{gl.dim}^\sprd\big(\Lcal \times \Pcal\big) < \infty$.
}
\end{center}
The result is based on \cite{blanchette-brustle-hanson-survey}, which characterizes the irreducible morphisms of the spread exact structure, and gives methods for proving existence of certain exact structures on the category of representations of infinite posets; see also \cite{aoki-tada}.

\paragraph{Bases for additive invariants}
Since two additive invariants can be equivalent without being equal, there may be more or less convenient representatives for a given equivalence class of invariants.
For example, for \fp{} representations of $\Rbb$, the rank invariant is equivalent to the barcode, but while the barcode consists of a finite multiset of intervals, the rank invariant records an uncountable number of ranks.
This motivates the search of convenient representations of invariants, and these can be obtained through bases.

A \define{basis} for an additive invariant $\alpha : \Acal \to G$ consists of a family of objects $\Bcal = \{M_i \in \Acal\}_{i \in I}$, such that, for every $M \in \Acal$, there exists a unique family of integers $\{c_i^M \in \Zbb\}_{i \in I}$, with the property that $c_i^M \neq 0$ for finitely many $i \in I$ and $\alpha(M) = \sum_{i \in I} c_i^M \cdot \alpha(M_i)$.

When the additive invariant $\alpha$ is seen as a group morphism $\alpha : \Ksf^{\spl}(\Acal) \to G$, such a basis induces an isomorphism of Abelian groups $\im(\alpha) \cong \Zbb^I$, mapping $\alpha(M)$ to the coefficients $\{c^M_i\}_{i \in I}$.
The additive invariant $\alpha$ expressed in the basis $\Bcal$ gives rise to a new additive invariant $\alpha^\Bcal(M) \coloneqq \sum_{i \in I} c^M_i \cdot \delta_i \in \Zbb^I$, which is equivalent to $\alpha$ but provides an alternative, sometimes more convenient, representation.

\paragraph{Bases from M\"obius inversion}
In one-parameter persistence, the set of (isomorphism classes of) spread representations forms a basis for the rank invariant, and the rank invariant represented in this basis is equal to the barcode, and thus complete \cite{calrsson-zomorodian}.
Interestingly, the change of basis is given by M\"obius inversion, in the following sense.
Let $\Pcal$ be a finite total order, and let $\Qcal = \{(a,b) : a\leq b \in \Pcal\}$ be the set of intervals of $\Pcal$, ordered by inclusion.
For every $a \leq b \in \Pcal$, the multiplicity of the interval module $\kbb_{[a,b]}$ in the indecomposable decomposition of $M$ is equal to $\widehat{\rk(M)}(a,b)$, where $\rk(M) : \Qcal \to \Zbb$ is the rank invariant of $M$, and $\widehat{-}$ denotes M\"obius inversion \cite{rota}.
This was observed in \cite{patel-mobius}.

An analogous observation is used in \cite{kim-memoli} to show that the set of spreads forms a basis for the generalized rank invariant over arbitrary finite posets.
The alternative representation of the generalized rank thus obtained is known as the \emph{generalized persistence diagram} (\cref{figure:invariants} \emph{(h)}).
Similarly, over an arbitrary poset, the standard rank invariant admits a basis given by the indicator representations of segments \cite[Corollary~2.14]{BotnanOppermannOudot2023}, and the alternative representation of the rank invariant thus obtained is known as the \emph{signed barcode} (\cref{figure:invariants} \emph{(g)}).

\paragraph{Bases from relative homological algebra}
M\"obius inversion has the following algebraic interpretation \cite[Lemma~7.3.5]{ladkani2008homological}:

\medskip
\noindent\emph{%
The zeta function of a finite poset $\Pcal$ equals the Cartan matrix of the incidence algebra of $\Pcal$.
Hence, the M\"obius function of $\Pcal$ is the inverse of the Cartan matrix of the incidence algebra of $\Pcal$.
}
\medskip

With this in mind, one can generalize the above story as follows.
Let $\Ycal$ be a finite set of pairwise non-isomorphic representations of a finite dimensional algebra $\Lambda$,
and denote also by $\Ycal$ the exact structure given by the sequences $0 \to A \to B \to C \to 0$ of representations of $\Lambda$ that are exact after taking $\hom$ from every object of $\Ycal$.
In this case, the set $\Ycal$ provides a basis for the $\Ycal$-dimhom invariant as well as for the relative Grothendieck group invariant,
and the (backwards) change of basis is given by the $\Ycal$-Cartan matrix $\Ccal^\Ycal$ (which records the dimension of the hom-space between elements of $\Ycal$):
\[
    \dimhom^\Ycal(M) = \Ccal^\Ycal \; \cdot \; \chi^\Ycal(M).
\]
See \cite{BlanchetteBrustleHanson2024} and \cite[Proposition~B.4]{brustle2025counts}.
The classical case is recovered by taking $\Lambda$ the incidence algebra of a finite poset $\Pcal$ and $\Ycal = \{\kbb_{\{q \in \Pcal : p \leq q\}}\}_{p \in \Pcal}$ the indecomposable projectives induced by principal upsets.

\paragraph{Relationships between invariants}
Besides easy observations (such as the fact that both the rank invariant and the standard Betti tables determine the dimension vector),
there are few results about the relationships between the different invariants in the literature.
Interesting results include: the fact that, in two-parameter persistence,
both the generalized persistence diagram
\cite{kim-moore} 
and the end-curves
\cite[Theorem~D]{brustle2025counts}
determine the standard Betti tables,
and \cref{equation:two-parameter-count}, below, which relates several of the invariants in the literature, also in the two-parameter case.
We overview end-curves in \cref{section:multiparameter-bottleneck}; see also \cref{figure:invariants}~\emph{(f)}.
For other known relationships, see
\cite{clause2022generalized,escolar2024barcoding},
where it is shown, for example, that the generalized persistence diagram and the spread-Euler characteristic are distinct and incomparable as invariants, in the sense that, in general, each one does not determine the other one \cite[Theorem~4.1$(ii)$]{escolar2024barcoding}.
For results on the time and space complexity of the computation of these invariants, see \cref{section:computational-aspects}.

A homological method for producing additive invariants of poset representations called \emph{M\"obius homology} was introduced in \cite{patel2026mobius}.
It was later observed that this homology theory is equivalent to standard Betti tables, or alternatively, to the definition in \cref{definition:graded-homology} for $\Lambda$ the incidence algebra \cite[Remark~4.2]{asashiba2026minimal};
see \cref{section:mobius-homology-equals-relative} for detailed statements and proofs.

\subsection{Bottleneck Stability, Negative Multiplicities, and Open Questions}
\label{section:multiparameter-bottleneck}

The most interesting results in persistence are often those that link algebra and geometry; the canonical example being algebraic/bottleneck stability, which links interleavings and matchings.
There are several newer results of this form, which we now describe.

The remarkable \cite[Theorems~4.3~and~4.12]{bjerkevik} say that stability of decompositions (in the sense of algebraic/bottleneck stability) does hold for suitable subcategories of multiparameter persistence modules:

\medskip
\noindent
\emph{%
For $M,N \in \veck^{\R^n}$
projective and $c_n = \max(n-1,1)$,
or $M$ and $N$ rectangle-decomposable and $c_n = 2n-1$, we have
$d_B(M,N) \leq c_n \cdot d_I(M,N)$.
}
\medskip

The constants $c_n$ are not known to be tight in general, and in fact \cite[Conjecture~6.5]{bjerkevik-2}, if true, would imply that $c_n$ can be made independent of $n$ in the case of projective representations.
This conjecture is equivalent to \cite[Conjecture~6.10]{bjerkevik-2}, which is a simple statement involving graphs and matrices, and which is the combinatorial essence of one of the main open questions in persistence: Bjerkevik's conjecture, below.

\newpage

Building on these results and global dimension bounds,
the following bottleneck stability result for Betti tables is proven in
\cite[Theorem~1.1]{oudot-scoccola} and \cite[Theorem~6.1]{botnan-et-al}:

\medskip
\noindent
\emph{%
For arbitrary $M,N \in \Veck_\fp^{\R^n}$,
and either $c_n = \max(n^2 - 1,2)$ and $\beta$ standard Betti tables,
or $c_n = (2n-1)^2$ and $\beta$ the $\rk$-Betti tables, we have
\begin{equation}
    \label{equation:bottleneck-stability-signed}
    d_B\Big(\;\beta_+(M) + \beta_-(N)\;,\; \beta_+(N) + \beta_-(M)\;\Big)\; \leq \; c_n \cdot d_I(M,N).
\end{equation}
}

\noindent
See \cref{figure:stability-signed-invariants} for an illustration of these stability results.
This type of stability result has a nice interpretation in terms of signed optimal transport, see \cite{divol-lacombe,bubenik-elchesen}, \cite[Section~6.4]{oudot-scoccola}, and \cite[Appendix~A.1]{pmlr-v235-scoccola24a}.

\begin{figure}
    \centering
    \includegraphics[width=\linewidth]{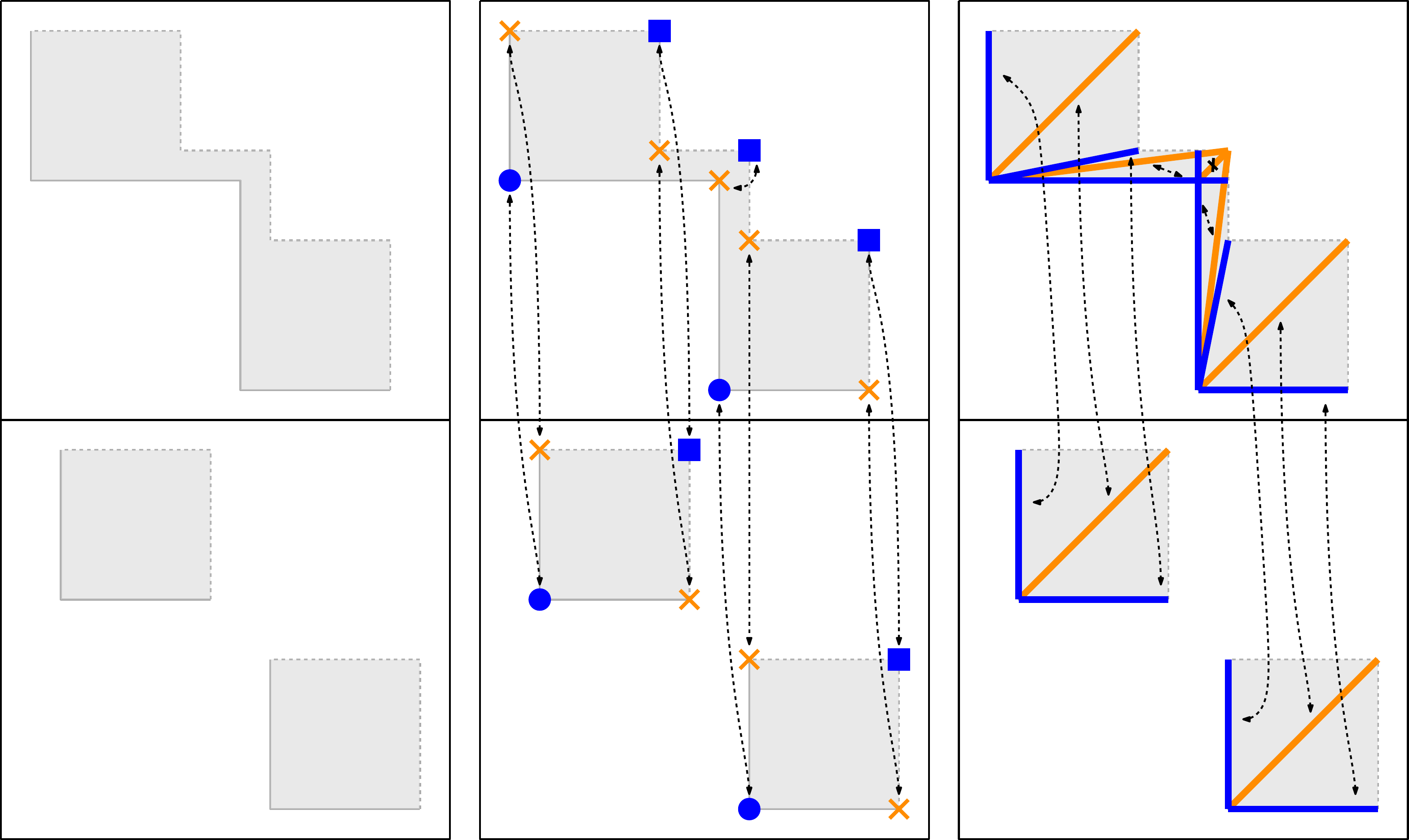}
    \caption{%
    \emph{Left.} Two spread-decomposable representations of $\Rbb^2$ that are close in the interleaving distance, but whose indecomposables admit no low-cost matching \cite{botnan-lesnick}.
    \emph{Center.} The signed matching guaranteed to exist by a main result in \cite{oudot-scoccola}.
    \emph{Right.} The signed matching guaranteed to exist by a main result in \cite{botnan-et-al}.}
    \label{figure:stability-signed-invariants}
\end{figure}

When expressing an additive invariant $\alpha(M)$ in a certain basis, the coefficients $\{c_i^M \in \Zbb\}$ are not always positive.
For example, this is usually the case for $\alpha$ the signed barcode or the generalized persistence diagram, when the poset is not a linear order.
The presence of negative multiplicities makes interpretation harder, but also has implications regarding bottleneck stability.
For example, the left-hand side of \cref{equation:bottleneck-stability-signed} does not satisfy the triangle inequality with respect to $M$ and $N$, and forcing it to do so makes the distance thus obtained only take the value $0$ unless $d_I(M,N) = \infty$ \cite[Proposition~1.4]{oudot-scoccola}; this makes the left-hand side of \cref{equation:bottleneck-stability-signed} often be a very weak lower bound.
This is a known issue concerning bottleneck/optimal transport-type distances in the presence of negatives \cite{mainini,mccleary-patel}.

In fact, there is a fundamental incompatibility between the interleaving distance and additive invariants.
Informally, it can be summarized as follows:
\begin{center}
\emph{All additive invariants of persistence modules are metrically trivial, in the uniformly continuous case.}
\end{center}
More formally, we have the following:

\medskip\noindent\textbf{No-Go Stability for Additive Invariants.}\;
\emph{Let $\Acal \subseteq \Veck^{\R^n}$ be any subcategory that is closed under shifts, finite sums, and pullbacks.
Let $G$ be a commutative monoid with a symmetric function $d_G : G \times G \to [0,\infty]$ with $d_G(g,g) = 0$ for all $g \in G$, and such that addition $G \times G \to G$ is uniformly continuous.
Let $\alpha : \Acal \to G$ be additive and uniformly continuous with respect to $d_I$ and $d_G$.
Assume:
\begin{enumerate}
    \item The function $d_G$ satisfies the triangle inequality;
    \item There exists a function $- : G \to G$ such that $d_G(c + (-c), 0) = 0$ for all $c \in G$.
\end{enumerate}
If $M, N \in \Acal$ are such that $d_I(M,N) < \infty$, then $d_G(\alpha(M), \alpha(N)) = 0$.}
\medskip

This result is a strong variant of \cite[Theorem~4.15]{berkouk-instability};
it is proven in the appendix as \cref{corollary:no-go}.
How to overcome this issue is an important open question in multiparameter persistence.
In the rest of this section, we outline some approaches.

First, it is instructive to understand why the no-go result above does not contradict some of the non-trivial stability results that we have seen so far.
The stability theorem of one-parameter persistence is non-trivial; it does not contradict the no-go result because $(2)$ is not satisfied in the case of one-parameter barcodes: barcodes only contain bars with positive multiplicities.
Similarly, the bottleneck stability results for multiparameter persistence in \cref{equation:bottleneck-stability-signed} hold in a setup where $(2)$ is satisfied, but where $(1)$ is not.

We conclude this section by describing several approaches to stability results in multiparameter persistence;
many of these deal with ``positive'' invariants, that is, additive invariants such that~$(2)$ is not satisfied.

\paragraph{Fibered barcode, erosion stability, and other metric approaches}
The \emph{fibered barcode} \cite{lesnick-wright-interactive} of a representation $M \in \veck^{\R^n}$ (introduced previously as the \emph{foliation method} \cite{biasotti-size-functions,cagliari-difabio-ferri})
consists of the barcodes of the restriction of $M$ to all monotonic lines in $\R^n$.
Restriction to lines of positive slope in every axis preserves interleavings, which leads to the stability result in \cite{cerri-betti-stable,landi-rank-stability}.
This stability result is not known to admit a compact combinatorial description in the form of a single matching, although certain coherent choices of matching exist in generic situations \cite{MR3533890,cerri-ethier-frosini}.
Other positive invariants include the rank invariant, the generalized rank invariant, and others such as \cite{Fersztand-et-al}, which is based on Harder--Narasimhan filtrations.
These invariants admit a form of stability sometimes known as \emph{erosion stability} \cite{puuska,kim-memoli,patel-mobius,fersztand}.
Erosion stability has no combinatorial representation as an infimum over matchings, and, in fact, it is a weaker form of stability than the usual bottleneck stability in the one-parameter case (in the sense that it provides a weaker lower bound for the interleaving distance); for an example, see \cite[Figure~5.2]{chazal-structure-stability} where the erosion distance is called box distance.
There also exist approaches that directly define other distances on representations using \emph{noise systems} \cite{MR3735858} and \emph{amplitudes} \cite{MR4776401}.

\paragraph{Approximate decompositions}
This approach \cite{bjerkevik-2} aims to circumvent the failure of the bottleneck stability of the indecomposable decomposition in multiparameter persistence (\cref{section:main-challenges}) more directly.
The way it does this is as follows: given interleaved representations $M$ and $N$, before looking for a matching between the indecomposable summands of $M$ and $N$, it first allows for certain controlled simplifications of the representations.
The following is the central construction in the approach.


Let $M \in \veck^{\R^n}$, and recall from \cref{section:fundamentals-one-parameter-persistence} the $\epsilon$-shift $M[\epsilon]$ of $M$ and the canonical morphism $\eta^M_\epsilon \colon M \to M[\epsilon]$.
Fix $\epsilon \geq 0$.
An \define{$\epsilon$-approximate endomorphism} of $M$ is simply a map $f : M \to M[2\epsilon]$.
The idea is to define the \define{$\epsilon$-pruning} as a certain subquotient of $M$ with the property that every $\epsilon$-approximate endomorphism of $M$ descends to an honest endomorphism of the $\epsilon$-pruning; see \cite[Definition~4.15]{bjerkevik-2}.
Interestingly, the $\epsilon$-pruning is an isomorphism invariant, but it is not additive.

The result \cite[Corollary~5.5]{bjerkevik-2} says that the pruning $\mathrm{Pru}_\epsilon(M)$ is a \emph{$2r\epsilon$-refinement} of every representation $\epsilon$-interleaved with $M$, where $r$ is the largest pointwise dimension of $M$,
and where an \define{$\epsilon$-refinement} of a representation $N$, with indecomposable decomposition $N \cong \bigoplus_i N_i$, is a representation isomorphic to a direct sum $\bigoplus_i A_i/B_i$, where $B_i \subseteq A_i \subseteq N_i$ are subrepresentations satisfying $\im\big(\eta^{N_i[-\epsilon]}_\epsilon\big) \subseteq A_i$ and $B_i \subseteq \ker\big(\eta^{N_i}_\epsilon\big)$.
How much the dependence on $r$ can be improved is the content of the following conjecture:

\medskip
\noindent
\textbf{Conjecture (Bjerkevik).}
\emph{Let $\epsilon \geq 0$, let $M, N \in \veck^{\R^n}$ be $\epsilon$-interleaved, and let $r$ be the maximum pointwise dimension among indecomposable summands of $M$.
There is a constant $c$, independent of $M$, $N$, $\epsilon$, and $r$, and a representation $Q$ that is a $cr\epsilon$-refinement of both $M$ and $N$.}
\medskip

\paragraph{Positive counts and end-curves}
The final approach we mention is that of \cite{brustle2025counts}, which hinges upon the theory of string algebras.
Given an additive invariant $\alpha$ together with a basis $\Bcal$, one defines the \emph{total multiplicity} of the signed invariant $\alpha^\Bcal(M)$ evaluated on a representation $M$ as the integer $\Ncal^{\alpha,\Bcal}(M) = \sum c_i^M \in \Nbb$, interpreted as the number of positive bars minus the number of negative bars according to the signed invariant $\alpha^\Bcal$.
It is proven in \cite{brustle2025counts} that, although there is no satisfactory notion of barcode for multiparameter persistence that is always positive, there does exist a canonical notion of \emph{bar-count} $\Ncal^2$ for two-parameter persistence that is analogous to the count of bars in one-parameter persistence, and that,
remarkably, is positive and can be computed as the total multiplicity of standard signed invariants such as the signed barcode, the generalized persistence diagram, and the spread Euler characteristic \cite[Theorem~A]{brustle2025counts}:
\begin{equation}
\label{equation:two-parameter-count}
    \Ncal^2(M) = \Ncal^{\mathrm{gen.pers.diag.}}(M) = \Ncal^{\mathrm{sign.barc.}}(M) = \Ncal^{\mathrm{spread.Eul.char.}}(M)\; \geq\; 0.
\end{equation}
This also has a geometric description: as the number of \emph{end-curves} of the representation $M$, which are analogues of the end-points of bars in one-parameter persistence; see \cref{figure:invariants}~\emph{(f)}.
The existence of suitable bottleneck stability results for the count, end-curves, or related invariants is an open question.

\section{Representation Theory of Finite and Infinite Posets}
\label{section:challenges-infinite-continuous}

The development of persistence gives further motivation for the study of the representation theory of posets.

\paragraph{Representation type}
The classification of finite posets according to their representation type has been carried out in 
\cite{loupias,zavadskij,drozdowski,leszczynski1,leszczynski2}.
This classification is much more involved than that of quivers without relations, but it is still done using the classical trichotomy finite type, tame type, and wild type.

To the best of our knowledge, there is no currently known trichotomy for categories of representations of infinite posets.
In order to prove such a result, the notions of finite, tame, and wild type need to be adjusted, since, for example, the poset $\Rbb$ ought to be of ``finite representation type'': even though it admits infinitely non-isomorphic indecomposable representations, the Structure Theorem implies that the classification of its \pfd{} representations is straightforward.
A possible extension of the notion of finite representation type is given in \cite{bauer-scoccola}, where a (potentially infinite) poset is said to be of \emph{pointwise finite representation type} if there exists a uniform bound on the pointwise dimension of all its \pfd{} indecomposable representations.

The classification of (potentially infinite) posets of pointwise finite representation type is an open question. Partial results in this direction are known, specifically in connection to
\emph{continuous quivers of type $A$} \cite{igusa-rock-todorov}
and \emph{thread quivers} \cite{berg-roosmalen,paquette2024categories}.
The representation theory of continuous quivers of type $A$ relates to relative interlevel set persistence
(\cref{section:levelset-persistence}), although this relationship is still being studied.
The \pfd{} representations of~$S^1$, a continuous extension of cyclic quivers, are classified in \cite{hanson-rock};
it is an open question whether there exists an algebraic/bottleneck stability result for these representations.
For related work on the representation type of locally bounded categories, see \cite{bongartz-gabriel,dung-garcia}.

\paragraph{Tameness assumptions on representations}
When working with finite posets (or quivers) it is standard to consider only finite dimensional representations.
When the poset is infinite, the natural generalization of this tameness assumption is that of being \pfd{}.
As stated in \cref{section:main-challenges}, \pfd{} representations always decompose uniquely into indecomposables \cite{CrawleyBoevey1994,botnan-crawley}.
However, \pfd{} is often not the right tameness assumption, and here we elaborate on this.

Many of the invariants in \cref{section:invariants} and in the literature have only been worked out for finite (or otherwise discrete) posets,
and several of these (e.g.,~Betti tables) do not admit a satisfactory generalization to \pfd{} representations of posets such as $\Rbb^n$.

A more stringent tameness assumption found in the literature is finitely presentable (\fp{}).
This is sufficiently general to encompass the sublevel set multiparameter persistent homology of finite filtered simplicial complexes,
but it does not include cases such as multidimensional Morse functions.
A tameness assumption more general than \fp{}, and more restrictive than \pfd{}, is that of \emph{m-tame} representations \cite{miller-siaga}\footnote{Although the reference uses the term \emph{tame}, we use m-tame to avoid confusion with other notions of tameness in the literature; the letter ``m'' refers to the author of \cite{miller-siaga}.}.
Briefly, a \pfd{} representation $M : \Pcal \to \Veck$ is m-tame if it factors as $\Pcal \to \Qcal \to \Veck$, with $\Qcal$ a finite poset.
This is significantly more general than \fp{}; for example, it includes
the sublevel set persistence of two-dimensional Morse functions
\cite[Corollary~2.11]{budney-kaczynski},
subanalytically constructible representations \cite[Theorem~4.5']{miller-stratification},
and, in particular, the representation $\kbb_D$ in \cref{figure:spread}.
It is shown in \cite{miller-siaga} that m-tame representations can be studied via resolutions by upset-decomposable representations\footnote{Note that these resolutions are exact in the standard sense, so this does not fit the framework of relative homological algebra of \cref{section:invariants}.} (i.e., representations decomposing as a direct sum of indicator representations of upsets).
In particular, the following syzygy result is proven there:

\medskip
\noindent\emph{%
A representation over any poset is m-tame if and only if
it admits a finite resolution by upset-decomposable representations.}
\medskip

See the reference for a more general statement.
It is an open question whether there exists a uniform bound for the length of these resolutions when the poset is, e.g., $\Rbb^n$; see \cite[Section~13]{miller2020essential}.
For more on m-tame representations, see also \cite{waas2024notes}.

Let us conclude by mentioning a generalization of \pfd{} representations.
A representation of $\Rbb^n$ is \emph{q-tame} 
if all its structure morphisms corresponding to pairs of indices $r \leq s \in \Rbb^n$ with $r_i < s_i$ for all $i \in \{1, \dots, n\}$ have finite rank.
This was defined in the one-parameter case in \cite{chazal-structure-stability}, and generalized to multiparameter persistence in \cite{bauer2026metrically}, building on \cite{berkouk-petit,harsu2024ephemeral}.
There are a few reasons to go beyond \pfd{} representations.
First there are cases of interest where the sublevel set persistence of a function is not \pfd{} but is q-tame \cite{cagliari-landi}.
Second, at this level of generality the algebraic and geometric properties of poset representations are particularly compelling \cite[Theorem~1.1]{bauer2026metrically}:

\medskip
\noindent\emph{%
The observable category of q-tame representations of $\Rbb^n$ satisfies the following:
\begin{itemize}
    \item It is Abelian and Krull--Schmidt, in the sense that every object decomposes as a direct sum of indecomposables in an essentially unique way.
    \item Two objects are isomorphic if and only if they are at interleaving distance zero.
    \item The extended metric induced by the interleaving distance is complete, in the sense that every Cauchy sequence has a limit.
\end{itemize}
}
\medskip

See the reference for the notion of observable category, first introduced in the one-parameter case in \cite{chazal-crawley-silva}.

\paragraph{The special case of zero-dimensional persistent homology}
As mentioned in \cref{section:main-challenges}, common posets in multiparameter persistence are of wild representation type.
However, in cases such as level set persistence (\cref{section:levelset-persistence}), the category of representations of interest is not the whole category of representations of a wild poset, but rather a subcategory of pointwise finite representation type.
With this motivation, it is proposed in \cite{bauer-botnan-oppermann-steen,brodzki2020complexity} to study the image of zero-dimensional multiparameter persistent homology, since homology in dimension zero is often structurally simpler and can be computed more efficiently.
Equivalently, they propose to study the image of the linearization functor
$\mathrm{Set}^\Pcal \to \Veck^\Pcal$, given by postcomposition with the free vector space functor $\mathrm{Set} \to \Veck$.

Using cotorsion torsion triples, it is shown in \cite{bauer-botnan-oppermann-steen} that, in the two-parameter case, certain categories containing this image unfortunately remain of wild representation type.
However, it is shown in \cite{bindua-brustle-scoccola} that, over tree posets, the image of zero-dimensional persistent homology is of finite representation type, even though tree posets (which are also tree quivers) are usually of wild type;
the characterization is in terms of \emph{tree modules} \cite{ringel} and \emph{reduced representations} over tree quivers \cite{kinser}.
Further algebraic and representation theoretic properties of zero-dimensional persistent homology are studied in \cite{bauer2025additive,morozov-scoccola}.

\section{Computational Complexity of Persistence}
\label{section:computational-aspects}

Applications of persistence have motivated the development of many algorithms for the efficient computation of persistent homology and of poset representation invariants.
These include the following:
\begin{itemize}
    \item The \emph{persistence algorithm}, a cubic-time algorithm for computing the barcode decomposition of the persistent homology of one-parameter filtered simplicial complexes \cite{edelsbrunner-letscher-zomorodian,zomorodian-carlsson},
    a theoretically more efficient algorithm running in matrix multiplication-time \cite{milosavljevic-morozov-skraba,morozov-skraba},
    and algorithms implementing several practical speedups \cite{bauer},
    which in practice often run in nearly linear time \cite{bauer-et-al}.
    In the special case of zero-dimensional persistent homology, a loglinear-time algorithm sometimes referred to as the \emph{elder rule} \cite[Chapter~VII.2]{edelsbrunner-harer-2}.
    \item A cubic-time algorithm for computing minimal presentations and Betti tables of finitely presented representations of $\Rbb^2$ and of bifiltered simplicial complexes \cite{lesnick-wright}. In the special case of zero-dimensional persistent homology, a loglinear-time algorithm for the two-parameter case, and a quadratic-time algorithm for computing a minimal presentation over arbitrary indexing posets \cite{morozov-scoccola}.
    Newer algorithms \cite{bauer2026dualities} exploit computational advantages of persistent cohomology, and classical dualities relating
    the Betti tables of a multiparameter persistence module to those of its dual \cite{miller-alexander-duality}.
    \item An output-sensitive algorithm for computing the $\rk$-Euler characteristic (and hence also the signed barcode) of finitely presented representations of $\Rbb^2$ and of bifiltered simplicial complexes \cite{morozov2023}, which is more efficient than the more direct quartic-time algorithm \cite{BotnanOppermannOudot2023}.
    \item A poly-time algorithm for computing the generalized rank invariant of a finitely presented representation of $\Rbb^2$ on a fixed spread \cite{dey-kim-memoli}.
    It is unknown if computing the generalized persistence diagram or the spread-Euler characteristic on a fixed spread (i.e., computing the multiplicity of the fixed spread) is poly-time.
    \item A poly-time algorithm for computing $\Ycal$-Betti tables when $\Ycal$ can be indexed by a lattice in a suitable way \cite{chacholski-et-al}.
    \item An algorithm for computing the indecomposable decomposition of finitely presented representations of $\Rbb^n$ \cite{dey2025}; it runs in polynomial time on uniquely graded modules and in cubic time on spread-decomposable modules, but its worst-case complexity is not polynomial in general.
    The decomposition problem itself is nonetheless solvable in polynomial time \cite{chistov-ivanyos-karpinski}; see also \cite{mallory-sayrafi}.
    \item A cubic-time algorithm for computing a minimal presentation of the end-curves \cite[Section~8]{brustle2025counts}.
    \item A fixed-parameter tractable approximation algorithm and, for a fixed number of parameters, a polynomial-time algorithm for computing the skyscraper invariant of \cite{Fersztand-et-al}; these also compute the wall-and-chamber decomposition of a finitely presented representation \cite{fersztand-jendrysiak}.
\end{itemize}
There are also several interesting results concerning the computational and space complexity of several concepts in persistence, including:
\begin{itemize}
    \item The interleaving distance is NP-hard to approximate within any constant factor smaller than~$3$ \cite{bjerkevik-botnan-kerber}.
    Computing the $p$-presentation distance is also NP-hard for every $p \in [1,\infty)$ \cite{bjerkevik2025computing}.
    \item In two-parameter persistence, the $\rk$-Euler characteristic (and hence also the signed barcode) is of size at most cubic \cite{morozov2023}.
    The $\rk$-Betti tables are poly-size in any number of parameters \cite[Theorem~5.18]{botnan-et-al}, and conjectured to also be cubic in two-parameter persistence \cite[Conjecture~5.30]{botnan-et-al}.
    \item The end-curves are of linear size
    \cite[Theorem~D~and~Corollary~E]{brustle2025counts}.
    \item The generalized persistence diagram is not poly-size already in two-parameter persistence \cite{kim-kim-lee}.
    In fact, we conjecture that the generalized persistence diagram and the spread-Euler characteristic are of exponential size.
\end{itemize}
Here the complexity is stated in terms of the size of the input, which is usually taken to be the number of simplices in the case of filtered simplicial complexes, and the number of generators and relations in a minimal presentation in the case of representations;
see the references, e.g., \cite[Result~E]{botnan-et-al}.

There is an important caveat here:
When computing homology in dimension~$d$, a geometric complex $K$ (\cref{section:persistence-data-analysis}) is often of size $|K| \in O(|X|^{d+2})$, where $X$ is the input data.
Thus, algorithms of polynomial complexity $O(|K|^q)$ in the input simplicial complex $K$ end up being of complexity $O(|X|^{(d+2)q})$ in this case.
This has motivated the study of sparsification of geometric complexes
\cite{mischaikow-nanda,boissonat-pritam,sheehy,lesnick2024sparse}
and of alternatives to geometric complexes based on covers \cite{leitao2026s,pmlr-v267-scoccola25a}.

%% file: content-appendix.tex
\appendix

\addtocontents{toc}{\protect\setcounter{tocdepth}{1}}

\section{Proofs}

\subsection{Betti Tables and M\"obius Homology}
\label{section:mobius-homology-equals-relative}

Fix a finite poset $\Pcal$.
In this appendix we prove the following results, which relate M\"obius homology (denoted $H^\downarrow$) \cite{patel-skraba} to previous invariants in persistence, specifically standard Betti tables and Betti tables relative to the rank exact structure.
See below for notations used.

\begin{theorem}
    \label{theorem:main-thm-1}
    If $M \in \veck^\Pcal$, $i \in \Pcal$, and $k \in \Nbb$,
    then $\dim (H_k^\downarrow M)(i) = \beta_M^k(i)$.
\end{theorem}

\begin{theorem}
    \label{theorem:main-thm-2}
    If $M \in \veck^\Pcal$, $a < b \in \Pcal \cup \{\infty\}$, and $k \in \Nbb$,
    then $\dim (H_k^\downarrow K_M)(a,b) = \rho_M^k(a,b)$.
\end{theorem}

If $i \in \Pcal$, let $\Psf_i \in \veck^\Pcal$ denote the corresponding indecomposable projective; this is the spread module with support the principal upset $\Pcal_{\geq i} \subseteq \Pcal$ corresponding to $i$.
Recall that any projective in $\veck^\Pcal$ is a finite direct sum of modules of the form $\Psf_i$.
Let also $\Ssf_i \in \veck^\Pcal$ denote the corresponding simple; this is the spread module with support $\{i \} \subseteq \Pcal$.

The \define{incidence algebra} of $\Pcal$, denoted $\kbb \Pcal$, is the $\kbb$-algebra generated, as a vector space, by pairs $[i,j]$ with $i \leq j \in \Pcal$, and with multiplication given by linearly extending the rule $[i,j] [k,\ell] = 0$ if $j \neq k$ and $[i,j] [j,\ell] = [i, \ell]$.
It is a standard fact that there is an equivalence of categories $\veck^\Pcal \simeq \mod_{\kbb \Pcal}$; see, e.g., \cite[Lemma~2.1]{botnan-et-al}.
In particular, this implies that any representation $M \in \veck^\Pcal$ admits a (finite) projective cover, unique up to isomorphism \cite[Chapter~I,~Theorem~4.2]{auslander-reiten-smalo}, and thus a projective resolution, also unique up to isomorphism.

Let $\Int\Pcal = \{(a,b) : a \leq b \in \Pcal \cup \{\infty\}\}$ with the product order.
Given $M \in \veck^\Pcal$, we get $K_M \in \veck^{\Int\Pcal}$ defined as $K_M(a,b) \coloneqq \ker(M(a) \to M(b))$, with the convention that $M(\infty) = 0$.

\begin{definition}
    Let $M \in \veck^\Pcal$ and $k \in \Nbb$.
    The \define{$k$th Betti table} of $M$ is the function $\beta_M^k : \Pcal \to \Zbb$
    where $\beta_M^k(i)$ equals the multiplicity of $\Psf_i$ in the $k$th homological degree of the minimal projective resolution of $M$.
\end{definition}

\begin{definition}
    The \define{hook module} corresponding to $(a,b) \in \{(a,b) \in \Int \Pcal : a \neq b\}$ is $\Lsf_{a,b} \in \veck^\Pcal$ given as the cokernel of $\Psf_b \xrightarrow{1} \Psf_a$ if $b < \infty$ and simply by $\Psf_a$ if $b = \infty$.
    A module is \define{hook-decomposable} if it is a direct sum of hooks.
\end{definition}

\begin{definition}
    Let $M \in \veck^\Pcal$ and $k \in \Nbb$.
    The \define{$k$th Betti table relative to the rank exact structure} of $M$ is the function $\rho_M^k : \{(a,b) \in \Int \Pcal : a \neq b\} \to \Zbb$
    where $\rho_M^k(a,b)$ equals the multiplicity of $\Lsf_{a,b}$ in the $k$th homological degree of the minimal rank-projective resolution of $M$.
\end{definition}

\begin{notation}
For readability, fix $i \in \Pcal$ and let $F_n(M) \coloneqq (H_n^\downarrow(M))(i)^*$.
\end{notation}

\begin{lemma}
    \label{lemma:ses-les}
    Let $0 \to L \to N \to M \to 0$ be a short exact sequence of representations $L,N,M \in \veck^\Pcal$.
    There exists a long exact sequence as follows, functorial as a map from short exact sequences to long exact sequences:
\begin{align*}
    0 &\to F_0 M \to F_0 N \to F_0 L \to F_1 M \to F_1 N \to F_1 L \to \cdots \\
    \cdots &\to F_n M \to F_n N \to F_n L \to F_{n+1} M \to \cdots
\end{align*}
\end{lemma}
\begin{proof}
    Since M\"obius homology is defined pointwise using relative homology of a category with coefficients in a functor, a short exact sequence of coefficients induces a long exact sequence of relative homology (this is used in, e.g., \cite[Proposition~6.14]{patel-skraba}).
    As $H^\downarrow$ is covariant in the coefficients, this long exact sequence increases in homological degree and has the form $\cdots \to H_n^\downarrow L \to H_n^\downarrow N \to H_n^\downarrow M \to H_{n-1}^\downarrow L \to \cdots$; applying the exact contravariant functor $(-)^*$ reverses all arrows and yields the stated sequence for $F_\bullet$.
\end{proof}

\begin{lemma}
    \label{lemma:mobius-homology-acyclic}
    We have $F_n \Psf_j \cong \kbb$ if $n=0$ and $i=j$, and $F_n \Psf_j = 0$ otherwise.
\end{lemma}
\begin{proof}
    If $j \nleq i$, then $C_\bullet(\Delta \Pcal_{\leq i}, \Delta \Pcal_{< i}; \Psf_j) = 0$, since $\Psf_j$ is $0$ on $\Pcal_{\leq i}$.

    If $j = i$, then $C_\bullet(\Delta \Pcal_{\leq i}, \Delta \Pcal_{< i}; \Psf_j) = C_\bullet(\Delta \Pcal_{\leq j}; \Psf_j)$, because $\Psf_j$ is zero on any element in $\Pcal_{< j}$.
    Since $\Psf_j$ restricted to $\Pcal_{\leq j}$ is $\kbb$ at $j$ and $0$ elsewhere, we have that $F_n(\Psf_j)$ is $0$ if $n>0$ and $\kbb$ if $n=0$.

    If $j < i$, let $\Rcal = \{a \in \Pcal : j \leq a \leq i\}$.
    Then $C_\bullet(\Delta \Pcal_{\leq i}, \Delta \Pcal_{< i}; \Psf_j) = C_\bullet(\Delta \Rcal, \Delta \Rcal_{< i}; \kbb)$, since any chain in $\Pcal$ whose starting element is strictly less than $j$ gets assigned the zero vector space, and every other chain gets assigned $\kbb$.
    Now, since $\Rcal$ has a maximal element, the simplicial complex $\Delta \Rcal$ is contractible, and since $\Rcal_{< i}$ has a minimal element, the simplicial complex $\Delta \Rcal_{< i}$ is also contractible.
    Since the homology of $C_\bullet(\Delta \Rcal, \Delta \Rcal_{< i}; \kbb)$ is usual relative homology of simplicial complexes, it follows that this homology is isomorphic to $\tilde{H}_\bullet(\Delta \Rcal /(\Delta \Rcal_{< i}))$, which is the reduced homology of a quotient of contractible simplicial complexes, and thus zero.
\end{proof}

\begin{lemma}
    \label{lemma:mobius-homology-degree-zero}
    Let $M \in \veck^\Pcal$.
    Then $F_0 M \cong \Hom(M, \Ssf_i)$, natural in $M$.
\end{lemma}
\begin{proof}
    We have the following isomorphisms, natural in $M$:
    \begin{align*}
        \Hom(M, \Ssf_i) &\cong \left(\frac{M(i)}{\sum_{j < i} \im(M(j) \to M(i))}\right)^*\\
                        &\cong \coker\big( C_1(\Delta \Pcal_{\leq i}, \Delta \Pcal_{< i}; \underline{M}) \to C_0(\Delta \Pcal_{\leq i}, \Delta \Pcal_{< i}; \underline{M}) \big)^*\\
                        &\cong (H_0^\downarrow M)(i)^* \cong F_0 M.\qedhere
    \end{align*}
\end{proof}

The following can be interpreted as a universal property/representability theorem of $\Ext$.
A more general universal property for derived functors can be proven using the language of $\delta$-functors (see, e.g., \cite{weibel}).
Informally, the basic idea is that two homology theories coincide when they have the same acyclic objects and are naturally isomorphic in degree zero.

\begin{theorem}[{\cite[Theorem~10.1]{maclane}}]
    \label{theorem:ext-universal-property}
    Let $R$ be a ring and $G$ an $R$-module.
    Let $\{Ex^n(-) : \mod_R^{\mathrm{op}} \to \Ab\}_{n \in \Nbb}$ be a family of functors satisfying the following properties:
    \begin{enumerate}
        \item $Ex^-(-)$ extends to a functor from short exact sequences to long exact sequences;
        \item $Ex^n(F) = 0$ for $n>0$ and all free $F \in \mod_R$;
        \item there is a natural isomorphism $Ex^0(M) \cong \Hom(M,G)$.
    \end{enumerate}
    Then, there are natural isomorphisms $\Ext^n(-,G) \cong Ex^n(-)$ for all $n \in \Nbb$.
    \qed
\end{theorem}



\begin{proposition}
    \label{proposition:mobius-homology-as-ext}
    Let $M \in \veck^\Pcal$, $i \in \Pcal$, and $k \in \Nbb$.
    We have a natural isomorphism
    \[
        (H_k^\downarrow M)(i)^* \cong \Ext^k_{\veck^\Pcal}(M, \Ssf_i),
    \]
    where $(-)^*$ denotes dualization of vector spaces.
\end{proposition}
\begin{proof}
    This follows from \cref{theorem:ext-universal-property}, using \cref{lemma:ses-les,lemma:mobius-homology-acyclic,lemma:mobius-homology-degree-zero}, to satisfy conditions (1), (2), and (3), respectively.
\end{proof}

\begin{remark}
    The dualization $(-)^*$ in \cref{proposition:mobius-homology-as-ext} is required to get the right variance.
    One can also use $\Tor$ and avoid dualization; in this case, one needs to observe that the simple $\Ssf_i$ is a $\kbb \Pcal$-bimodule.
\end{remark}

The following is standard; see, e.g., \cite[Lemma~5.20]{botnan-et-al}.

\begin{lemma}
    \label{remark:mobius-is-betti}
    If $M \in \veck^\Pcal$, $i \in \Pcal$, and $k \in \Nbb$, then
    $\dim(\Ext^k_{\veck^\Pcal}(M, \Ssf_i))$ is equal to the multiplicity of $\Psf_i$ in the $k$th homological degree of the minimal resolution of $M$.
    \qed
\end{lemma}

\begin{proof}[Proof of \cref{theorem:main-thm-1}]
    This follows from \cref{proposition:mobius-homology-as-ext,remark:mobius-is-betti}.
\end{proof}

\begin{notation}
    Let $\Lcal \coloneqq \bigoplus_{a < b \in \Pcal \cup \{\infty\}} \Lsf_{a,b}$ and let $\End(\Lcal)$ be its endomorphism algebra.
\end{notation}


\begin{definition}
    A sequence $0 \to L \to N \to M \to 0$ in $\veck^\Pcal$ is \define{rank-exact}
    if it is exact and additive on rank invariants, in the sense that $\rk(N) = \rk(L) + \rk(M)$.
\end{definition}

\begin{proposition}[{\cite{BotnanOppermannOudot2023}}]
    We have the following properties of the rank exact structure.
    \begin{itemize}
    \item A sequence $E : 0 \to L \to N \to M \to 0$ in $\veck^\Pcal$ is rank-exact if and only if $\Hom(\Lcal, E)$ is a short exact sequence of vector spaces.
    \item A module is rank-projective if and only if it is hook decomposable.\qed
    \end{itemize}
\end{proposition}

Since $\Lcal$ is a left $\End(\Lcal)$-module, the functor $\Hom(\Lcal, -)$ gives a functor
\[
    \Hom(\Lcal, -) : \veck^\Pcal \to \mod_{\End(\Lcal)}.
\]

\begin{lemma}
    \label{lemma:interval-algebra-and-endo-algebra}
    Define $e \in \kbb(\Int \Pcal)$ by $e \coloneqq \sum_{c \in \Pcal} [(c,c),(c,c)]$.
    We have an isomorphism of algebras
    \[
        \End(\Lcal) \cong \kbb(\Int \Pcal) / (e).
    \]
    Under this identification, we have the following
    \begin{enumerate}
    \item If $M \in \veck^\Pcal$, we have $\Hom(\Lcal,M) \cong K_M \otimes_{\kbb(\Int \Pcal)} \left(\kbb(\Int \Pcal) / (e)\right)$.
    \item If $a < b \in \Pcal \cup \{\infty\}$, then $\Hom(\Lcal, \Lsf_{a,b}) \cong \Psf_{(a,b)} \otimes_{\kbb(\Int \Pcal)} \left(\kbb(\Int \Pcal) / (e)\right)$, where $\Psf_{(a,b)} \in \veck^{\Int \Pcal}$ is the corresponding indecomposable projective.
    \item If $a \in \Pcal \cup \{\infty\}$, then $\Psf_{(a,a)} \otimes_{\kbb(\Int \Pcal)} \left(\kbb(\Int \Pcal) / (e)\right) = 0$.
    \end{enumerate}
\end{lemma}
\begin{proof}
    The isomorphism of algebras is proven in \cite[Lemma~5.10]{botnan-et-al}
    (note that in \cite[Section~5.1]{botnan-et-al} the poset is assumed to be a lattice, but this extra structure is not used in Lemma~5.10).

    Statement (1) is due to the fact that $K_M(a,b) = \ker(M(a) \to M(b)) \cong \hom(\Lsf_{a,b},M)$, and statements (2) and (3) are a straightforward check.
\end{proof}

\begin{proposition}[{cf.~\cite[Lemma~5.6]{botnan-et-al}}]
    \label{lemma:preservation-betti-tables-outside-diagonal}
    Let $N \in \veck^{\Int \Pcal}$ be such that, for every $a \leq b \leq c \leq d \in \Pcal \cup \{\infty\}$ the structure morphism $N(a,b) \to N(c,d)$ is zero.
    For every $d > 0 \in \Nbb$, we have
    \[
        \Tor_d^{\kbb(\Int \Pcal)}
        \left(
        N, \kbb(\Int \Pcal) / (e)
        \right) = 0.
    \]
\end{proposition}
\begin{proof}
    We have an exact sequence of left $\kbb(\Int \Pcal)$-modules
    \[
        0 \to (e) \to \kbb(\Int \Pcal) \to \kbb(\Int \Pcal) / (e) \to 0.
    \]
    This implies that, for all $d \geq 2$, we have
    \[
        \Tor_d^{\kbb(\Int \Pcal)}
        \left(
        N, \kbb(\Int \Pcal) / (e)
        \right)
        \cong \Tor^{\kbb(\Int \Pcal)}_{d-1}\left(N, (e)\right),
    \]
    and that, in the case $d=1$, we have an injection of the left-hand side into the right-hand side.
    Thus, it suffices to show that
    $\Tor^{\kbb(\Int \Pcal)}_{d}\left(N, (e)\right) = 0$ for all $d \geq 0$.
    We first make an observation, and then consider the cases $d = 0$ and $d > 0$ separately.

    \medskip

    \noindent \emph{Observation.}
    We have $(e) = \Span \{ [(a,b), (c,d)] : b \leq c \}$, where the $\Span$ is as a vector space.

    \medskip

    \noindent \emph{Case $d=0$.}
    We have $\Tor_0(N, (e)) \cong N \otimes_{\kbb(\Int \Pcal)} (e)$, and we show that all elementary tensors vanish.
    Let $x \in N$ and let $[(a,b), (c,d)] \in (e)$.
    Then
    \[
        x \otimes [(a,b),(c,d)] = x [(a,b),(c,d)] \otimes [(c,d),(c,d)] = 0 \otimes [(c,d),(c,d)] = 0,
    \]
    since $[(a,b),(c,d)] = [(a,b),(c,d)] \, [(c,d),(c,d)]$ and the structure morphism $N(a,b) \to N(c,d)$ is zero, by assumption.

    \medskip

    \noindent \emph{Case $d>0$.}
    It is enough to prove that $(e)$ is projective as a left $\kbb(\Int \Pcal)$-module.
    As a left module, we have $(e) \cong \bigoplus_{c \leq d} \Span\{[(a,b), (c,d)] : b \leq c\}$, where the $\Span$ is as left module, so it is sufficient to prove that $\Span\{[(a,b), (c,d)] : b \leq c\}$ is projective for fixed $c \leq d$.
    We claim that $\Span\{[(a,b), (c,d)] : b \leq c\} \cong \Psf_{(c,c)}$, as left modules over $\kbb(\Int \Pcal)$.
    To see this, note that $\Psf_{(c,c)} \cong \Span\{[(a,b),(c,c)] : b \leq c\}$, by definition,
    and that there is an isomorphism
    \[
        \Span\{[(a,b),(c,c)] : b \leq c\} \cong
        \Span\{[(a,b),(c,d)] : b \leq c\}
    \]
    given by multiplying by $[(c,c),(c,d)]$ on the right.
\end{proof}

The following is a general fact \cite[Chapter~II,~Section~2]{auslander-reiten-smalo} instantiated to the hook modules.

\begin{lemma}[{cf.~\cite[Lemma~5.8]{botnan-et-al}}]
    \label{lemma:relative-to-absolute}
    The functor $\Hom(\Lcal, -): \veck^\Pcal \to \mod_{\End(\Lcal)}$ restricts to an equivalence of categories between the full subcategory of $\veck^\Pcal$ spanned by hook decomposable modules and the full subcategory of $\mod_{\End(\Lcal)}$ spanned by projective modules.
\end{lemma}

In particular, rank exact resolutions of $M \in \veck^\Pcal$ are in natural bijection with $\mod_{\End(\Lcal)}$-projective resolutions of $\Hom(\Lcal,M) \in \mod_{\End(\Lcal)}$.

\begin{proof}[Proof of \cref{theorem:main-thm-2}]
    By \cref{lemma:relative-to-absolute},
    it is enough to prove that $\dim(H_k^\downarrow K_M)(a,b)$ equals the multiplicity of $\Hom(\Lcal, \Lsf_{a,b})$
    in homological degree $k$ in the minimal $\mod_{\End(\Lcal)}$-projective resolution of $\Hom(\Lcal, M)$.
    By \cref{theorem:main-thm-1}, we have that $\dim(H_k^\downarrow K_M)(a,b)$ equals the multiplicity of $\Psf_{(a,b)} \in \veck^{\Int \Pcal}$ in homological degree $k$ in the minimal projective resolution of $K_M$.
    So we are left with showing that
    the multiplicity of $\Hom(\Lcal, \Lsf_{a,b})$
    in homological degree $k$ in the minimal $\mod_{\End(\Lcal)}$-projective resolution of $\Hom(\Lcal, M)$ equals the multiplicity of $\Psf_{(a,b)} \in \veck^{\Int \Pcal}$ in homological degree $k$ in the minimal projective resolution of $K_M$.

    By \cref{lemma:interval-algebra-and-endo-algebra}(1), we have $\Hom(\Lcal,M) \cong K_M \otimes_{\kbb(\Int \Pcal)} \left(\kbb(\Int \Pcal) / (e)\right)$.
    Tensoring with
    $\kbb(\Int \Pcal) / (e)$ preserves any projective resolution of $K_M$, by \cref{lemma:preservation-betti-tables-outside-diagonal}.
    Using the minimal projective resolution of $K_M$, we get, after tensoring, a resolution by projective $\End(\Lcal)$-modules in which, by \cref{lemma:interval-algebra-and-endo-algebra}(2) and (3), the indecomposable projective $\Psf_{(a,b)}$ is sent to $\Hom(\Lcal, \Lsf_{a,b})$ when $a < b$ and to $0$ when $a = b$; in particular only the projectives corresponding to pairs $(a,b)$ with $a<b$ are kept.
    This resolution is again minimal: the projectives $\Hom(\Lcal, \Lsf_{a,b})$ with $a < b$ are pairwise non-isomorphic indecomposables, so tensoring sends radical morphisms between projectives to radical morphisms, and hence the differentials of the minimal resolution of $K_M$ remain in the radical.
    The result follows.
\end{proof}

\subsection{No-Go Stability for Additive Invariants}
\label{section:no-go-result}

If $A$ is a commutative monoid, $d_A$ is an extended pseudometric on $A$, and $p \in [1,\infty]$, we say that $d_A$ is \define{$p$-subadditive} if $d_A(a+b,c+d) \leq \|(d_A(a,c), d_A(b,d))\|_p$ for all $a,b,c,d \in A$.
For example, the interleaving distance on (the set of isomorphism classes of objects of) $\Veck^{\R^n}$ is $\infty$-subadditive.

Recall that the \define{length} of a continuous path $\gamma : [0,L] \to A$, where $A$ carries the topology induced by $d_A$, is
\[
    L(\gamma) \coloneqq \sup_{0 = t_0 < t_1 < \cdots < t_n = L} \; \sum_{i=0}^{n-1} d_A(\gamma(t_i), \gamma(t_{i+1})) \;\in\; [0,\infty].
\]
We say that $d_A$ is \define{intrinsic} if $d_A(a,b) = \inf_\gamma L(\gamma)$ for all $a,b \in A$, the infimum being taken over all continuous paths $\gamma$ from $a$ to $b$.

\begin{theorem}
    \label{theorem:main-no-go}
    Let $A$ be a commutative monoid, and let $d_A$ be an extended pseudometric on $A$ that is intrinsic and $p$-subadditive for some $p \in (1,\infty]$.
    Let $G$ be a commutative monoid equipped with an extended pseudometric $d_G$, and with a function $-: G \to G$ such that $d_G(c + (-c), 0) = 0$ for all $c \in G$.
    Assume that addition $+ : G \times G \to G$ is uniformly continuous, where $G \times G$ carries the metric $d_{G \times G}\big((u,c),(v,c')\big) = \max\big(d_G(u,v), d_G(c,c')\big)$.
    If $f : A \to G$ is additive and uniformly continuous, then $d_G(f(a), f(b)) = 0$ for all $a, b \in A$ with $d_A(a,b) < \infty$.
\end{theorem}
\begin{proof}
    Fix $a, b \in A$ with $d_A(a,b) < \infty$, and let $\omega$ be a modulus of uniform continuity for $f$, so that $d_G(f(x), f(y)) \leq \omega(d_A(x,y))$ for all $x,y$, where $\omega(0) = 0$ and $\omega(t) \to 0$ as $t \to 0^+$.
    Applying uniform continuity of addition with equal second arguments yields a function $\eta : [0,\infty] \to [0,\infty]$ with $\eta(0) = 0$ and $\eta(t) \to 0$ as $t \to 0^+$ such that
    \begin{equation}
        \label{dagger}
        d_G(u + c, \, v + c) \leq \eta(d_G(u,v)) \qquad \text{for all } u, v, c \in G.
    \end{equation}

    Since $d_A$ is intrinsic, for every $k \in \Nbb$ there is a continuous path from $a$ to $b$ of length at most $d_A(a,b) + 1/k$; parametrizing it by arc length gives a path $\gamma^k : [0, L_k] \to A$ with $L_k \leq d_A(a,b) + 1/k$ and $d_A(\gamma^k(s), \gamma^k(s')) \leq |s - s'|$.
    Pick an integer $n_k \geq 1$ with $L_k / n_k^{1 - 1/p} < 1/k$, which is possible since $1 - 1/p > 0$ for $p > 1$ (for $p = \infty$, pick instead $n_k$ with $L_k / n_k < 1/k$).
    Setting $x_i^k \coloneqq \gamma^k(i L_k / n_k)$ for $0 \leq i \leq n_k$, we have $x_0^k = a$, $x_{n_k}^k = b$, and $d_A(x_i^k, x_{i+1}^k) \leq L_k / n_k$, so
    \begin{equation}
        \label{eq:chain-bound}
        \big\| \big(d_A(x_i^k, x_{i+1}^k)\big)_{0 \leq i < n_k} \big\|_p \leq \big(n_k \, (L_k / n_k)^p\big)^{1/p} = \frac{L_k}{n_k^{1 - 1/p}} < \frac{1}{k}
    \end{equation}
    (with the evident modification for $p = \infty$).

    By induction on $n$, $p$-subadditivity gives $d_A\big(\textstyle\sum_{i=0}^{n-1} a_i, \sum_{i=0}^{n-1} c_i\big) \leq \big\|(d_A(a_i, c_i))_{i=0}^{n-1}\big\|_p$ for all $a_i, c_i \in A$.
    Applying this with $a_i = x_i^k$ and $c_i = x_{i+1}^k$, and writing $P_k \coloneqq x_1^k + x_2^k + \cdots + x_{n_k - 1}^k$ (the empty sum when $n_k = 1$ being $0$), commutativity of $A$ gives $\sum_{i=0}^{n_k - 1} x_i^k = a + P_k$ and $\sum_{i=0}^{n_k - 1} x_{i+1}^k = b + P_k$, so by \cref{eq:chain-bound},
    \begin{equation}
        \label{eq:A-bound}
        d_A(a + P_k, \, b + P_k) < \frac{1}{k}.
    \end{equation}

    By additivity and uniform continuity of $f$, \cref{eq:A-bound} gives $d_G\big(f(a) + f(P_k), \, f(b) + f(P_k)\big) \leq \omega(1/k)$.
    Applying \cref{dagger} with $c = -f(P_k)$,
    \begin{equation}
        \label{eq:G-bound}
        d_G\big(f(a) + f(P_k) + (-f(P_k)), \; f(b) + f(P_k) + (-f(P_k))\big) \leq \eta(\omega(1/k)).
    \end{equation}
    On the other hand, $d_G\big(0, \, f(P_k) + (-f(P_k))\big) = 0$ by the cancellation property and symmetry, so applying \cref{dagger} with $c = f(a)$, and using that $0$ is the identity of $G$,
    \[
        d_G\big(f(a), \; f(a) + f(P_k) + (-f(P_k))\big) \leq \eta(0) = 0,
    \]
    and likewise with $b$ in place of $a$.
    Combining these two equalities with \cref{eq:G-bound} through the triangle inequality and symmetry of $d_G$ yields $d_G(f(a), f(b)) \leq \eta(\omega(1/k))$.
    Letting $k \to \infty$, we have $\omega(1/k) \to 0$ and hence $\eta(\omega(1/k)) \to 0$, so $d_G(f(a), f(b)) = 0$.
\end{proof}

For a related argument, see \cite[Example~4.7(3)]{bubenik-elchesen}.

\begin{lemma}
    \label{lemma:subcategory-intrinsic}
    Let $\Acal \subseteq \Veck^{\R^n}$ be a subcategory that is closed under shifts (i.e., $M \in \Acal$ implies $M[\epsilon] \in \Acal$ for every $\epsilon \in \Rbb$), finite sums, and pullbacks.
    Then the interleaving distance restricted to $\Acal$ is intrinsic.
\end{lemma}
\begin{proof}
    This fact is implicit in the interpolation results in \cite{bubenik-silva-nanda,chazal-structure-stability}.
    At this level of generality, it follows from \cite[Example~3.2.7,~Proposition~3.2.19,~and~Theorem~4.4.2]{scoccola2020locally}, since being closed under finite sums and pullbacks is equivalent to being closed under finite limits.
\end{proof}

\begin{corollary}
    \label{corollary:no-go}
    Let $\Acal \subseteq \Veck^{\R^n}$ be a subcategory that is closed under shifts, finite sums, and pullbacks.
    Let $G$ be a commutative monoid with a symmetric function $d_G : G \times G \to [0,\infty]$ with $d_G(g,g) = 0$ for all $g \in G$, and such that addition $G \times G \to G$ is uniformly continuous.
    Let $\alpha : \Acal \to G$ be additive and uniformly continuous with respect to $d_I$ and $d_G$.
    Assume that:
    \begin{enumerate}
        \item The function $d_G$ satisfies the triangle inequality;
        \item There exists a function $- : G \to G$ such that $d_G(c + (-c), 0) = 0$ for all $c \in G$.
    \end{enumerate}
    If $M, N \in \Acal$ are such that $d_I(M,N) < \infty$, then $d_G(\alpha(M), \alpha(N)) = 0$.
\end{corollary}
\begin{proof}
    We apply \cref{theorem:main-no-go} with $p = \infty$, $A = \Acal$, and $d_A = d_I$, the interleaving distance.
    Intrinsicness is satisfied thanks to \cref{lemma:subcategory-intrinsic}.
    To see that the interleaving distance restricted to $\Acal$ is $\infty$-subadditive, note that $d_I(M \oplus N, M' \oplus N') \leq \max\{d_I(M,M'), d_I(N,N')\}$, by taking the direct sum of an interleaving between $M$ and $M'$ and an interleaving between $N$ and $N'$.
\end{proof}

%% file: arxiv.bbl
\newcommand{\etalchar}[1]{$^{#1}$}
\providecommand{\bysame}{\leavevmode\hbox to3em{\hrulefill}\thinspace}
\providecommand{\MR}{\relax\ifhmode\unskip\space\fi MR }
\providecommand{\MRhref}[2]{%
  \href{http://www.ams.org/mathscinet-getitem?mr=#1}{#2}
}
\providecommand{\href}[2]{#2}
\begin{thebibliography}{WNvdW{\etalchar{+}}21}

\bibitem[AENY23a]{asashiba-et-al}
Hideto Asashiba, Emerson~G. Escolar, Ken Nakashima, and Michio Yoshiwaki, \emph{Approximation by interval-decomposables and interval resolutions of persistence modules}, J. Pure Appl. Algebra \textbf{227} (2023), no.~10, Paper No. 107397, 20. \MR{4576885}

\bibitem[AENY23b]{AsashibaEscolarNakashimaYoshiwaki2023}
\bysame, \emph{Approximation by interval-decomposables and interval resolutions of persistence modules}, J. Pure Appl. Algebra \textbf{227} (2023), no.~10, Paper No. 107397, 20. \MR{4576885}

\bibitem[AET25]{aoki-et-al}
Toshitaka Aoki, Emerson~G. Escolar, and Shunsuke Tada, \emph{Summand-injectivity of interval covers and monotonicity of interval resolution global dimensions}, J. Appl. Comput. Topol. \textbf{9} (2025), no.~2, Paper No. 13, 34. \MR{4908881}

\bibitem[AP26]{asashiba2026minimal}
Hideto Asashiba and Amit~K Patel, \emph{Minimal projective resolutions, {M}\"{o}bius inversion, and bottleneck stability}, arXiv preprint arXiv:2602.15726 (2026).

\bibitem[ARS97]{auslander-reiten-smalo}
Maurice Auslander, Idun Reiten, and Sverre~O. Smal\o, \emph{Representation theory of {A}rtin algebras}, Cambridge Studies in Advanced Mathematics, vol.~36, Cambridge University Press, Cambridge, 1997, Corrected reprint of the 1995 original. \MR{1476671}

\bibitem[AS94]{auslander-solberg}
M.~Auslander and \O~. Solberg, \emph{Relative homology}, Finite-dimensional algebras and related topics ({O}ttawa, {ON}, 1992), NATO Adv. Sci. Inst. Ser. C: Math. Phys. Sci., vol. 424, Kluwer Acad. Publ., Dordrecht, 1994, pp.~347--359. \MR{1308995}

\bibitem[AT25]{aoki-tada}
Toshitaka Aoki and Shunsuke Tada, \emph{On preservation of relative resolutions for poset representations}, arXiv preprint arXiv:2506.21227 (2025).

\bibitem[Azu50]{Azumaya1950}
Gor\^{o} Azumaya, \emph{Corrections and supplementaries to my paper concerning {K}rull-{R}emak-{S}chmidt's theorem}, Nagoya Math. J. \textbf{1} (1950), 117--124. \MR{37832}

\bibitem[Bac75]{baclawski}
Kenneth Bac{\l}awski, \emph{Whitney numbers of geometric lattices}, Advances in Math. \textbf{16} (1975), 125--138. \MR{387086}

\bibitem[Bar94]{barannikov}
S.~A. Barannikov, \emph{The framed {M}orse complex and its invariants}, Singularities and bifurcations, Adv. Soviet Math., vol.~21, Amer. Math. Soc., Providence, RI, 1994, pp.~93--115.

\bibitem[Bau21]{bauer}
Ulrich Bauer, \emph{Ripser: efficient computation of {V}ietoris-{R}ips persistence barcodes}, J. Appl. Comput. Topol. \textbf{5} (2021), no.~3, 391--423. \MR{4298669}

\bibitem[BB25]{bjerkevik2025computing}
H{\aa}vard~Bakke Bjerkevik and Magnus~Bakke Botnan, \emph{Computing p-presentation distances is hard}, Discrete \& Computational Geometry (2025), 1--41.

\bibitem[BBBF24]{bauer-botnan-fluhr}
Ulrich Bauer, Magnus Bakke~Botnan, and Benedikt Fluhr, \emph{Universal distances for extended persistence}, J. Appl. Comput. Topol. \textbf{8} (2024), no.~3, 475--530. \MR{4799021}

\bibitem[BBH24]{BlanchetteBrustleHanson2024}
Benjamin Blanchette, Thomas Br\"{u}stle, and Eric~J. Hanson, \emph{Homological approximations in persistence theory}, Canad. J. Math. \textbf{76} (2024), no.~1, 66--103. \MR{4687766}

\bibitem[BBH25]{blanchette-brustle-hanson-survey}
\bysame, \emph{Exact structures for persistence modules}, Representations of algebras and related topics, EMS Ser. Congr. Rep., Eur. Math. Soc., Z\"{u}rich, [2025] \copyright 2025, pp.~121--160. \MR{4993173}

\bibitem[BBK20]{bjerkevik-botnan-kerber}
H\aa vard~Bakke Bjerkevik, Magnus~Bakke Botnan, and Michael Kerber, \emph{Computing the interleaving distance is {NP}-hard}, Found. Comput. Math. \textbf{20} (2020), no.~5, 1237--1271. \MR{4156997}

\bibitem[BBK{\etalchar{+}}24]{benjamin-et-al}
Katherine Benjamin, Aneesha Bhandari, Jessica~D Kepple, Rui Qi, Zhouchun Shang, Yanan Xing, Yanru An, Nannan Zhang, Yong Hou, Tanya~L Crockford, et~al., \emph{Multiscale topology classifies cells in subcellular spatial transcriptomics}, Nature \textbf{630} (2024), no.~8018, 943--949.

\bibitem[BBOS20]{bauer-botnan-oppermann-steen}
Ulrich Bauer, Magnus~B. Botnan, Steffen Oppermann, and Johan Steen, \emph{Cotorsion torsion triples and the representation theory of filtered hierarchical clustering}, Adv. Math. \textbf{369} (2020), 107171, 51. \MR{4091895}

\bibitem[BBOS25]{bauer2025additive}
Ulrich Bauer, Magnus~Bakke Botnan, Steffen Oppermann, and Johan Steen, \emph{On the additive image of 0th persistent homology}, arXiv preprint arXiv:2501.09132 (2025).

\bibitem[BBP20]{brodzki2020complexity}
Jacek Brodzki, Matthew Burfitt, and Mariam Pirashvili, \emph{On the complexity of zero-dimensional multiparameter persistence}, arXiv preprint arXiv:2008.11532 (2020).

\bibitem[BBS24]{bindua-brustle-scoccola}
Riju Bindua, Thomas Brüstle, and Luis Scoccola, \emph{Decomposing zero-dimensional persistent homology over rooted tree quivers}, 2024.

\bibitem[BC24]{bubenik-catanzaro}
Peter Bubenik and Michael~J. Catanzaro, \emph{Multiparameter persistent homology via generalized {M}orse theory}, Toric topology and polyhedral products, Fields Inst. Commun., vol.~89, Springer, Cham, 2024, pp.~55--79.

\bibitem[BCB20]{botnan-crawley}
Magnus~Bakke Botnan and William Crawley-Boevey, \emph{Decomposition of persistence modules}, Proc. Amer. Math. Soc. \textbf{148} (2020), no.~11, 4581--4596. \MR{4143378}

\bibitem[BCF{\etalchar{+}}08]{biasotti-size-functions}
S.~Biasotti, A.~Cerri, P.~Frosini, D.~Giorgi, and C.~Landi, \emph{Multidimensional size functions for shape comparison}, J. Math. Imaging Vision \textbf{32} (2008), no.~2, 161--179. \MR{2434687}

\bibitem[BCM25]{balitskiy-coskunuzer-memoli}
Alexey Balitskiy, Baris Coskunuzer, and Facundo M\'{e}moli, \emph{Geometric bounds for persistence}, Trans. Amer. Math. Soc. \textbf{378} (2025), no.~12, 8437--8486. \MR{4982324}

\bibitem[BCY18]{boissonat-chazal-yvinec}
Jean-Daniel Boissonnat, Fr\'{e}d\'{e}ric Chazal, and Mariette Yvinec, \emph{Geometric and topological inference}, Cambridge Texts in Applied Mathematics, Cambridge University Press, Cambridge, 2018. \MR{3837127}

\bibitem[BCZty]{biran-cornea-zhang}
Paul Biran, Octav Cornea, and Jun Zhang, \emph{Triangulation, persistence, and fukaya categories}, (to appear in) Memoirs of the European Mathematical Society.

\bibitem[BDHS25]{blanchette2025stabilization}
Benjamin Blanchette, Justin Desrochers, Eric~J Hanson, and Luis Scoccola, \emph{Stabilization of the spread-global dimension}, arXiv preprint arXiv:2506.01828 (2025).

\bibitem[BDL25]{brustle2025generalized}
Thomas Br{\"u}stle, Justin Desrochers, and Samuel Leblanc, \emph{Generalized rank via minimal subposet}, arXiv preprint arXiv:2510.10837 (2025).

\bibitem[BdSN17]{bubenik-silva-nanda}
Peter Bubenik, Vin de~Silva, and Vidit Nanda, \emph{Higher interpolation and extension for persistence modules}, SIAM J. Appl. Algebra Geom. \textbf{1} (2017), no.~1, 272--284. \MR{3683688}

\bibitem[BdSS15]{bubenik-silva-scott}
Peter Bubenik, Vin de~Silva, and Jonathan Scott, \emph{Metrics for generalized persistence modules}, Found. Comput. Math. \textbf{15} (2015), no.~6, 1501--1531. \MR{3413628}

\bibitem[BE22]{bubenik-elchesen}
Peter Bubenik and Alex Elchesen, \emph{Virtual persistence diagrams, signed measures, {W}asserstein distances, and {B}anach spaces}, J. Appl. Comput. Topol. \textbf{6} (2022), no.~4, 429--474. \MR{4496687}

\bibitem[BEMP13]{bendich2013interlevel}
Paul Bendich, Herbert Edelsbrunner, Dmitriy Morozov, and Amit Patel, \emph{Homology and robustness of level and interlevel sets}, Homology Homotopy Appl. \textbf{15} (2013), no.~1, 51--72. \MR{3031814}

\bibitem[Ber23]{berkouk-instability}
Nicolas Berkouk, \emph{Persistence and the sheaf-function correspondence}, Forum Math. Sigma \textbf{11} (2023), Paper No. e113, 20. \MR{4679252}

\bibitem[BF22]{bauer2022relativeinterlevelsetcohomology}
Ulrich Bauer and Benedikt Fluhr, \emph{Relative interlevel set cohomology categorifies extended persistence diagrams}, 2022.

\bibitem[BG22]{berkouk-ginot}
Nicolas Berkouk and Gr\'{e}gory Ginot, \emph{A derived isometry theorem for sheaves}, Adv. Math. \textbf{394} (2022), Paper No. 108033, 39. \MR{4355732}

\bibitem[BG82]{bongartz-gabriel}
K.~Bongartz and P.~Gabriel, \emph{Covering spaces in representation-theory}, Invent. Math. \textbf{65} (1981/82), no.~3, 331--378. \MR{643558}

\bibitem[BGS96]{beilinson-et-al}
Alexander Beilinson, Victor Ginzburg, and Wolfgang Soergel, \emph{Koszul duality patterns in representation theory}, J. Amer. Math. Soc. \textbf{9} (1996), no.~2, 473--527. \MR{1322847}

\bibitem[BGS26]{bauer2026metrically}
Ulrich Bauer, Cameron Gusel, and Luis Scoccola, \emph{A metrically complete and krull--schmidt space of multiparameter persistence modules}, arXiv preprint arXiv:2603.12049 (2026).

\bibitem[Bje21]{bjerkevik}
H\aa vard~Bakke Bjerkevik, \emph{On the stability of interval decomposable persistence modules}, Discrete Comput. Geom. \textbf{66} (2021), no.~1, 92--121. \MR{4270636}

\bibitem[Bje26]{bjerkevik-2}
\bysame, \emph{Stabilizing {D}ecomposition of {M}ultiparameter {P}ersistence {M}odules}, Found. Comput. Math. \textbf{26} (2026), no.~2, 939--998. \MR{5010368}

\bibitem[BK23]{budney-kaczynski}
Ryan Budney and Tomasz Kaczynski, \emph{Bifiltrations and persistence paths for 2-{M}orse functions}, Algebr. Geom. Topol. \textbf{23} (2023), no.~6, 2895--2924.

\bibitem[BKS17]{bobrowski-kahle-skraba}
Omer Bobrowski, Matthew Kahle, and Primoz Skraba, \emph{Maximally persistent cycles in random geometric complexes}, Ann. Appl. Probab. \textbf{27} (2017), no.~4, 2032--2060. \MR{3693519}

\bibitem[BL15]{bauer-lesnick-induced}
Ulrich Bauer and Michael Lesnick, \emph{Induced matchings and the algebraic stability of persistence barcodes}, J. Comput. Geom. \textbf{6} (2015), no.~2, 162--191. \MR{3333456}

\bibitem[BL18]{botnan-lesnick}
Magnus~Bakke Botnan and Michael Lesnick, \emph{Algebraic stability of zigzag persistence modules}, Algebr. Geom. Topol. \textbf{18} (2018), no.~6, 3133--3204. \MR{3868218}

\bibitem[BL23a]{blumberg-lesnick}
Andrew~J. Blumberg and Michael Lesnick, \emph{Universality of the homotopy interleaving distance}, Trans. Amer. Math. Soc. \textbf{376} (2023), no.~12, 8269--8307. \MR{4669297}

\bibitem[BL23b]{botnan-lesnick-survey}
Magnus~Bakke Botnan and Michael Lesnick, \emph{An introduction to multiparameter persistence}, Representations of algebras and related structures, EMS Ser. Congr. Rep., EMS Press, Berlin, [2023] \copyright 2023, pp.~77--150. \MR{4693638}

\bibitem[BL24]{blumberg-lesnick-2}
Andrew~J. Blumberg and Michael Lesnick, \emph{Stability of 2-parameter persistent homology}, Found. Comput. Math. \textbf{24} (2024), no.~2, 385--427. \MR{4733354}

\bibitem[BLL26]{bauer2026dualities}
Ulrich Bauer, Fabian Lenzen, and Michael Lesnick, \emph{Dualities in multiparameter persistence}, arXiv preprint arXiv:2603.18224 (2026).

\bibitem[BLO22]{botnan-lebovici-oudot}
Magnus~Bakke Botnan, Vadim Lebovici, and Steve Oudot, \emph{On rectangle-decomposable 2-parameter persistence modules}, Discrete Comput. Geom. \textbf{68} (2022), no.~4, 1078--1101. \MR{4517095}

\bibitem[BLO23]{botnan-lebovici-oudot-2}
Magnus~B. Botnan, Vadim Lebovici, and Steve Oudot, \emph{Local characterizations for decomposability of 2-parameter persistence modules}, Algebr. Represent. Theory \textbf{26} (2023), no.~6, 3003--3046. \MR{4681340}

\bibitem[BMG{\etalchar{+}}24]{bauer-et-al}
Ulrich Bauer, Talha~Bin Masood, Barbara Giunti, Guillaume Houry, Michael Kerber, and Abhishek Rathod, \emph{Keeping it sparse: computing persistent homology revisited}, Comput. Geom. Topol. \textbf{3} (2024), no.~1, Art. 6, 26. \MR{4842013}

\bibitem[BMMS24]{bauer-medina-schmahl}
Ulrich Bauer, Anibal~M. Medina-Mardones, and Maximilian Schmahl, \emph{Persistent homology for functionals}, Commun. Contemp. Math. \textbf{26} (2024), no.~10, Paper No. 2350055, 40.

\bibitem[BMW25]{block-manin-weinberger}
Jonathan Block, Fedor Manin, and Shmuel Weinberger, \emph{Persistent homology of function spaces}, 2025, arXiv:2501.08867.

\bibitem[BOO25]{BotnanOppermannOudot2023}
Magnus~Bakke Botnan, Steffen Oppermann, and Steve Oudot, \emph{Signed barcodes for multi-parameter persistence via rank decompositions and rank-exact resolutions}, Found. Comput. Math. \textbf{25} (2025), no.~5, 1815--1874. \MR{4983902}

\bibitem[BOOS24]{botnan-et-al}
Magnus~Bakke Botnan, Steffen Oppermann, Steve Oudot, and Luis Scoccola, \emph{On the bottleneck stability of rank decompositions of multi-parameter persistence modules}, Adv. Math. \textbf{451} (2024), Paper No. 109780, 53. \MR{4761951}

\bibitem[BOST25]{brustle2025counts}
Thomas Br{\"u}stle, Steve Oudot, Luis Scoccola, and Hugh Thomas, \emph{Counts and end-curves in two-parameter persistence}, arXiv preprint arXiv:2505.13412 (2025).

\bibitem[BP20]{boissonat-pritam}
Jean-Daniel Boissonnat and Siddharth Pritam, \emph{Edge collapse and persistence of flag complexes}, 36th {I}nternational {S}ymposium on {C}omputational {G}eometry, LIPIcs. Leibniz Int. Proc. Inform., vol. 164, Schloss Dagstuhl. Leibniz-Zent. Inform., Wadern, 2020, pp.~Art. No. 19, 15. \MR{4117732}

\bibitem[BP21]{berkouk-petit}
Nicolas Berkouk and Fran\c{c}ois Petit, \emph{Ephemeral persistence modules and distance comparison}, Algebr. Geom. Topol. \textbf{21} (2021), no.~1, 247--277. \MR{4224741}

\bibitem[BPP{\etalchar{+}}24a]{buhovsky-et-al}
Lev Buhovsky, Jordan Payette, Iosif Polterovich, Leonid Polterovich, Egor Shelukhin, and Vuka{\v{s}}in Stojisavljevi{\'c}, \emph{Coarse nodal count and topological persistence}, Journal of the European Mathematical Society (2024).

\bibitem[BPP{\etalchar{+}}24b]{bujovsky-2}
Lev Buhovsky, Iosif Polterovich, Leonid Polterovich, Egor Shelukhin, and Vuka{\v{s}}in Stojisavljevi{\'c}, \emph{Persistent transcendental {B}\'ezout theorems}, Forum Math. Sigma \textbf{12} (2024), Paper No. e72, 28.

\bibitem[BR24]{bauer-roll}
Ulrich Bauer and Fabian Roll, \emph{Wrapping cycles in {D}elaunay complexes: bridging persistent homology and discrete {M}orse theory}, 40th {I}nternational {S}ymposium on {C}omputational {G}eometry, LIPIcs. Leibniz Int. Proc. Inform., vol. 293, Schloss Dagstuhl. Leibniz-Zent. Inform., Wadern, 2024, pp.~Art. No. 15, 16. \MR{4757910}

\bibitem[BS14]{bubenik-scott}
Peter Bubenik and Jonathan~A. Scott, \emph{Categorification of persistent homology}, Discrete Comput. Geom. \textbf{51} (2014), no.~3, 600--627. \MR{3201246}

\bibitem[BS24]{bobrowski2024universality}
Omer Bobrowski and Primoz Skraba, \emph{Universality in random persistent homology and scale-invariant functionals}, arXiv preprint arXiv:2406.05553 (2024).

\bibitem[BS25]{bauer-scoccola}
Ulrich Bauer and Luis Scoccola, \emph{Multi-parameter persistence modules are generically indecomposable}, Int. Math. Res. Not. IMRN (2025), no.~5, Paper No. rnaf034, 31. \MR{4870578}

\bibitem[BvR14]{berg-roosmalen}
Carl~Fredrik Berg and Adam-Christiaan van Roosmalen, \emph{Representations of thread quivers}, Proc. Lond. Math. Soc. (3) \textbf{108} (2014), no.~2, 253--290. \MR{3166353}

\bibitem[Car09]{carlsson}
Gunnar Carlsson, \emph{Topology and data}, Bull. Amer. Math. Soc. (N.S.) \textbf{46} (2009), no.~2, 255--308. \MR{2476414}

\bibitem[Car23]{cardona}
Robert Cardona, \emph{Variations and approximations of interleaving distances}, Ph.D. thesis, State University of New York at Albany, 2023.

\bibitem[CB94]{CrawleyBoevey1994}
William Crawley-Boevey, \emph{Locally finitely presented additive categories}, Communications in Algebra \textbf{22} (1994), no.~5, 1641--1674.

\bibitem[CB15]{crawley-boevey}
\bysame, \emph{Decomposition of pointwise finite-dimensional persistence modules}, J. Algebra Appl. \textbf{14} (2015), no.~5, 1550066, 8. \MR{3323327}

\bibitem[CCBdS16]{chazal-crawley-silva}
Fr\'{e}d\'{e}ric Chazal, William Crawley-Boevey, and Vin de~Silva, \emph{The observable structure of persistence modules}, Homology Homotopy Appl. \textbf{18} (2016), no.~2, 247--265. \MR{3575998}

\bibitem[CCLL22]{cardona-et-al}
Robert Cardona, Justin Curry, Tung Lam, and Michael Lesnick, \emph{The universal {$\ell^p$}-metric on merge trees}, 38th {I}nternational {S}ymposium on {C}omputational {G}eometry, LIPIcs. Leibniz Int. Proc. Inform., vol. 224, Schloss Dagstuhl. Leibniz-Zent. Inform., Wadern, 2022, pp.~Art. No. 24, 20. \MR{4470903}

\bibitem[CCSG{\etalchar{+}}09]{chazal-socg}
Fr{\'e}d{\'e}ric Chazal, David Cohen-Steiner, Marc Glisse, Leonidas~J Guibas, and Steve~Y Oudot, \emph{Proximity of persistence modules and their diagrams}, Proceedings of the twenty-fifth annual symposium on Computational geometry, 2009, pp.~237--246.

\bibitem[CDFF10]{cagliari-difabio-ferri}
Francesca Cagliari, Barbara Di~Fabio, and Massimo Ferri, \emph{One-dimensional reduction of multidimensional persistent homology}, Proc. Amer. Math. Soc. \textbf{138} (2010), no.~8, 3003--3017. \MR{2644911}

\bibitem[CDFF{\etalchar{+}}13]{cerri-betti-stable}
Andrea Cerri, Barbara Di~Fabio, Massimo Ferri, Patrizio Frosini, and Claudia Landi, \emph{Betti numbers in multidimensional persistent homology are stable functions}, Math. Methods Appl. Sci. \textbf{36} (2013), no.~12, 1543--1557. \MR{3083259}

\bibitem[CdSGO16]{chazal-structure-stability}
Fr\'{e}d\'{e}ric Chazal, Vin de~Silva, Marc Glisse, and Steve Oudot, \emph{The structure and stability of persistence modules}, SpringerBriefs in Mathematics, Springer, [Cham], 2016. \MR{3524869}

\bibitem[CdSM09]{carlsson-de-silva-morozov}
Gunnar Carlsson, Vin de~Silva, and Dmitriy Morozov, \emph{Zigzag persistent homology and real-valued functions}, Proceedings of the Twenty-Fifth Annual Symposium on Computational Geometry (New York, NY, USA), SCG '09, Association for Computing Machinery, 2009, p.~247–256.

\bibitem[CdSO14]{chazal-desilva-oudot}
Fr\'{e}d\'{e}ric Chazal, Vin de~Silva, and Steve Oudot, \emph{Persistence stability for geometric complexes}, Geom. Dedicata \textbf{173} (2014), 193--214. \MR{3275299}

\bibitem[CEF16]{MR3533890}
Andrea Cerri, Marc Ethier, and Patrizio Frosini, \emph{The coherent matching distance in 2{D} persistent homology}, Computational topology in image context, Lecture Notes in Comput. Sci., vol. 9667, Springer, [Cham], 2016, pp.~216--227. \MR{3533890}

\bibitem[CEF19]{cerri-ethier-frosini}
Andrea Cerri, Marc Ethier, and Patrizio Frosini, \emph{On the geometrical properties of the coherent matching distance in 2{D} persistent homology}, J. Appl. Comput. Topol. \textbf{3} (2019), no.~4, 381--422.

\bibitem[CGL21]{chacholski-giunti-landi}
Wojciech Chach\'{o}lski, Barbara Giunti, and Claudia Landi, \emph{Invariants for tame parametrised chain complexes}, Homology Homotopy Appl. \textbf{23} (2021), no.~2, 183--213. \MR{4281380}

\bibitem[CGOS09]{MR2807544}
Fr\'{e}d\'{e}ric Chazal, Leonidas~J. Guibas, Steve~Y. Oudot, and Primoz Skraba, \emph{Analysis of scalar fields over point cloud data}, Proceedings of the {T}wentieth {A}nnual {ACM}-{SIAM} {S}ymposium on {D}iscrete {A}lgorithms, SIAM, Philadelphia, PA, 2009, pp.~1021--1030. \MR{2807544}

\bibitem[CGOS13]{chazal-JACM}
\bysame, \emph{Persistence-based clustering in {R}iemannian manifolds}, J. ACM \textbf{60} (2013), no.~6, Art. 41, 38. \MR{3144911}

\bibitem[CGR{\etalchar{+}}25]{chacholski-et-al}
Wojciech Chach\'{o}lski, Andrea Guidolin, Isaac Ren, Martina Scolamiero, and Francesca Tombari, \emph{Koszul complexes and relative homological algebra of functors over posets}, Found. Comput. Math. \textbf{25} (2025), no.~4, 1121--1165. \MR{4951524}

\bibitem[CIK97]{chistov-ivanyos-karpinski}
Alexander Chistov, G\'abor Ivanyos, and Marek Karpinski, \emph{Polynomial time algorithms for modules over finite dimensional algebras}, Proceedings of the 1997 International Symposium on Symbolic and Algebraic Computation (ISSAC '97), ACM, New York, 1997, pp.~68--74.

\bibitem[CKM22]{clause2022generalized}
Nathaniel Clause, Woojin Kim, and Facundo Memoli, \emph{The generalized rank invariant: M$\backslash$" obius invertibility, discriminating power, and connection to other invariants}, arXiv preprint arXiv:2207.11591 (2022).

\bibitem[CL11]{cagliari-landi}
Francesca Cagliari and Claudia Landi, \emph{Finiteness of rank invariants of multidimensional persistent homology groups}, Appl. Math. Lett. \textbf{24} (2011), no.~4, 516--518. \MR{2749737}

\bibitem[CL18]{chambers}
Erin~Wolf Chambers and David Letscher, \emph{Persistent homology over directed acyclic graphs}, Research in computational topology, Assoc. Women Math. Ser., vol.~13, Springer, Cham, 2018, pp.~11--32. \MR{3904999}

\bibitem[CM10]{carlsson-memoli}
Gunnar Carlsson and Facundo M\'{e}moli, \emph{Characterization, stability and convergence of hierarchical clustering methods}, J. Mach. Learn. Res. \textbf{11} (2010), 1425--1470. \MR{2645457}

\bibitem[CM21]{chazal2021introduction}
Fr{\'e}d{\'e}ric Chazal and Bertrand Michel, \emph{An introduction to topological data analysis: fundamental and practical aspects for data scientists}, Frontiers in artificial intelligence \textbf{4} (2021), 667963.

\bibitem[CO08]{MR2504289}
Fr\'{e}d\'{e}ric Chazal and Steve~Y. Oudot, \emph{Towards persistence-based reconstruction in {E}uclidean spaces}, Computational geometry ({SCG}'08), ACM, New York, 2008, pp.~232--241. \MR{2504289}

\bibitem[CO20]{cochoy}
J\'{e}r\'{e}my Cochoy and Steve Oudot, \emph{Decomposition of exact pfd persistence bimodules}, Discrete Comput. Geom. \textbf{63} (2020), no.~2, 255--293. \MR{4057439}

\bibitem[CS25]{curto2025topological}
Carina Curto and Nicole Sanderson, \emph{Topological neuroscience: linking circuits to function}, Annual Review of Neuroscience \textbf{48} (2025).

\bibitem[CSEH07]{cohen-steiner-edelsbrunner-harer}
David Cohen-Steiner, Herbert Edelsbrunner, and John Harer, \emph{Stability of persistence diagrams}, Discrete Comput. Geom. \textbf{37} (2007), no.~1, 103--120. \MR{2279866}

\bibitem[CSEH09]{cohen-steiner-edelsbrunner-harer2009extending-persistence}
\bysame, \emph{Extending persistence using {P}oincar\'{e} and {L}efschetz duality}, Found. Comput. Math. \textbf{9} (2009), no.~1, 79--103. \MR{2472288}

\bibitem[Cur14]{curry}
Justin~Michael Curry, \emph{Sheaves, cosheaves and applications}, University of Pennsylvania, 2014.

\bibitem[CYC{\etalchar{+}}25]{calles-et-al}
Juan Calles, Jacky~HT Yip, Gabriella Contardo, Jorge Nore{\~n}a, Adam Rouhiainen, and Gary Shiu, \emph{Cosmology with persistent homology: parameter inference via machine learning}, Journal of Cosmology and Astroparticle Physics \textbf{2025} (2025), no.~09, 064.

\bibitem[CZ09]{calrsson-zomorodian}
Gunnar Carlsson and Afra Zomorodian, \emph{The theory of multidimensional persistence}, Discrete Comput. Geom. \textbf{42} (2009), no.~1, 71--93. \MR{2506738}

\bibitem[CZCG04]{carlsson2004persistence}
Gunnar Carlsson, Afra Zomorodian, Anne Collins, and Leonidas Guibas, \emph{Persistence barcodes for shapes}, Proceedings of the 2004 Eurographics/ACM SIGGRAPH symposium on Geometry processing, 2004, pp.~124--135.

\bibitem[DFL03]{d2003optimal}
Michele D'Amico, Patrizio Frosini, and Claudia Landi, \emph{Optimal matching between reduced size functions}, DISMI, Univ. di Modena e Reggio Emilia, Italy, Technical report \textbf{35} (2003).

\bibitem[DG01]{dung-garcia}
Nguyen~Viet Dung and Jos\'{e}~Luis Garc\'{\i}a, \emph{Additive categories of locally finite representation type}, J. Algebra \textbf{238} (2001), no.~1, 200--238. \MR{1822190}

\bibitem[DJK25]{dey2025}
Tamal~K. Dey, Jan Jendrysiak, and Michael Kerber, \emph{Decomposing multiparameter persistence modules}, 2025.

\bibitem[DKM24]{dey-kim-memoli}
Tamal~K. Dey, Woojin Kim, and Facundo M\'{e}moli, \emph{Computing generalized rank invariant for 2-parameter persistence modules via zigzag persistence and its applications}, Discrete Comput. Geom. \textbf{71} (2024), no.~1, 67--94. \MR{4685709}

\bibitem[DL21]{divol-lacombe}
Vincent Divol and Th\'{e}o Lacombe, \emph{Understanding the topology and the geometry of the space of persistence diagrams via optimal partial transport}, J. Appl. Comput. Topol. \textbf{5} (2021), no.~1, 1--53. \MR{4224153}

\bibitem[DL26]{dey2026limit}
Tamal~K Dey and Michael Lesnick, \emph{Limit computation over posets via minimal initial functors}, arXiv preprint arXiv:2601.00209 (2026).

\bibitem[DRSS99]{draxler-et-al}
Peter Dr\"{a}xler, Idun Reiten, Sverre~O. Smal\o, and \O~yvind Solberg, \emph{Exact categories and vector space categories}, Trans. Amer. Math. Soc. \textbf{351} (1999), no.~2, 647--682, With an appendix by B. Keller. \MR{1608305}

\bibitem[DS78]{drozdowski}
Grzegorz Drozdowski and Daniel Simson, \emph{Remarks on posets of finite representation type}, Preprint, Tor\'{u}n (1978), 1--23.

\bibitem[dSMP16]{desilva-munch-patel}
Vin de~Silva, Elizabeth Munch, and Amit Patel, \emph{Categorified {R}eeb graphs}, Discrete Comput. Geom. \textbf{55} (2016), no.~4, 854--906. \MR{3505333}

\bibitem[dSMS18]{silva-munch-stefanou}
V.~de~Silva, E.~Munch, and A.~Stefanou, \emph{Theory of interleavings on categories with a flow}, Theory Appl. Categ. \textbf{33} (2018), Paper No. 21, 583--607. \MR{3812461}

\bibitem[dSMVJ11]{deSilva-Vejdemo-Johansson-circular}
Vin de~Silva, Dmitriy Morozov, and Mikael Vejdemo-Johansson, \emph{Persistent cohomology and circular coordinates}, Discrete Comput. Geom. \textbf{45} (2011), no.~4, 737--759. \MR{2787567}

\bibitem[DW22]{dey-wang}
Tamal~Krishna Dey and Yusu Wang, \emph{Computational topology for data analysis}, Cambridge University Press, Cambridge, 2022. \MR{4381505}

\bibitem[EH08]{edelsbrunner-harer}
Herbert Edelsbrunner and John Harer, \emph{Persistent homology---a survey}, Surveys on discrete and computational geometry, Contemp. Math., vol. 453, Amer. Math. Soc., Providence, RI, 2008, pp.~257--282. \MR{2405684}

\bibitem[EH10]{edelsbrunner-harer-2}
Herbert Edelsbrunner and John~L. Harer, \emph{Computational topology}, American Mathematical Society, Providence, RI, 2010, An introduction. \MR{2572029}

\bibitem[EK24]{escolar2024barcoding}
Emerson~G Escolar and Woojin Kim, \emph{Barcoding invariants and their equivalent discriminating power}, arXiv preprint arXiv:2412.04995 (2024).

\bibitem[ELZ02]{edelsbrunner-letscher-zomorodian}
Herbert Edelsbrunner, David Letscher, and Afra Zomorodian, \emph{Topological persistence and simplification}, vol.~28, 2002, Discrete and computational geometry and graph drawing (Columbia, SC, 2001), pp.~511--533. \MR{1949898}

\bibitem[EM17]{edelsbrunner2017persistent}
Herbert Edelsbrunner and Dmitriy Morozov, \emph{Persistent homology}, Handbook of Discrete and Computational Geometry, Chapman and Hall/CRC, 2017, pp.~637--661.

\bibitem[Fer24]{fersztand}
Marc Fersztand, \emph{Harder-narasimhan filtrations of persistence modules: metric stability}, 2024.

\bibitem[FJ26]{fersztand-jendrysiak}
Marc Fersztand and Jan Jendrysiak, \emph{Computing the skyscraper invariant}, 2026.

\bibitem[FJNT24]{Fersztand-et-al}
Marc Fersztand, Emile Jacquard, Vidit Nanda, and Ulrike Tillmann, \emph{Harder-{N}arasimhan filtrations of persistence modules}, Trans. London Math. Soc. \textbf{11} (2024), no.~1, Paper No. e70003, 40. \MR{4836795}

\bibitem[FLR{\etalchar{+}}14]{fasy-et-al}
Brittany~Terese Fasy, Fabrizio Lecci, Alessandro Rinaldo, Larry Wasserman, Sivaraman Balakrishnan, and Aarti Singh, \emph{Confidence sets for persistence diagrams}, Ann. Statist. \textbf{42} (2014), no.~6, 2301--2339. \MR{3269981}

\bibitem[Flu24]{fluhr-dissertation}
Benedikt Fluhr, \emph{Cohomological and derived persistence theory of functions}, Ph.D. thesis, Technische Universität München, 2024, p.~260.

\bibitem[FOOO09]{fukaya-1}
Kenji Fukaya, Yong-Geun Oh, Hiroshi Ohta, and Kaoru Ono, \emph{Lagrangian intersection {F}loer theory: anomaly and obstruction. {P}art {I}}, AMS/IP Studies in Advanced Mathematics, vol. 46.1, American Mathematical Society, Providence, RI; Somerville, MA, 2009.

\bibitem[FOOO13]{fukaya-2}
\bysame, \emph{Displacement of polydisks and {L}agrangian {F}loer theory}, J. Symplectic Geom. \textbf{11} (2013), no.~2, 231--268.

\bibitem[Fro90]{frosini-2}
Patrizio Frosini, \emph{A distance for similarity classes of submanifolds of a {E}uclidean space}, Bull. Austral. Math. Soc. \textbf{42} (1990), no.~3, 407--416. \MR{1083277}

\bibitem[Fro92]{frosini1992measuring}
\bysame, \emph{Measuring shapes by size functions}, Intelligent robots and computer vision X: algorithms and techniques, vol. 1607, SPIE, 1992, pp.~122--133.

\bibitem[Fry87]{fry1987defining}
Steve Fry, \emph{Defining and sizing-up mountains}, Summit, Jan.-Feb (1987), 16--21.

\bibitem[Gab72]{gabriel}
Peter Gabriel, \emph{Unzerlegbare {D}arstellungen. {I}}, Manuscripta Math. \textbf{6} (1972), 71--103; correction, ibid. 6 (1972), 309. \MR{332887}

\bibitem[GHP{\etalchar{+}}22]{gardner-grid-cells}
Richard~J Gardner, Erik Hermansen, Marius Pachitariu, Yoram Burak, Nils~A Baas, Benjamin~A Dunn, May-Britt Moser, and Edvard~I Moser, \emph{Toroidal topology of population activity in grid cells}, Nature \textbf{602} (2022), no.~7895, 123--128.

\bibitem[Ghr08]{ghrist}
Robert Ghrist, \emph{Barcodes: the persistent topology of data}, Bull. Amer. Math. Soc. (N.S.) \textbf{45} (2008), no.~1, 61--75. \MR{2358377}

\bibitem[GK15]{gay-kirby}
David~T. Gay and Robion Kirby, \emph{Indefinite {M}orse 2-functions: broken fibrations and generalizations}, Geom. Topol. \textbf{19} (2015), no.~5, 2465--2534.

\bibitem[GLR22]{DONUT}
Barbara Giunti, J{\=a}nis Lazovskis, and Bastian Rieck, \emph{{DONUT}: {D}atabase of {O}riginal \& {N}on-{T}heoretical {U}ses of {T}opology}, 2022, \url{https://donut.topology.rocks}.

\bibitem[GM23]{geist-miller}
Nathan Geist and Ezra Miller, \emph{Global dimension of real-exponent polynomial rings}, Algebra Number Theory \textbf{17} (2023), no.~10, 1779--1788. \MR{4643302}

\bibitem[GNOW24]{MR4776401}
Barbara Giunti, John~S. Nolan, Nina Otter, and Lukas Waas, \emph{Amplitudes in persistence theory}, J. Pure Appl. Algebra \textbf{228} (2024), no.~12, Paper No. 107770, 35. \MR{4776401}

\bibitem[GR92]{GabrielRoiter1992}
P.~Gabriel and A.~V. Ro\u{\i}ter, \emph{Representations of finite-dimensional algebras}, Encyclopaedia Math. Sci., vol.~73, Springer, Berlin, 1992, With a chapter by B. Keller. \MR{1239447}

\bibitem[GS83]{gerstenhaber-schack}
Murray Gerstenhaber and Samuel~D. Schack, \emph{Simplicial cohomology is {H}ochschild cohomology}, J. Pure Appl. Algebra \textbf{30} (1983), no.~2, 143--156. \MR{722369}

\bibitem[Hap88]{happel1988}
Dieter Happel, \emph{Triangulated categories in the representation theory of finite-dimensional algebras}, London Mathematical Society Lecture Note Series, vol. 119, Cambridge University Press, Cambridge, 1988. \MR{935124}

\bibitem[HH24]{harsu2024ephemeral}
Manu Harsu and Eero Hyry, \emph{Ephemeral modules and scott sheaves on a continuous poset}, arXiv preprint arXiv:2411.16235 (2024).

\bibitem[Hil90]{hilbert}
David Hilbert, \emph{Ueber die {T}heorie der algebraischen {F}ormen}, Math. Ann. \textbf{36} (1890), no.~4, 473--534. \MR{1510634}

\bibitem[HL81]{hoeppner-lenzing}
Michael H\"{o}ppner and Helmut Lenzing, \emph{Projective diagrams over partially ordered sets are free}, J. Pure Appl. Algebra \textbf{20} (1981), no.~1, 7--12. \MR{596150}

\bibitem[HR24]{hanson-rock}
Eric~J. Hanson and Job~Daisie Rock, \emph{Decomposition of pointwise finite-dimensional {$\Bbb{S}^1$} persistence modules}, J. Algebra Appl. \textbf{23} (2024), no.~3, Paper No. 2450054, 24. \MR{4688822}

\bibitem[HST18]{hiraoka-et-al}
Yasuaki Hiraoka, Tomoyuki Shirai, and Khanh~Duy Trinh, \emph{Limit theorems for persistence diagrams}, Ann. Appl. Probab. \textbf{28} (2018), no.~5, 2740--2780. \MR{3847972}

\bibitem[IRT23]{igusa-rock-todorov}
Kiyoshi Igusa, Job~D. Rock, and Gordana Todorov, \emph{Continuous quivers of type {$A$} ({I}) foundations}, Rend. Circ. Mat. Palermo (2) \textbf{72} (2023), no.~2, 833--868. \MR{4559075}

\bibitem[IZ90]{igusa-zacharia}
Kiyoshi Igusa and Dan Zacharia, \emph{On the cohomology of incidence algebras of partially ordered sets}, Comm. Algebra \textbf{18} (1990), no.~3, 873--887. \MR{1052771}

\bibitem[Jar20]{jardine2020persistent}
John~F Jardine, \emph{Persistent homotopy theory}, arXiv preprint arXiv:2002.10013 (2020).

\bibitem[JCC{\etalchar{+}}21]{jiang2021topological}
Yi~Jiang, Dong Chen, Xin Chen, Tangyi Li, Guo-Wei Wei, and Feng Pan, \emph{Topological representations of crystalline compounds for the machine-learning prediction of materials properties}, npj computational materials \textbf{7} (2021), no.~1, 28.

\bibitem[KB21]{assif-baryshnikov}
Mishal Assif~P K and Yuliy Baryshnikov, \emph{Biparametric persistence for smooth filtrations}, 2021, arXiv:2107.06800.

\bibitem[Kin08]{kinser-2}
Ryan Kinser, \emph{The rank of a quiver representation}, J. Algebra \textbf{320} (2008), no.~6, 2363--2387. \MR{2437505}

\bibitem[Kin10]{kinser}
\bysame, \emph{Rank functions on rooted tree quivers}, Duke Math. J. \textbf{152} (2010), no.~1, 27--92. \MR{2643056}

\bibitem[KKL25]{kim-kim-lee}
Donghan Kim, Woojin Kim, and Wonjun Lee, \emph{Super-polynomial growth of the generalized persistence diagram}, 41st {I}nternational {S}ymposium on {C}omputational {G}eometry, LIPIcs. Leibniz Int. Proc. Inform., vol. 332, Schloss Dagstuhl. Leibniz-Zent. Inform., Wadern, 2025, pp.~Art. No. 64, 20. \MR{4934417}

\bibitem[KM21]{kim-memoli}
Woojin Kim and Facundo M\'{e}moli, \emph{Generalized persistence diagrams for persistence modules over posets}, J. Appl. Comput. Topol. \textbf{5} (2021), no.~4, 533--581. \MR{4334502}

\bibitem[KM23]{kim-memoli-survey}
\bysame, \emph{Persistence over posets}, Notices Amer. Math. Soc. \textbf{70} (2023), no.~8, 1214--1225. \MR{4626730}

\bibitem[KM24]{kim-moore}
Woojin Kim and Samantha Moore, \emph{Bigraded {B}etti numbers and generalized persistence diagrams}, J. Appl. Comput. Topol. \textbf{8} (2024), no.~3, 727--760. \MR{4799028}

\bibitem[Knu08]{knudson}
Kevin~P. Knudson, \emph{A refinement of multi-dimensional persistence}, Homology Homotopy Appl. \textbf{10} (2008), no.~1, 259--281. \MR{2399474}

\bibitem[KS18]{kashiwara-schapira}
Masaki Kashiwara and Pierre Schapira, \emph{Persistent homology and microlocal sheaf theory}, J. Appl. Comput. Topol. \textbf{2} (2018), no.~1-2, 83--113. \MR{3873181}

\bibitem[KS21]{kashiwara-schapira-2}
\bysame, \emph{Piecewise linear sheaves}, Int. Math. Res. Not. IMRN (2021), no.~15, 11565--11584. \MR{4294126}

\bibitem[Lad08]{ladkani2008homological}
Sefi Ladkani, \emph{Homological properties of finite partially ordered sets}, Hebrew University, 2008.

\bibitem[Lan18]{landi-rank-stability}
Claudia Landi, \emph{The rank invariant stability via interleavings}, Research in computational topology, Assoc. Women Math. Ser., vol.~13, Springer, Cham, 2018, pp.~1--10. \MR{3904998}

\bibitem[LBD{\etalchar{+}}17]{lee-et-al}
Yongjin Lee, Senja~D Barthel, Pawe{\l} D{\l}otko, S~Mohamad Moosavi, Kathryn Hess, and Berend Smit, \emph{Quantifying similarity of pore-geometry in nanoporous materials}, Nature communications \textbf{8} (2017), no.~1, 15396.

\bibitem[Lei26]{leitao2026s}
Ant{\'o}nio Leit{\~a}o, \emph{It's all about covers: Persistent homology of cover refinements}, arXiv preprint arXiv:2602.22784 (2026).

\bibitem[Les02]{leszczynski1}
Zbigniew Leszczy\'{n}ski, \emph{The completely separating incidence algebras of tame representation type}, Colloquium Mathematicae \textbf{94} (2002), no.~2, 243--262 (eng).

\bibitem[Les03]{leszczynski2}
\bysame, \emph{Representation-tame incidence algebras of finite posets}, Colloquium Mathematicae \textbf{96} (2003), no.~2, 293--305 (eng).

\bibitem[Les15]{lesnick-focm}
Michael Lesnick, \emph{The theory of the interleaving distance on multidimensional persistence modules}, Found. Comput. Math. \textbf{15} (2015), no.~3, 613--650. \MR{3348168}

\bibitem[Les23]{lesnick-course}
\bysame, \emph{Notes on multiparameter persistence (for amat 840)}, University at Albany (2023).

\bibitem[Lim24]{lim}
Uzu Lim, \emph{Strange random topology of the circle}, 40th {I}nternational {S}ymposium on {C}omputational {G}eometry, LIPIcs. Leibniz Int. Proc. Inform., vol. 293, Schloss Dagstuhl. Leibniz-Zent. Inform., Wadern, 2024, pp.~Art. No. 70, 17. \MR{4757965}

\bibitem[LM24a]{lesnick2024nerve}
Michael Lesnick and Kenneth McCabe, \emph{Nerve models of subdivision bifiltrations}, arXiv preprint arXiv:2406.07679 (2024).

\bibitem[LM24b]{lesnick2024sparse}
\bysame, \emph{Sparse approximation of the subdivision-rips bifiltration for doubling metrics}, arXiv preprint arXiv:2408.16716 (2024).

\bibitem[LMO24]{lim-memoli-okutan}
Sunhyuk Lim, Facundo M\'emoli, and Osman~Berat Okutan, \emph{Vietoris-{R}ips persistent homology, injective metric spaces, and the filling radius}, Algebr. Geom. Topol. \textbf{24} (2024), no.~2, 1019--1100.

\bibitem[LMS24]{morozov-scoccola}
Yuan Luo, Dmitriy Morozov, and Luis Scoccola, \emph{Computing betti tables and minimal presentations of zero-dimensional persistent homology}, arXiv preprint arXiv:2410.22242 (2024).

\bibitem[Lou75]{loupias}
Mich\`ele Loupias, \emph{Indecomposable representations of finite ordered sets}, Representations of algebras ({P}roc. {I}nternat. {C}onf., {C}arleton {U}niv., {O}ttawa, {O}nt., 1974), Lecture Notes in Math., vol. Vol. 488, Springer, Berlin-New York, 1975, pp.~201--209. \MR{412210}

\bibitem[LS23]{lanari-scoccola}
Edoardo Lanari and Luis Scoccola, \emph{Rectification of interleavings and a persistent {W}hitehead theorem}, Algebr. Geom. Topol. \textbf{23} (2023), no.~2, 803--832. \MR{4587317}

\bibitem[LSB{\etalchar{+}}19]{lawson2019persistent}
Peter Lawson, Andrew~B Sholl, J~Quincy Brown, Brittany~Terese Fasy, and Carola Wenk, \emph{Persistent homology for the quantitative evaluation of architectural features in prostate cancer histology}, Scientific reports \textbf{9} (2019), no.~1, 1139.

\bibitem[LW15]{lesnick-wright-interactive}
Michael Lesnick and Matthew Wright, \emph{Interactive visualization of 2-{D} persistence modules}, arXiv preprint arXiv:1512.00180 (2015).

\bibitem[LW22]{lesnick-wright}
\bysame, \emph{Computing minimal presentations and bigraded {B}etti numbers of 2-parameter persistent homology}, SIAM J. Appl. Algebra Geom. \textbf{6} (2022), no.~2, 267--298. \MR{4429408}

\bibitem[Mai11]{mainini}
E.~Mainini, \emph{A description of transport cost for signed measures}, Zap. Nauchn. Sem. S.-Peterburg. Otdel. Mat. Inst. Steklov. (POMI) \textbf{390} (2011), no.~Teoriya Predstavleni\u{\i}, Dinamicheskie Sistemy, Kombinatornye Metody. XX, 147--181, 308--309. \MR{2870233}

\bibitem[MBW13]{morozov-et-al}
Dmitriy Morozov, Kenes Beketayev, and Gunther Weber, \emph{Interleaving distance between merge trees}, Discrete and Computational Geometry \textbf{49} (2013), no.~22-45, 52.

\bibitem[MH17]{mcinnes2017accelerated}
Leland McInnes and John Healy, \emph{Accelerated hierarchical density based clustering}, 2017 IEEE international conference on data mining workshops (ICDMW), IEEE, 2017, pp.~33--42.

\bibitem[Mil00]{miller-alexander-duality}
Ezra Miller, \emph{The {A}lexander duality functors and local duality with monomial support}, J. Algebra \textbf{231} (2000), no.~1, 180--234. \MR{1779598}

\bibitem[Mil20]{miller2020essential}
\bysame, \emph{Essential graded algebra over polynomial rings with real exponents}, arXiv preprint arXiv:2008.03819 (2020).

\bibitem[Mil23]{miller-stratification}
\bysame, \emph{Stratifications of real vector spaces from constructible sheaves with conical microsupport}, J. Appl. Comput. Topol. \textbf{7} (2023), no.~3, 473--489. \MR{4621004}

\bibitem[Mil25]{miller-siaga}
\bysame, \emph{Homological algebra of modules over posets}, SIAM J. Appl. Algebra Geom. \textbf{9} (2025), no.~3, 483--524. \MR{4941899}

\bibitem[Mit65]{mitchell-2}
Barry Mitchell, \emph{Theory of categories}, Pure and Applied Mathematics, vol. Vol. XVII, Academic Press, New York-London, 1965. \MR{202787}

\bibitem[Mit68]{mitchell-3}
\bysame, \emph{On the dimension of objects and categories. {II}. {F}inite ordered sets}, J. Algebra \textbf{9} (1968), 341--368. \MR{237605}

\bibitem[Mit78]{mitchell1978module}
\bysame, \emph{Some applications of module theory to functor categories}, Bull. Amer. Math. Soc. \textbf{84} (1978), no.~5, 867--885. \MR{499732}

\bibitem[ML98]{maclane}
Saunders Mac~Lane, \emph{Categories for the working mathematician}, second ed., Graduate Texts in Mathematics, vol.~5, Springer-Verlag, New York, 1998. \MR{1712872}

\bibitem[MMv11]{milosavljevic-morozov-skraba}
Nikola Milosavljevi\'{c}, Dmitriy Morozov, and Primo\v{z} \v{S}kraba, \emph{Zigzag persistent homology in matrix multiplication time}, Computational geometry ({SCG}'11), ACM, New York, 2011, pp.~216--225. \MR{2919613}

\bibitem[MMZ25]{medina-zhou}
Anibal~M Medina-Mardones and Ling Zhou, \emph{Persistent cohomology operations and gromov-hausdorff estimates}, arXiv preprint arXiv:2503.17130 (2025).

\bibitem[MN13]{mischaikow-nanda}
Konstantin Mischaikow and Vidit Nanda, \emph{Morse theory for filtrations and efficient computation of persistent homology}, Discrete Comput. Geom. \textbf{50} (2013), no.~2, 330--353. \MR{3090522}

\bibitem[MN26]{mcfaddin-et-al}
Patrick~K. McFaddin and Tom Needham, \emph{Interleaving distances, monoidal actions and 2-categories}, Algebr. Geom. Topol. \textbf{26} (2026), no.~1, 227--281. \MR{5018430}

\bibitem[Mor40]{morse}
Marston Morse, \emph{Rank and span in functional topology}, Ann. of Math. (2) \textbf{41} (1940), 419--454.

\bibitem[MP23]{morozov2023}
Dmitriy Morozov and Amit Patel, \emph{Output-sensitive computation of generalized persistence diagrams for 2-filtrations}, 2023.

\bibitem[MP25]{mccleary-patel}
Alexander McCleary and Amit Patel, \emph{Erratum: ``{E}dit distance and persistence diagrams over lattices''}, SIAM J. Appl. Algebra Geom. \textbf{9} (2025), no.~2, 481--482. \MR{4925170}

\bibitem[MS05]{miller-strumfels}
Ezra Miller and Bernd Sturmfels, \emph{Combinatorial commutative algebra}, Graduate Texts in Mathematics, vol. 227, Springer-Verlag, New York, 2005. \MR{2110098}

\bibitem[MS24]{mallory-sayrafi}
Devlin Mallory and Mahrud Sayrafi, \emph{Computing direct sum decompositions}, 2024.

\bibitem[MS25]{morozov-skraba}
Dmitriy Morozov and Primoz Skraba, \emph{Persistent (co)homology in matrix multiplication time}, 41st {I}nternational {S}ymposium on {C}omputational {G}eometry, LIPIcs. Leibniz Int. Proc. Inform., vol. 332, Schloss Dagstuhl. Leibniz-Zent. Inform., Wadern, 2025, pp.~Art. No. 68, 16. \MR{4934421}

\bibitem[Mun84]{munkres}
James~R. Munkres, \emph{Elements of algebraic topology}, Addison-Wesley Publishing Company, Menlo Park, CA, 1984. \MR{755006}

\bibitem[MW24]{mader-waas}
Tim M\"{a}der and Lukas Waas, \emph{From samples to persistent stratified homotopy types}, J. Appl. Comput. Topol. \textbf{8} (2024), no.~3, 761--838. \MR{4799029}

\bibitem[MZ24]{memoli-zhou}
Facundo M\'emoli and Ling Zhou, \emph{Ephemeral persistence features and the stability of filtered chain complexes}, J. Comput. Geom. \textbf{15} (2024), no.~2, 258--328. \MR{4878036}

\bibitem[NCBJ24]{natarajan2024morse}
Abhinav Natarajan, Thomas Chaplin, Adam Brown, and Maria-Jose Jimenez, \emph{Morse theory for chromatic delaunay triangulations}, arXiv preprint arXiv:2405.19303 (2024).

\bibitem[NHH{\etalchar{+}}15]{nakamura-et-al}
Takenobu Nakamura, Yasuaki Hiraoka, Akihiko Hirata, Emerson~G Escolar, and Yasumasa Nishiura, \emph{Persistent homology and many-body atomic structure for medium-range order in the glass}, Nanotechnology \textbf{26} (2015), no.~30, 304001.

\bibitem[NR72]{nazarova-roiter}
L.~A. Nazarova and A.~V. Ro\u{\i}ter, \emph{Representations of partially ordered sets}, Zap. Nau\v{c}n. Sem. Leningrad. Otdel. Mat. Inst. Steklov. (LOMI) \textbf{28} (1972), 5--31, Investigations on the theory of representations. \MR{340121}

\bibitem[NR73]{nazarova-roiter-2}
Lyudmyla Nazarova and Andrei Ro\u{i}ter, \emph{{Kategornye matrichnye zadachi i problema Brau\'era-Tr\'ella}}, Izdat. ``Naukova Dumka'', Kiev, 1973. \MR{412233}

\bibitem[NSW08]{niyogi-smale-weinberger}
Partha Niyogi, Stephen Smale, and Shmuel Weinberger, \emph{Finding the homology of submanifolds with high confidence from random samples}, Discrete Comput. Geom. \textbf{39} (2008), no.~1-3, 419--441. \MR{2383768}

\bibitem[Oh05]{oh}
Yong-Geun Oh, \emph{Construction of spectral invariants of {H}amiltonian paths on closed symplectic manifolds}, The breadth of symplectic and {P}oisson geometry, Progr. Math., vol. 232, Birkh\"auser Boston, Boston, MA, 2005, pp.~525--570.

\bibitem[OS24]{oudot-scoccola}
Steve Oudot and Luis Scoccola, \emph{On the stability of multigraded {B}etti numbers and {H}ilbert functions}, SIAM J. Appl. Algebra Geom. \textbf{8} (2024), no.~1, 54--88. \MR{4695718}

\bibitem[Oud]{oudot2024differential}
Steve Oudot, \emph{Differential calculus and optimization in persistence module categories}, To appear in the Proceedings of the 9th European Congress of Mathematics.

\bibitem[Oud15]{oudot-book}
Steve~Y. Oudot, \emph{Persistence theory: from quiver representations to data analysis}, Mathematical Surveys and Monographs, vol. 209, American Mathematical Society, Providence, RI, 2015. \MR{3408277}

\bibitem[OW26]{oudot2026function}
Steve Oudot and Lukas Waas, \emph{Function-rips complexes in persistent homotopy theory: Local stability and latschev theorems}, To appear in the Proceedings of the Symposium on Computational Geometry (2026).

\bibitem[Pat18]{patel-mobius}
Amit Patel, \emph{Generalized persistence diagrams}, J. Appl. Comput. Topol. \textbf{1} (2018), no.~3-4, 397--419. \MR{3975559}

\bibitem[PDHH15]{perea-time-series}
Jose~A Perea, Anastasia Deckard, Steve~B Haase, and John Harer, \emph{Sw1pers: Sliding windows and 1-persistence scoring; discovering periodicity in gene expression time series data}, BMC bioinformatics \textbf{16} (2015), no.~1, 257.

\bibitem[Pee11]{peeva}
Irena Peeva, \emph{Graded syzygies}, Algebra and Applications, vol.~14, Springer-Verlag London, Ltd., London, 2011. \MR{2560561}

\bibitem[PET{\etalchar{+}}14]{petri-et-al}
Giovanni Petri, Paul Expert, Federico Turkheimer, Robin Carhart-Harris, David Nutt, Peter~J Hellyer, and Francesco Vaccarino, \emph{Homological scaffolds of brain functional networks}, Journal of The Royal Society Interface \textbf{11} (2014), no.~101.

\bibitem[PRSZ20]{polterovich-rosen-samvelyan-zhang}
Leonid Polterovich, Daniel Rosen, Karina Samvelyan, and Jun Zhang, \emph{Topological persistence in geometry and analysis}, University Lecture Series, vol.~74, American Mathematical Soc., 2020.

\bibitem[PRY24]{paquette2024categories}
Charles Paquette, Job~Daisie Rock, and Emine Y{\i}ld{\i}r{\i}m, \emph{Categories of generalized thread quivers}, arXiv preprint arXiv:2410.14656 (2024).

\bibitem[PS16]{polterovich-shelukhin}
Leonid Polterovich and Egor Shelukhin, \emph{Autonomous {H}amiltonian flows, {H}ofer's geometry and persistence modules}, Selecta Math. (N.S.) \textbf{22} (2016), no.~1, 227--296.

\bibitem[PS26a]{patel2026mobius}
Amit Patel and Primoz Skraba, \emph{M{\"o}bius homology}, Transactions of the American Mathematical Society (2026).

\bibitem[PS26b]{patel-skraba}
\bysame, \emph{M{\"o}bius homology}, Transactions of the American Mathematical Society (2026).

\bibitem[Puu20]{puuska}
Ville Puuska, \emph{Erosion distance for generalized persistence modules}, Homology Homotopy Appl. \textbf{22} (2020), no.~1, 233--254. \MR{4040293}

\bibitem[RCK{\etalchar{+}}17]{rizvi-et-al}
Abbas~H Rizvi, Pablo~G Camara, Elena~K Kandror, Thomas~J Roberts, Ira Schieren, Tom Maniatis, and Raul Rabadan, \emph{Single-cell topological rna-seq analysis reveals insights into cellular differentiation and development}, Nature biotechnology \textbf{35} (2017), no.~6, 551--560.

\bibitem[Rei08]{reineke}
Markus Reineke, \emph{Moduli of representations of quivers}, Trends in representation theory of algebras and related topics, EMS Ser. Congr. Rep., Eur. Math. Soc., Z\"{u}rich, 2008, pp.~589--637. \MR{2484736}

\bibitem[Rin98]{ringel}
Claus~Michael Ringel, \emph{Exceptional modules are tree modules}, Proceedings of the {S}ixth {C}onference of the {I}nternational {L}inear {A}lgebra {S}ociety ({C}hemnitz, 1996), vol. 275/276, 1998, pp.~471--493. \MR{1628405}

\bibitem[Rot64]{rota}
Gian-Carlo Rota, \emph{On the foundations of combinatorial theory. {I}. {T}heory of {M}\"{o}bius functions}, Z. Wahrscheinlichkeitstheorie und Verw. Gebiete \textbf{2} (1964), 340--368 (1964). \MR{174487}

\bibitem[RS24]{rolle-scoccola}
Alexander Rolle and Luis Scoccola, \emph{Stable and consistent density-based clustering via multiparameter persistence}, J. Mach. Learn. Res. \textbf{25} (2024), Paper No. 258, 74. \MR{4810964}

\bibitem[Sch00]{schwarz}
Matthias Schwarz, \emph{On the action spectrum for closed symplectically aspherical manifolds}, Pacific J. Math. \textbf{193} (2000), no.~2, 419--461.

\bibitem[Sch20]{schweinhart-2}
Benjamin Schweinhart, \emph{Fractal dimension and the persistent homology of random geometric complexes}, Adv. Math. \textbf{372} (2020), 107291, 59.

\bibitem[SCL{\etalchar{+}}17]{MR3735858}
Martina Scolamiero, Wojciech Chach\'{o}lski, Anders Lundman, Ryan Ramanujam, and Sebastian \"{O}berg, \emph{Multidimensional persistence and noise}, Found. Comput. Math. \textbf{17} (2017), no.~6, 1367--1406. \MR{3735858}

\bibitem[Sco20]{scoccola2020locally}
Luis Scoccola, \emph{Locally persistent categories and metric properties of interleaving distances}, Ph.D. thesis, The University of Western Ontario (Canada), 2020.

\bibitem[SGB{\etalchar{+}}23]{scoccola-toroidal}
Luis Scoccola, Hitesh Gakhar, Johnathan Bush, Nikolas Schonsheck, Tatum Rask, Ling Zhou, and Jose~A. Perea, \emph{Toroidal coordinates: decorrelating circular coordinates with lattice reduction}, 39th {I}nternational {S}ymposium on {C}omputational {G}eometry, LIPIcs. Leibniz Int. Proc. Inform., vol. 258, Schloss Dagstuhl. Leibniz-Zent. Inform., Wadern, 2023, pp.~Art. No. 57, 20. \MR{4604042}

\bibitem[She13]{sheehy}
Donald~R. Sheehy, \emph{Linear-size approximations to the {V}ietoris-{R}ips filtration}, Discrete \& Computational Geometry \textbf{49} (2013), no.~4, 778--796.

\bibitem[She22]{shelukhin}
Egor Shelukhin, \emph{On the {H}ofer-{Z}ehnder conjecture}, Ann. of Math. (2) \textbf{195} (2022), no.~3, 775--839.

\bibitem[Sim92]{Simson1992}
Daniel Simson, \emph{Linear representations of partially ordered sets and vector space categories}, Algebra, Logic and Applications, vol.~4, Gordon and Breach Science Publishers, Montreux, 1992. \MR{1241646}

\bibitem[SLG23]{sale-et-al}
Nicholas Sale, Biagio Lucini, and Jeffrey Giansiracusa, \emph{Probing center vortices and deconfinement in su (2) lattice gauge theory with persistent homology}, Physical Review D \textbf{107} (2023), no.~3, 034501.

\bibitem[SLH25]{pmlr-v267-scoccola25a}
Luis Scoccola, Uzu Lim, and Heather~A. Harrington, \emph{Cover learning for large-scale topology representation}, Proceedings of the 42nd International Conference on Machine Learning (Aarti Singh, Maryam Fazel, Daniel Hsu, Simon Lacoste-Julien, Felix Berkenkamp, Tegan Maharaj, Kiri Wagstaff, and Jerry Zhu, eds.), Proceedings of Machine Learning Research, vol. 267, PMLR, 13--19 Jul 2025, pp.~53728--53756.

\bibitem[SLM23]{schneider-grid-cells}
Steffen Schneider, Jin~Hwa Lee, and Mackenzie~Weygandt Mathis, \emph{Learnable latent embeddings for joint behavioural and neural analysis}, Nature \textbf{617} (2023), no.~7960, 360--368.

\bibitem[Sma73]{smale}
Steve Smale, \emph{Global analysis and economics. {I}. {P}areto optimum and a generalization of {M}orse theory}, Dynamical systems ({P}roc. {S}ympos., {U}niv. {B}ahia, {S}alvador, 1971), Academic Press, New York-London, 1973, pp.~531--544.

\bibitem[SSL{\etalchar{+}}24]{pmlr-v235-scoccola24a}
Luis Scoccola, Siddharth Setlur, David Loiseaux, Mathieu Carri\`{e}re, and Steve Oudot, \emph{Differentiability and optimization of multiparameter persistent homology}, Proceedings of the 41st International Conference on Machine Learning (Ruslan Salakhutdinov, Zico Kolter, Katherine Heller, Adrian Weller, Nuria Oliver, Jonathan Scarlett, and Felix Berkenkamp, eds.), Proceedings of Machine Learning Research, vol. 235, PMLR, 21--27 Jul 2024, pp.~43986--44011.

\bibitem[SUP25]{spitz2025topological}
Daniel Spitz, Julian~M Urban, and Jan~M Pawlowski, \emph{Topological data analysis of the deconfinement transition in su (3) lattice gauge theory}, Physical Review D \textbf{111} (2025), no.~11, 114519.

\bibitem[TMMH14]{turner-et-al}
Katharine Turner, Yuriy Mileyko, Sayan Mukherjee, and John Harer, \emph{Fr\'{e}chet means for distributions of persistence diagrams}, Discrete Comput. Geom. \textbf{52} (2014), no.~1, 44--70. \MR{3231030}

\bibitem[Ush11]{usher}
Michael Usher, \emph{Boundary depth in {F}loer theory and its applications to {H}amiltonian dynamics and coisotropic submanifolds}, Israel J. Math. \textbf{184} (2011), 1--57.

\bibitem[Ush13]{usher-2}
\bysame, \emph{Hofer's metrics and boundary depth}, Ann. Sci. \'Ec. Norm. Sup\'er. (4) \textbf{46} (2013), no.~1, 57--128.

\bibitem[UZ16]{usher-zhang}
Michael Usher and Jun Zhang, \emph{Persistent homology and {F}loer-{N}ovikov theory}, Geom. Topol. \textbf{20} (2016), no.~6, 3333--3430.

\bibitem[Vit92]{viterbo}
Claude Viterbo, \emph{Symplectic topology as the geometry of generating functions}, Math. Ann. \textbf{292} (1992), no.~4, 685--710.

\bibitem[VUFF93]{frosini-3}
Alessandro Verri, Claudio Uras, Patrizio Frosini, and Massimo Ferri, \emph{On the use of size functions for shape analysis}, Biological cybernetics \textbf{70} (1993), no.~2, 99--107.

\bibitem[Waa24]{waas2024notes}
Lukas Waas, \emph{Notes on abelianity of categories of finitely encoded persistence modules}, arXiv preprint arXiv:2407.08666 (2024).

\bibitem[Wan75]{wan}
Y.~H. Wan, \emph{Morse theory for two functions}, Topology \textbf{14} (1975), no.~3, 217--228.

\bibitem[Wei94]{weibel}
Charles~A. Weibel, \emph{An introduction to homological algebra}, Cambridge Studies in Advanced Mathematics, vol.~38, Cambridge University Press, Cambridge, 1994. \MR{1269324}

\bibitem[Wei11]{weinberger}
Shmuel Weinberger, \emph{What is{$\ldots$}persistent homology?}, Notices Amer. Math. Soc. \textbf{58} (2011), no.~1, 36--39. \MR{2777589}

\bibitem[Wei19]{weinberger-2}
\bysame, \emph{Interpolation, the rudimentary geometry of spaces of {L}ipschitz functions, and geometric complexity}, Found. Comput. Math. \textbf{19} (2019), no.~5, 991--1011. \MR{4017679}

\bibitem[Whi55]{whitney}
Hassler Whitney, \emph{On singularities of mappings of euclidean spaces. {I}. {M}appings of the plane into the plane}, Ann. of Math. (2) \textbf{62} (1955), 374--410.

\bibitem[WNvdW{\etalchar{+}}21]{wilding-et-al}
Georg Wilding, Keimpe Nevenzeel, Rien van~de Weygaert, Gert Vegter, Pratyush Pranav, Bernard~JT Jones, Konstantinos Efstathiou, and Job Feldbrugge, \emph{Persistent homology of the cosmic web--i. hierarchical topology in $\lambda$cdm cosmologies}, Monthly Notices of the Royal Astronomical Society \textbf{507} (2021), no.~2, 2968--2990.

\bibitem[Xi02]{changchang}
Changchang Xi, \emph{Representation dimension and quasi-hereditary algebras}, Adv. Math. \textbf{168} (2002), no.~2, 193--212. \MR{1912131}

\bibitem[ZC05]{zomorodian-carlsson}
Afra Zomorodian and Gunnar Carlsson, \emph{Computing persistent homology}, Discrete Comput. Geom. \textbf{33} (2005), no.~2, 249--274. \MR{2121296}

\bibitem[ZS76]{zavadskij}
AG~Zavadskij and AS~Shkabara, \emph{Commutative quivers and matrix algebras of finite type}, Preprint IM-76-3, Institute of Mathematics AN USSR, Kiev (in Russian) (1976).

\end{thebibliography}
